\documentclass[12pt, a4paper, reqno]{amsart}

\usepackage[centering, left=3cm, right=3cm, top=3cm, bottom=3cm]{geometry}
\allowdisplaybreaks

\usepackage{setspace}
\makeatletter

\def\section@numfont{\mdseries}
\def\subsection@numfont{\bfseries}
\def\subsubsection@numfont{\bfseries}

\renewcommand{\@seccntformat}[1]{%
  \protect\textup{%
    \ifcsname #1@numfont\endcsname
      \csname #1@numfont\endcsname
    \else
      \mdseries
    \fi
    \csname the#1\endcsname
    \protect\@secnumpunct
  }%
}

\renewcommand{\section}{\@startsection{section}{1}%
  \z@{-3.5ex \@plus -1ex \@minus -.2ex}%
  {2.3ex \@plus .2ex}%
  {\normalfont\scshape\centering}}

\renewcommand{\subsection}{\@startsection{subsection}{2}%
  \z@{-3.25ex \@plus -1ex \@minus -.2ex}%
  {1.5ex \@plus .2ex}%
  {\normalfont\bfseries}}

\renewcommand{\subsubsection}{\@startsection{subsubsection}{3}%
  \z@{-3.25ex \@plus -1ex \@minus -.2ex}%
  {1.5ex \@plus .2ex}%
  {\normalfont\bfseries}}

\makeatother

\usepackage{amssymb}
\usepackage{mathrsfs}
\usepackage{upgreek}
\usepackage{extarrows}
\usepackage{empheq}

\usepackage{graphicx,subcaption}
\usepackage{multicol}
\usepackage{array}
\usepackage{cellspace}
\usepackage{nicematrix}
\usepackage{arydshln}
\usepackage{tikz-cd}

\usepackage{enumitem}

\usepackage[style=alphabetic, backend=biber, url=false, doi=false, isbn=false, maxnames=99, maxalphanames=99, sorting=nyt, giveninits=true]{biblatex}
\DeclareDelimFormat[bib,biblist]{nametitledelim}{\addcomma\space}

\DefineBibliographyStrings{english}{byeditor={edited by}, byauthor={by}}
\DeclareFieldFormat[article,inbook,incollection,inproceedings]{title}{#1}
\renewbibmacro*{in:}{\ifentrytype{article}{}{\printtext{\bibstring{in}\space}}}
\DeclareFieldFormat[article]{volume}{\textbf{#1}}
\renewbibmacro*{volume+number+eid}{%
  \printfield{volume}%
  \setunit*{\addcomma\addspace}%
  \printfield{number}%
  \setunit{\addcomma\space}%
  \printfield{eid}}
\DeclareFieldFormat[article]{number}{no\adddot\space #1}
\DeclareFieldFormat{url}{available at\addcolon\space\url{#1}}

\usepackage{xcolor}
\definecolor{tealblue}{rgb}{0.0, 0.45, 0.45}
\definecolor{brickred}{rgb}{0.6, 0.1, 0.1}
\definecolor{darkviolet}{rgb}{0.4, 0.0, 0.4}
\definecolor{purnku}{rgb}{0.494,0.047,0.431}
\definecolor{winered}{rgb}{0.5,0,0}
\definecolor{softblue}{RGB}{25,25,112}

\usepackage[hidelinks]{hyperref}
\hypersetup{colorlinks=true, linkcolor=brickred, citecolor=tealblue, urlcolor=darkviolet}
\usepackage[capitalize,noabbrev,sort&compress]{cleveref}

\numberwithin{equation}{section}

\theoremstyle{plain}
\newtheorem{thm}{Theorem}[section]
\newtheorem{lem}[thm]{Lemma}
\newtheorem{prop}[thm]{Proposition}
\newtheorem{cor}[thm]{Corollary}
\newtheorem{conj}[thm]{Conjecture}
\newtheorem{prob}[thm]{Problem}
\newtheorem{set}[thm]{Setting}

\theoremstyle{definition}
\newtheorem{defn}[thm]{Definition}
\newtheorem{rmk}[thm]{Remark}
\newtheorem{exam}[thm]{Example}

\AddToHook{env/lem/begin}{\crefalias{thm}{lem}}
\AddToHook{env/prop/begin}{\crefalias{thm}{prop}}
\AddToHook{env/cor/begin}{\crefalias{thm}{cor}}
\AddToHook{env/conj/begin}{\crefalias{thm}{conj}}
\AddToHook{env/prob/begin}{\crefalias{thm}{prob}}
\AddToHook{env/set/begin}{\crefalias{thm}{set}}
\AddToHook{env/defn/begin}{\crefalias{thm}{defn}}
\AddToHook{env/rmk/begin}{\crefalias{thm}{rmk}}
\AddToHook{env/exam/begin}{\crefalias{thm}{exam}}

\crefname{thm}{Theorem}{Theorems}
\crefname{conj}{Conjecture}{Conjectures}
\crefname{prop}{Proposition}{Propositions}
\crefname{set}{Setting}{Settings}
\crefname{prob}{Problem}{Problems}
\crefname{cor}{Corollary}{Corollaries}
\crefname{lem}{Lemma}{Lemmas}
\crefname{defn}{Definition}{Definitions}
\crefname{rmk}{Remark}{Remarks}
\crefname{exam}{Example}{Examples}

\newcommand*{\bbr}{\mathbb{R}}

\newcommand*{\bbs}{\mathbb{S}}
\newcommand*{\bbt}{\mathbb{T}}

\newcommand*{\bbz}{\mathbb{Z}}
\newcommand*{\calb}{\mathcal{B}}

\newcommand*{\calf}{\mathcal{F}}

\newcommand*{\calh}{\mathcal{H}}
\newcommand*{\calv}{\mathcal{V}}
\newcommand*{\calu}{\mathcal{U}}
\newcommand*{\calr}{\mathcal{R}}
\newcommand*{\scrl}{\mathcal{L}}
\newcommand*{\calw}{\mathcal{W}}
\newcommand*{\caly}{\mathcal{Y}}

\newcommand*{\calq}{\mathcal{Q}}

\newcommand*{\scrr}{\mathscr{R}}
\newcommand*{\scru}{\mathscr{U}}

\newcommand*{\cm}{\mathrm{c}}

\newcommand*{\gl}{\mathrm{GL}}
\newcommand*{\ortho}{\mathrm{O}}
\newcommand*{\di}{\mathrm{d}}

\newcommand*{\rml}{\mathrm{L}}
\newcommand*{\rmh}{\mathrm{H}}
\newcommand*{\pro}{{\mathrm{pr}}}

\DeclareMathOperator{\ind}{ind}

\DeclareMathOperator{\inter}{Int}

\DeclareMathOperator{\supp}{supp}
\DeclareMathOperator{\diver}{div}
\DeclareMathOperator{\rk}{rank}

\DeclareMathOperator{\ch}{ch}

\DeclareMathOperator{\id}{Id}
\DeclareMathOperator{\dist}{dist}
\DeclareMathOperator{\wid}{width}

\newcommand*{\bv}{{\beta,\varepsilon}}

\newcommand*{\fp}{{F_1^\perp}}
\newcommand*{\ffp}{{F_2^\perp}}
\newcommand*{\fz}{{F^\perp}}
\newcommand*{\f}{{\calf_1^\perp}}
\newcommand*{\ff}{{\calf_2^\perp}}
\newcommand*{\gp}{{Q^\perp}}
\newcommand*{\gb}{{\calq_1^\perp}}
\newcommand*{\ggb}{{\calq_2^\perp}}

\newcommand*{\ffin}{{\calf_{2}^\perp}}

\newcommand*{\hfz}{{\widetilde{\calf}}}
\newcommand*{\hf}{{\widetilde{\calf}_{1}^\perp}}
\newcommand*{\hff}{{\widetilde{\calf}_{2}^\perp}}
\newcommand*{\hgz}{{\widetilde{\calq}}}
\newcommand*{\hg}{{\widetilde{\calq}_{1}^\perp}}

\newcommand*{\rtwo}{\mathrm{\uppercase\expandafter{\romannumeral2}}}

\newcommand*{\win}{{\calw^\infty_{r}}}
\newcommand*{\ws}{{\calw_r}}
\newcommand*{\wsin}{{\calw_r^\infty}}
\newcommand*{\wwin}{{\widetilde{\calw}^\infty_{r}}}

\newcommand*{\ww}{{\widetilde{\calw}_{r}}}

\newcommand*{\pmww}{{\partial_\pm\widetilde{\calw}_{r}}}

\newcommand*{\nbv}{\nabla^{TM,\bv}}
\newcommand*{\nnbv}{\nabla^{\bv}}

\newcommand*{\rbv}{R^{TM,\bv}}
\newcommand*{\cbv}{c_{\bv}}
\newcommand*{\we}{\bar{e}}
\newcommand*{\wnu}{\bar{\nu}}

\newcommand*\abs[1]{\lvert#1\rvert}
\newcommand*\babs[1]{\big\lvert#1\big\rvert}
\newcommand*\bbabs[1]{\Big\lvert#1\Big\rvert}
\newcommand*\bbbabs[1]{\bigg\lvert#1\bigg\rvert}
\newcommand*\inn[1]{\langle#1\rangle}
\newcommand*\binn[1]{\big\langle#1\big\rangle}
\newcommand*\bbinn[1]{\Big\langle#1\Big\rangle}
\newcommand*\bbbinn[1]{\bigg\langle#1\bigg\rangle}
\newcommand*\ket[1]{(#1)}
\newcommand*\bket[1]{\big(#1\big)}
\newcommand*\bbket[1]{\Big(#1\Big)}
\newcommand*\bbbket[1]{\bigg(#1\bigg)}

\newcommand*\bla[1]{\big\|#1\big\|}

\newcommand*\bcu[1]{\big[#1\big]}

\newcommand*\bac[1]{\big\{#1\big\}}
\newcommand*\bbac[1]{\Big\{#1\Big\}}
\newcommand*\bbbac[1]{\bigg\{#1\bigg\}}
\newcommand*\bbbbac[1]{\Bigg\{#1\Bigg\}}
\newcommand*\tinn[1]{\langle#1\rangle}

\newcommand*\ttbinn[1]{\big\langle#1\big\rangle_\bv}

\newcommand*\norm[1]{\lVert#1\rVert}
\newcommand*\en[1]{\mathrm{End}(#1)}

\title[Positive Scalar Curvature on Foliations: Leafwise Band Width]{Positive Scalar Curvature on Foliations:\\ Leafwise Band Width}
\author{Hengyu Chen}
\address{Chern Institute of Mathematics \& LPMC,\newline
\hspace*{18pt}Nankai University, Tianjin 300071, P. R. China}
\email{hichy@mail.nankai.edu.cn}
\date{September 21, 2026}

\begin{document}

\begin{abstract}
  We establish an optimal leafwise width estimate for compact, spin, foliated bands with positive leafwise scalar curvature and infinite vertical $\widehat{A}$-cowaist. This result addresses a special case of the foliated version of Gromov's band width conjecture and generalizes the scalar and mean curvature comparison theorem for spin bands due to Cecchini and Zeidler. Our approach, which also applies to spin foliations, relies on a local boundary value problem for deformed sub-Dirac operators on the Connes fibration. A key technical ingredient is a Lichnerowicz-type formula in the adiabatic limit for almost isometric foliations, which captures the sharp coefficient $\operatorname{rank} F/(\operatorname{rank} F - 1)$ associated with the integrable subbundle $F$.
\end{abstract}

\maketitle
\tableofcontents

\section{Introduction}

In the spirit of Connes \cite{connes}, smooth manifolds admitting foliations with positive leafwise scalar curvature behave, in many aspects, as manifolds themselves admitting metrics of positive scalar curvature (cf.\ \cite[\S3.15]{g23}). 

Inspired by Gromov \cite[p.~259, Footnote~280]{g23}, we combine the methods of \cite{zhang17,cz24} to generalize the band width estimate of Cecchini and Zeidler \cite[Thm.~7.6]{cz24} to foliations. First recall Gromov's band width conjecture.

\begin{conj}[{\cite[\S11.12, Conj.~C]{g18}}]\label{1.1}
    Let $N^{n-1}$ $(n\ge 6)$ be a closed manifold admitting no metric of positive scalar curvature. For $\alpha>0$, if $g$ is a Riemannian metric on the band $N \times [-1, 1]$ with scalar curvature $k_g \ge \alpha^2 n(n-1)$, then
    \begin{equation}\label{dis1}
        \dist_g\!\bket{N\times \{-1\}, N\times \{1\}} \le \frac{2\pi}{\alpha n}.
    \end{equation}
\end{conj}

According to \cite[pp.~653--654]{g18}, the estimate \eqref{dis1} is optimal. Let $g_{0}$ be a flat metric on the torus $\bbt^{n-1}$. For the projection $\pro \colon \bbt^{n-1}\times \ket{-\frac{\pi}{\alpha n},\frac{\pi}{\alpha n}} \to \bbt^{n-1}$, the model example in \cite{g18} is the warped product
\begin{equation}\label{warp}
    \bbt^{n-1}\times \bbket{\!-\frac{\pi}{\alpha n},\frac{\pi}{\alpha n}},\quad g_n = \bbac{\bbket{\cos\dfrac{\alpha nt}{2}}^{\frac{4}{n}}\pro^*g_{0}} \oplus \di t^2,
\end{equation}
which has constant scalar curvature $\alpha^2 n(n-1)$. For related progress via the minimal surface techniques, see, e.g., \cite{g18,zhu21,g23,ra23,gh24,ks25,lz25}. For related progress via the Dirac operator and index theory approaches, see, e.g., \cite{ce20,ze20,ze22,gxy23,wxy24,cz24,liu24,shi25a,song25}.

A band $W$ is defined as a smooth manifold with boundary $\partial W$ such that $\partial W = \partial_- W \sqcup \partial_+ W$, where $\partial_\pm W$ are nonempty unions of connected components \cite[\S2]{g18}. Starting from \cite{cz24}, scalar and mean curvature comparison principles between $(W,g)$ and the model band \eqref{warp} have been established in the following cases:

\begin{enumerate}[label={\rm{(\arabic*)}}]
    \item Let $W^n$ be a connected, compact, spin band with infinite vertical $\widehat{A}$-cowaist\footnote{This is called infinite vertical $\widehat{A}$-area in \cite[Def.~7.3]{cz24}; compare \cref{ahat}\,(3).}.
    \item Suppose $W^n = N^{n-1}\times[-1,1]$ and $\partial_\pm W = N\times\{\pm 1\}$, where $N$ is an orientable compact manifold admitting no metric of positive scalar curvature and\footnote{The case $\dim N = 4$ is excluded due to the exceptional obstruction to positive scalar curvature via Seiberg-Witten invariants: there exists a closed, simply connected, nonspin manifold $N^4$ violating \eqref{dis1}; see \cite[Rem.~1.25]{r07}.} $\dim N \in \{1,\dots,10\}\setminus\{4\}$.
\end{enumerate}
\begin{thm}[{\cite{cz24,ra23,shi25a}}]\label{1.2thm}
    Suppose either $(1)$ or $(2)$ holds. For $\alpha>0$ and $-\frac{\pi}{\alpha n} < t_- < t_+ < \frac{\pi}{\alpha n}$, assume the band $W$ carries a Riemannian metric $g$ whose scalar curvature and mean curvature satisfy
    \begin{equation}\label{eq:scalmean6}
        k_g \ge \alpha^2 n(n-1) \quad\text{and}\quad
        H_g|_{\partial_\pm W} \ge \mp \alpha\tan\!\bbket{\frac{\alpha n t_{\pm}}{2}},
    \end{equation}
    where $H_g$ is defined as in \eqref{h100} with respect to the inward unit normal. Then
    \begin{equation}
        \wid_g(W) \coloneqq \dist_g(\partial_-W,\partial_+W) \le t_+ - t_-.
    \end{equation}
\end{thm}

In particular, \cref{1.1} holds in these cases with strict inequality in \eqref{dis1}. For \cref{1.2thm}\,(1), the odd-dimensional case is proved in \cite[Thm.~7.6]{cz24} using Dirac operators, and the even-dimensional case is proved in \cite[Thm.~1.12]{shi25a} via spectral flow. 
The case $\dim N\le 6$ of \cref{1.2thm}\,(2) is proved in \cite[Cor.~2.25]{ra23} using $\mu$-bubbles, which can be improved to $\dim N\le 10$ by \cite[Thm.~1.2]{cms23} and \cite[Thm.~1.2]{cmsw25}.

Now we turn to foliations. Let $F\subseteq TW$ be a smooth integrable subbundle transverse to $\partial W$, that is, for any $X,Y\in\Gamma(F)$ and any $x \in \partial W$,
\begin{equation}\label{intr}
    [X,Y]\in \Gamma(F)
    \quad\text{and}\quad
    F_x + T_x \partial W = T_x W.
\end{equation}
This is equivalent to a foliation on $W$ with all leaves transverse to $\partial W$; see, e.g., \cite[Def.~1.1.17--1.1.18]{cc} and the Frobenius Theorem \cite[Thm.~1.3.8]{cc}.

Suppose $\rk F \ge 2$. Let $g^F$ be a smooth metric on $F$. The \emph{leafwise scalar curvature} $k^F\in C^\infty(W)$ and the \emph{leafwise boundary mean curvature} $H^F\in C^\infty(\partial W)$ are defined as that of the leaves; see \eqref{k1}--\eqref{h1}.

The \emph{leafwise band width} $\wid_F(W)\in [0,\infty]$ is defined as the infimum of the lengths of piecewise smooth curves $\gamma\colon [0,1]\to W$ such that $\gamma(0)\in \partial_- W$, $\gamma(1)\in \partial_+ W$, and $\gamma([0,1])$ is contained in a leaf; see \cite[p.~259,~Footnote~279]{g23}. If no such leafwise curve exists, we define $\wid_F(W) = \infty$.

Let $E$ be a Hermitian vector bundle over $W$ with Hermitian connection $\nabla^E$. For $x\in W$, we define the norms of the curvature $R^E = (\nabla^E)^2$ along $F$ by
\begin{equation}\label{normcur}
    \babs{R^E_F}_x = \sup_{\substack{X,Y\in F_x\\
    X\perp Y,\;|X|=|Y|=1}}\babs{R^E(X,Y)}
    \quad\text{and}\quad
    \bla{R^E_F}_\infty = \sup_{x\in W}\babs{R^E_F}_x,
\end{equation}
where $|R^E(X,Y)|$ is the operator norm. The following definition generalizes \cite[Def.~7.3]{cz24} and is independent of the choice of the metric $g^F$ on $F$ since $W$ is compact.

\begin{defn}\label{ahat}
    Let $(W,F)$ be a compact, oriented, foliated band.
\begin{enumerate}[label={\rm{(\arabic*)}}]
    \item When $\dim W$ is odd, we say \emph{$(W,F)$ has infinite vertical $\widehat{A}$-cowaist}\footnote{Gromov introduced the concept of $K$-area in \cite[\S4]{g96} and used $\widehat{A}$-area in \cite[\S5$\frac{3}{8}$]{g96}. Recently, Gromov \cite[p.~127]{g23} suggests that the name $K$-cowaist is more appropriate. So we adopt the terminology  ``infinite vertical $\widehat{A}$-cowaist".} if for each $\varpi > 0$, there exists a Hermitian vector bundle $(E,\nabla^E)$ over $W$ such that
    \begin{equation}
        \bla{R^{E}_F}_\infty < \varpi
        \quad \text{and} \quad 
        \big\langle \widehat{A}(T\partial_-W)\ch(E), [\partial_- W] \big\rangle \ne 0.
    \end{equation}
    \item When $\dim W$ is even, we say \emph{$(W,F)$ has infinite vertical $\widehat{A}$-cowaist} if $(W\times\bbs^1,\pro^*_W F)$ does\footnote{For the circle $\bbs^1$ and the projection $\pro_W\colon W\times\bbs^1 \to W$, view $\pro^*_W F$ as a subbundle of $T(W\times\bbs^1)$.}.
    \item We say $W$ \emph{has infinite vertical $\widehat{A}$-cowaist} if the foliated band $(W,TW)$ does.
\end{enumerate}
\end{defn}

If $W$ has infinite vertical $\widehat{A}$-cowaist, then so is $(W,F)$ for any integrable subbundle $F$. In particular, $\bbt^{n-1}\times [-1,1]$ and any band $W$ with $\widehat{A}(\partial_-W)\ne 0$ have infinite vertical $\widehat{A}$-cowaist along any integrable subbundle; see \cref{ex0,ex1}.

Our main theorem generalizes \cite[Thm.~7.6]{cz24} to foliations:

\begin{thm}\label{widmain}
    Let $(W,F)$ be a connected, compact, {spin}, foliated band with infinite vertical $\widehat{A}$-cowaist, where the integrable subbundle $F$ is transverse to $\partial W$ and $q=\rk F \ge 2$. For $\alpha>0$ and $-\frac{\pi}{\alpha q} < t_- < t_+ < \frac{\pi}{\alpha q}$, if $F$ carries a smooth metric $g^F$ whose
    leafwise scalar curvature and leafwise boundary mean curvature satisfy
    \begin{equation}\label{scalmean}
        k^F\ge\alpha^2 q(q-1) \quad\text{and}\quad
        H^{F}|_{\partial_\pm W} \ge \mp \alpha\tan\!\bbket{\frac{\alpha q t_{\pm}}{2}},
    \end{equation}
    then the leafwise band width of $W$ satisfies
    \begin{equation}\label{width}
        \wid_F(W) \le t_+ - t_-.
    \end{equation}
    In particular, if $g^F$ is the restriction of some Riemannian metric $g$ on $W$, then
    \begin{equation}\label{diswid}
        \wid_g(W)\le \wid_F(W) \le t_+ - t_-.
    \end{equation}
\end{thm}

Similar result holds when $F$ is spin and $W$ is nonspin; see \cref{widmain1}. As in \cite[Cor.~7.9]{cz24} and \cite[Cor.~2.9]{ra23}, we have $\mp\tan(\frac{\alpha q t_\pm}{2})\to-\infty$ as $t_\pm\to \pm\frac{\pi}{\alpha q}$, and then \cref{widmain} implies (see \eqref{eq1}--\eqref{eq2} for a detailed proof)
\begin{cor}\label{widmain2}
    Let $(W,F)$ be a connected, compact, spin, foliated band with infinite vertical $\widehat{A}$-cowaist, where $F$ is transverse to $\partial W$ and $q = \rk F \ge 2$. If $F$ carries a smooth metric $g^F$ with $k^F\ge\alpha^2 q(q-1)$ for some $\alpha>0$, then we have
    \begin{enumerate}[label={\rm{(\arabic*)}}]
        \item $\wid_F(W) < \frac{2\pi}{\alpha q}$.
        \item $\inf_{\partial_- W} H^F + \inf_{\partial_+ W} H^F< 0$.
    \end{enumerate}
\end{cor}

This is a special case of the foliated version of Gromov's \cref{1.1}. Meanwhile, \cref{widmain2} implies that: if $(W, F)$ is a connected, compact, spin, foliated band with infinite vertical $\widehat{A}$-cowaist such that $\wid_F(W) = \infty$ and $F$ is transverse to $\partial W$, then $F$ admits no metric of positive leafwise scalar curvature. 

\begin{rmk}
    In a forthcoming paper \cite{cs26}, we will address Gromov's Long Neck Foliated Conjecture \cite[pp.~258--259]{g23} for spin manifolds and spin foliations, using methods developed in the present paper.
\end{rmk}

\subsubsection*{(I) \cref{widmain} is optimal} 

As Gromov points out in \cite[p.~259, Footnote~280]{g23}, a direct combination of the arguments from \cite{zhang17,cz24} will lead to
\begin{equation}\label{compare}
    \wid_F(W) < \frac{2\pi}{\alpha}\sqrt{\frac{n-1}{n q(q-1)}}
\end{equation}
under the conditions of \cref{widmain2}, where $n=\dim W$ and $q=\rk F$.

Our estimates \eqref{width} and \cref{widmain2}\,(1) are optimal\footnote{Let $M$ be any closed manifold. Suppose the band $M \times \bbt^{q-1} \times [t_-,t_+]$ is foliated by the leaves $\bac{\{x\}\times \bbt^{q-1} \times [t_-,t_+]}_{x\in M}$. The metric $g_q$, which is defined by taking $n=q$ in \eqref{warp}, induces a metric along the leaves attaining the equalities in \eqref{scalmean}--\eqref{width}.} and improve \eqref{compare} when $q<n$. This is achieved by developing a Lichnerowicz-type formula with coefficient $\frac{q}{q-1}$ in the adiabatic limit for almost isometric foliations; see \cref{part3,partcur1}. This constitutes a substantial portion of this paper. Even when $(W,F)$ is a Riemannian foliation, obtaining such an optimal estimate on the leafwise band width is nontrivial and requires \cref{part,partcurva}.

As mentioned in \cite[Rem.~4.20]{cz23}, when $F = TW$ and $\dim W$ is even, applying the odd-dimensional result \cite[Thm.~7.6]{cz24} to $W\times\bbs^1$ naively leads to a nonoptimal estimate. The optimal estimate for nonfoliated even-dimensional band is obtained in \cite[Thm.~1.12]{shi25a} via spectral flow. Our \cref{widmain} provides an alternative approach by considering the foliation $(W\times \bbs^1,\pro^*_W TW)$. Note that our \cref{ahat}\,(3) for infinite vertical $\widehat{A}$-cowaist is different from \cite[p.~30]{shi25a} in the even-dimensional case.

\subsubsection*{(II) \cref{widmain} is an essential generalization} 

\cref{widmain} is not a direct consequence of \cref{1.1} or \cref{1.2thm}:

\begin{exam}\label{gene}
    The standard product metric on $N\coloneqq\bbs^3\times\bbt^q$ $(q\ge 3)$ has positive scalar curvature. By \cite{ku94}, there exists a $1$-dimensional smooth foliation on the sphere $\bbs^3$ whose leaves are all diffeomorphic to $\bbr$. Consider an irrational linear foliation on $\bbt^q$ whose leaves are all diffeomorphic to $\bbr^{q-1}$. Then the $q$-dimensional product foliation on $N$ has all its leaves diffeomorphic to $\bbr^q$. Note that the space $\bbr^q$ $(q\ge 3)$ admits complete metrics of uniformly positive scalar curvature by the surgery construction in \cite[pp.~425--427]{gl80b}. 
    
    Take the $(q+1)$-dimensional product foliation on the band $N\times[-1,1]$; see \eqref{product}. Then \cref{1.1} and \cref{1.2thm} do not apply to $N\times [-1,1]$ or its leaves.  

    However, $N$ is compactly $\Lambda^2$-enlargeable along the leaves (see \cref{ex1}) since $\bbt^{q}$ is enlargeable and the foliation on $\bbs^3$ is $1$-dimensional. So the foliated band $N\times [-1,1]$ has infinite vertical $\widehat{A}$-cowaist along the leaves, and \cref{widmain} applies.
\end{exam}

\subsubsection*{(III) Leafwise Band Width Conjecture}

Inspired by Gromov's Long Neck Foliated Conjecture \cite[p.~258]{g23} (cf.\ \cite[Conj.~6.5]{gh24}), we have the following foliated versions of Gromov's \cref{1.1}. 

\begin{conj}[Leafwise Band Width Conjecture]\label{con1}
    Let $(W,F)$ be a connected, compact, foliated band, where $F$ is transverse to $\partial W$ and $q = \rk F\ge 2$ $(q\ne 5)$. Suppose either $(1)$ $\partial_-W$ admits no metric of positive scalar curvature; or $(2)$ $F\cap T\partial_-W$ admits no metric of positive leafwise scalar curvature. For $\alpha>0$, if $F$ is equipped with a metric $g^F$ such that $k^F\ge\alpha^2 q(q-1)$, then $\wid_F(W) < \frac{2\pi}{\alpha q}$.
\end{conj}

Our \cref{widmain2} is a special case of \cref{con1}\,(2), since $F\cap T\partial_-W$ admits no metric of positive leafwise scalar curvature in this case by \cite[Thm.~1.2]{sw25}.

Let $(N,F_0)$ be a closed, foliated band. Consider the product foliated band 
\begin{equation}\label{product}
    N\times[-1,1],\quad F = \pro^*F_0 \oplus T[-1,1],
\end{equation}
where $\pro\colon N\times [-1,1]\to N$ is the projection. Note that for any smooth metric $g^{F_0}$ on $F_0$, the smooth function $k^{F_0}$ is bounded from below on the compact manifold $N$. Then the argument in \cite[pp.~653--654]{g18} shows that \eqref{product} always admits metrics of positive leafwise scalar curvature. So \cref{con1} is not vacuous.

If $(N,F_0)$ carries a smooth metric $g^{F_0}$ of positive leafwise scalar curvature, then the product metric $\pro^*g^{F_0}\oplus \di t^2$ on $N\times\bbr$ has uniformly positive leafwise scalar curvature, which is a complete metric on each leaf. So if \cref{con1}\,(1) is confirmed for \eqref{product}, then the following longstanding open question in foliation theory has a positive answer (at least when the dimension of the foliation $\ne 4$):
\begin{prob}[{cf.\ \cite[Rem.~C14]{z06}}]\label{que}
On a closed, foliated manifold $(M,F)$, does the existence of $g^F$ with $k^F>0$ imply the existence of a metric $g^{TM}$ with $k^{TM}>0$?
\end{prob}

Conversely, if the answer to \cref{que} is affirmative and \cref{con1}\,(2) holds, then \cref{con1}\,(1) follows as a direct corollary. 

\cref{que} admits an easy positive answer when $(M,F)$ is a Riemannian foliation. An approach to this question for codimension one foliations is outlined in \cite[p.~193]{g96}. Zhang \cite[Cor.~0.4]{zhang17} provides a positive answer to this question for closed, simply connected manifolds of dimension $\ge 5$.

\vspace{0.5\baselineskip}
\noindent{\bf Organization of the paper.} Following \cite[Thm.~1.1]{gz93} and \cite[p.144]{g96b}, we first prove in \cref{3.2} that the leafwise distance is the adiabatic limit of the Riemannian distance. In \cref{pre}, we adapt the local boundary value problem from \cite{cz24} to our setting. 
To achieve the optimal leafwise band width estimate, \cref{3} establishes a Lichnerowicz-type formula with coefficient $\frac{q}{q-1}$ for almost isometric foliations in the adiabatic limit.

In \cref{4}, we modify the compact manifolds associated with the Connes fibration used in \cite[\S2.3]{zhang17} and analyze the corresponding local boundary value problem to prove \cref{widmain}. Similar results for spin foliations are obtained in \cref{5} using the sub-Dirac operator introduced in \cite{lz}.

\section{Leafwise Distance as an Adiabatic Limit}\label{3.2}

In differential geometry, the \emph{adiabatic limit} refers to the limiting process of stretching the metric in the horizontal direction of a fibration (or, more generally, in the direction transverse to the leaves of a foliation). This allows one to extract information about the fibers (or leaves). Equivalently, up to a rescaling, this corresponds to shrinking the metric along the the fibers (or leaves).

Let $(M,F)$ be a connected, foliated manifold. Given a metric $g^F$ on $F$ and nonempty compact subsets $K_0,K_1\subset M$, the \emph{leafwise distance} $\dist_F(K_0,K_1)$ is defined by (cf.\ \cite[p.~259,~Footnote~279]{g23})
\begin{equation}\label{w3}
    \dist_F(K_0,K_1) = \inf_{\scrl}\big\{\dist_{\scrl}(\scrl\cap K_0,\scrl\cap K_1)\big\}\in [0,\infty],
\end{equation}
where $\scrl$ ranges over all leaves of $F$ and $\dist_{\scrl}$ is the distance function on $\scrl$. We define $\dist_F(K_0,K_1) = \infty$ if no leaf intersects both $K_0$ and $K_1$. 

Next, suppose $g^F$ is the restriction of a Riemannian metric $g$ on $M$. Then
\begin{equation}\label{noless}
    \dist_{g}(K_0,K_1) \le \dist_F(K_0,K_1).
\end{equation} 
Let $TM = F \oplus F^\perp$ be the orthogonal splitting with $g = g^F \oplus g^{F^\perp}$. For $0<\beta\le 1$, let $\dist_\beta$ denote the distance function induced by the following Riemannian metric
\begin{equation}\label{met23}
    g_\beta = g^F \oplus \frac{g^{F^\perp}}{\beta^2}.
\end{equation}

If $F$ is a smooth distribution that generates $TM$ via iterated Lie brackets, the Riemannian distance converges to the Carnot--Carathéodory distance in the adiabatic limit; see \cite[Thm.~1.1]{gz93} and \cite[\S1.4.D,~p.144]{g96b}. Similarly, we have
\begin{prop}\label{key2}
    Let $(M,g)$ be a connected, complete Riemannian manifold without boundary, and let $F \subseteq TM$ be an integrable subbundle. For any nonempty compact subsets $K_0, K_1 \subset M$, we have
        \begin{equation}\label{d15}
            \lim_{\beta\to 0^+} \dist_\beta(K_0,K_1) = \dist_F(K_0,K_1).
        \end{equation}
\end{prop}
\begin{proof} Without loss of generality, we assume that $K_0,K_1$ are disjoint. For $0<\beta<\beta'\le 1$, we have $g_{\beta'} \le g_\beta$, which implies
    \begin{equation}\label{d4}
        \dist_{\beta'}(K_0,K_1) \le \dist_\beta(K_0,K_1).
    \end{equation}
    So we have $\lim_{\beta\to 0^+} \dist_\beta(K_0,K_1) \in [0,\infty]$. As in \eqref{noless}, for $0<\beta\le 1$, we have
    \begin{equation}\label{d5}
        \dist_\beta(K_0,K_1) \le \dist_F(K_0,K_1).
    \end{equation}

    Suppose to the contrary to \eqref{d15} that there exists $C\in(0,\infty)$ such that
    \begin{equation}\label{d6}
        \lim_{\beta\to 0^+} \dist_\beta(K_0,K_1) \le C < \dist_F(K_0,K_1).
    \end{equation}
    Since $g\le g_\beta$ when $0<\beta\le 1$, the Riemannian metrics $g_\beta$ are complete since $g$ is. Let $\gamma_{\beta}\colon [0,1]\to M$ be a minimizing geodesic in $(M,g_\beta)$ realizing $\dist_\beta(K_0,K_1)$; that is $\gamma_\beta(0)\in K_0$, $\gamma_\beta(1)\in K_1$ and the length of $\gamma_\beta$ equals $\dist_\beta(K_0,K_1)$.
    
    Adapting the argument in \cite[Thm.~1.1]{gz93}, we derive a contradiction. We sketch the argument below. Let $|\cdot|$ and $|\cdot|_\beta$ be the norms induced by $g$ and $g_\beta$, respectively. Since $|\dot{\gamma}_\beta|_\beta$ is constant, it follows from \eqref{d6} that
    \begin{equation}\label{bounded}
        \int_{0}^{1}|\dot{\gamma}_\beta|^2 \le \int_{0}^{1} \bbac{ |p(\dot{\gamma}_\beta)|^2 + \beta^{-2}|p^\perp(\dot{\gamma}_\beta)|^2 } =
         \int_{0}^{1}|\dot{\gamma}_\beta|^2_\beta \le C^2.
    \end{equation}
    By isometrically embedding $M$ into $\mathbb{R}^N$, the Arzelà--Ascoli theorem and the Banach--Alaoglu theorem allow us to extract a subsequence $\{\gamma_{\beta_j}\}$ converging uniformly to an absolutely continuous curve $\zeta\colon [0,1]\to M$, with their derivatives converging weakly in $\rml^2([0,1],\mathbb{R}^N)$. For the orthogonal projection $p^\perp\colon TM\to F^\perp$, it follows from \eqref{bounded} that $\lim_{j\to\infty}\int_0^1|p^\perp(\dot{\gamma}_{\beta_j})|^2 = 0$, which forces $p^\perp(\dot{\zeta}) = 0$ almost everywhere; see \cite[(1.5)]{gz93}. Since $F$ is integrable and $\dot{\zeta}(t)\in F_{\zeta(t)}$ almost everywhere, the curve $\zeta$ is contained in a single leaf. Furthermore, we have $\zeta(0)\in K_0$ and $\zeta(1)\in K_1$. Then the weak lower semicontinuity of $\rml^2$-norm and \eqref{bounded} imply
    \begin{align}
        \begin{split}
            \dist_F(K_0,K_1) &\le \int_{0}^{1}|\dot{\zeta}| \le \bbbac{\int_{0}^{1}|\dot{\zeta}|^2}^{1/2}\le \liminf_{j\to\infty}\bbbac{\int_{0}^{1}|\dot{\gamma}_{\beta_j}|^2}^{1/2}
            \le C,
        \end{split}
    \end{align}
    which contradicts \eqref{d6}.
\end{proof}
An alternative proof can be found in \cref{leafwisedis}.
\begin{cor}\label{key6}
    Let $(M,g)$ be a connected, complete Riemannian manifold with nonempty boundary, and let $F \subseteq TM$ be an integrable subbundle transverse to $\partial M$. Let $K_0, K_1 \subset M$ be compact subsets. Suppose either $K_0 = \partial M$, or $M$ is a band with $K_0 = \partial_- M$ and $K_1 = \partial_+ M$. We define $g_\beta$ as in \eqref{met23}. Then $\dist_\beta(K_0,K_1)$ can be realized by a minimizing geodesic in $(M,g_\beta)$, and the formula \eqref{d15} remains valid.
\end{cor} 
\begin{proof}
    Since $F$ is transverse to $\partial M$, the subbundle $F\cap T\partial M$ is integrable on $\partial M$, which further induces a product foliation on $\partial M \times (0,\infty)$ as in \eqref{product}. By \cite[Ex.~3.3.1]{cc}, we can attach this cylinder to $M$ by identifying $\partial M \times (0,1]$ with a foliated collar neighborhood. Then we obtain a smooth foliated manifold without boundary
    \begin{equation}\label{glue}
        M^\infty = M \cup_{\partial M} \bket{\partial M \times [1,\infty)}.
    \end{equation}

    Let $F$ still denote the integrable subbundle on $M^\infty$. We extend the orthogonal splitting \eqref{met23} to complete metrics of the form $g^{TM^\infty}_\beta = g^F \oplus \beta^{-2}g^{F^\perp}$ on $M^\infty$.
    
    By \cref{key2}, we have
    \begin{equation}\label{k01}
        \lim_{\beta\to 0^+} \dist_{M^\infty,\beta}(K_0,K_1) = \dist_{M^\infty,F}(K_0,K_1),
    \end{equation}
    where $\dist_{M^\infty,\beta}$ is the distance on $(M,g^{TM^\infty}_\beta)$ and $\dist_{M^\infty,F}$ is the leafwise distance on $M^\infty$. Let $\gamma_\beta\colon[0,1]\to M^\infty$ be a minimizing geodesic realizing $\dist_{M^\infty,\beta}(K_0,K_1)$. Since $\partial M$ separates $M$ and $M^\infty\setminus M$, it follows from the assumptions on $K_0,K_1$ that the curve $\gamma_\beta$ is fully contained in $M$, which implies
    \begin{equation}\label{k012}
        \dist_{M,\beta}(K_0,K_1) = \dist_{M^\infty,\beta}(K_0,K_1).
    \end{equation}
    Similarly, we have $\dist_{M,F}(K_0,K_1) = \dist_{M^\infty,F}(K_0,K_1)$. The result follows.
\end{proof}

\section{Deformed Dirac Operators and Boundary Value Problems}\label{pre}
In \cref{2.1}, we review the framework of \cite{cz24} and construct a deformed Dirac operator combining \cite[(2.21)]{zhang17} and \cite[(3-1)]{cz24}. \cref{2.2} addresses the associated boundary value problem with the local boundary condition from \cite[(2-2)]{cz24}. In \cref{2.3}, we examine a special case on odd-dimensional compact bands following \cite[Ex.~2.6, Cor.~3.10]{cz24}.

\vspace{0.5\baselineskip}
\noindent {\bf Notation Conventions.} In this paper, manifolds are smooth and may have boundaries unless otherwise specified. All metrics on vector bundles are smooth. 

For any Euclidean or Hermitian vector bundle $E$ over a manifold $M$, let $\inn{\cdot,\cdot}$ and $|\cdot|$ denote its fiberwise metric and norm. Let $\Gamma(E)$ and $\Gamma_\cm(E)$ denote the spaces of smooth sections and compactly supported smooth sections of $E$, respectively. For a smooth bundle endomorphism $A\in \Gamma(\en{E})$, let $|A|$ be the fiberwise operator norm. We define $\|A\|_\infty = \sup_{M}|A|$.

\subsection{Clifford Modules and Deformed Dirac Operators}\label{2.1}

Let $\ket{M,g^{TM}}$ be a Riemannian manifold with Levi-Civita connection $\nabla^{TM}$. Recall that for any $X,Y,Z\in\Gamma(TM)$,
\begin{align}\label{levi}
    \begin{split}
        2\inn{\nabla^{TM}_XY,Z} &= X\inn{Y,Z} + Y\inn{X,Z}-Z\inn{X,Y}\\ &\qquad+ \inn{[X,Y],Z} - \inn{[X,Z],Y} - \inn{[Y,Z],X}.
    \end{split}
\end{align}
\begin{defn}[{\cite[\S11]{abs64}\,\cite[\S1]{gl83}}]\label{dir}
    A complex (resp.\ real) \emph{Clifford module}\footnote{This is also called a Dirac bundle in \cite[II.~Def.~5.2]{lm89} and \cite[Def.~2.1]{cz24}.} on $\ket{M,g^{TM}}$ is a triple $(S,\nabla,c)$, where $S$ is a Hermitian (resp.\ Euclidean) vector bundle over $M$ equipped with a metric connection $\nabla$ and a smooth bundle map $c\colon TM\to\en{S}$, called \emph{Clifford action}, such that $c(X)$ is skew-adjoint and
    \begin{align}
        c(X)\,c(Y) + c(Y)\,c(X) &= -2\inn{X,Y},\label{a1}\\
        [\nabla_X,c(Y)] &= c\ket{\nabla^{TM}_X Y},\label{a3}
    \end{align}
    where $X,Y\in \Gamma(TM)$. If there is an orthogonal splitting $S = S_+\oplus S_-$ such that $\nabla$ is even and $c(X)$ is odd for any $X\in TM$, then we say $(S,\nabla,c)$ is \emph{$\bbz_2$-graded}. 
\end{defn}

Throughout this paper, our Clifford modules are Hermitian vector bundles unless otherwise specified. Spinor bundles over spin manifolds are Clifford modules. More generally, twisting by a graded vector bundle yields a new Clifford module:
\begin{exam}[{cf.\ \cite[II.~Prop.~5.10]{lm89}}]\label{twi}
Let $(\mu,\nabla^\mu,c^{\,\mu})$ be a Clifford module on $(M, g^{TM})$. Suppose $\xi = \xi^+\oplus\xi^-$ is a $\bbz_2$-graded Hermitian vector bundle over $M$ equipped with an even Hermitian connection $\nabla^\xi$. Then the twisted bundle
\begin{equation}\label{twicli}
    (\mu\otimes \xi,\nabla^{\mu\otimes \xi},c^{\,\mu} \otimes \tau^\xi)
\end{equation}
is also a Clifford module on $(M, g^{TM})$. Here $\tau^\xi = \pm \id|_{\xi^\pm}$ denotes the $\bbz_2$-grading and $\nabla^{\mu\otimes \xi} = \nabla^\mu \otimes \id + \id\otimes \nabla^\xi$ is the induced tensor product connection.
\end{exam}

From now on, suppose $n = \dim M$, and let $(S,\nabla,c)$ be a Clifford module on $\ket{M,g^{TM}}$. The associated \emph{Dirac operator} $D$ is defined by
\begin{equation}\label{dirac}
    D = \sum_{j=1}^n c(e_j)\nabla_{e_j}\colon \Gamma(S)\to \Gamma(S),
\end{equation}
where $\{e_j\}_{j=1}^n$ is any local orthonormal frame of $TM$.

For $u, v \in \Gamma_\cm(S)$, the $\rml^2$-inner product and the $\rml^2$-norm $\|u\|$ are defined by
\begin{equation}
    \int_M \inn{u,v}
    \quad\text{and}\quad 
    \|u\|^2 = \int_M |u|^2,
\end{equation}
with respect to the Riemannian measure induced by $g^{TM}$. The Dirac operator $D$ is a formally self-adjoint first-order elliptic differential operator; see \cite[Prop.~1.11]{gl83}, \cite[II.~ Prop.~5.3]{lm89}.

If $\partial M \ne \varnothing$, let $\nu$ be the \emph{inward} unit normal on $\partial M$. Following \cite[(1.18--19)]{gilkey93}, we have a Clifford module $(S_\partial = S|_{\partial M},\nabla^\partial,c^\partial)$ on $\partial M$ defined for any $X \in T\partial M$ by
\begin{equation}\label{bdclif}
c^\partial(X) = c(X)\,c(\nu)
\quad \text{and} \quad
\nabla^\partial_X = \nabla_X + \frac{1}{2}c^\partial\bket{\nabla^{TM}_X \nu}.
\end{equation}

Let $\{e_j\}_{j=2}^n$ be a local orthonormal frame on $\partial M$. The \emph{boundary Dirac operator}
\begin{equation}\label{bddir}
    D^\partial = \sum_{j=2}^{n} c^\partial(e_j)\,\nabla^\partial_{e_j}\colon \Gamma(S|_{\partial M}) \to \Gamma(S|_{\partial M})
\end{equation}
is intrinsically defined on $\partial M$ by \cite[Lem.~2.1--2.2]{gilkey93}. Then (see also \eqref{bdcnu}) 
\begin{equation}\label{parti}
    D^\partial c(\nu) = - c(\nu)D^\partial.
\end{equation}
Let $H^{TM}\in C^\infty(\partial M)$ denote the mean curvature with respect to $\nu$, given by
\begin{equation}\label{h100}
    H^{TM} = \dfrac{1}{n-1}\sum_{j=2}^{n}\binn{\nabla^{TM}_{e_j} e_j, \nu} = -\dfrac{1}{n-1}\sum_{j=2}^{n} \binn{\nabla^{TM}_{e_j} \nu,e_j}.
\end{equation}
By \cite[(2.27)]{gilkey93}, the restriction of any $u\in \Gamma(S)$ to $\partial M$ satisfies (see also \eqref{bdint})
\begin{equation}\label{ddh}
    D^\partial u = \frac{n-1}{2}H^{TM}u - \bket{c(\nu)D + \nabla_\nu }u.
\end{equation}

The curvature $R = \nabla^2$ of the connection $\nabla$ on the Clifford module $S$ is an $\en{S}$-valued $2$-form on $M$. We have the Bochner--Weitzenböck formula (see \cite[Prop.~2.5]{gl83}, \cite[II.~Thm.~8.2]{lm89})
\begin{equation}\label{licher}
    D^2 = - \Delta + \calr,
\end{equation}
where the Bochner Laplacian $\Delta$ and the curvature term $\calr$ are given by
\begin{equation}\label{cur-1}
    \Delta = \sum_{j=1}^{n}\bac{\nabla_{e_j}\nabla_{e_j} - \nabla_{\nabla^{TM}_{e_j} e_j}} \quad \text{and} \quad \calr = \frac{1}{2}\sum_{i,j = 1}^{n}c(e_i)\,c(e_j)\,R(e_i,e_j).
\end{equation}
By \eqref{ddh}--\eqref{licher} and Green's formulas, for every $u\in \Gamma_\cm(S)$, we have
\begin{align}\label{lic}
    \begin{split}
        \int_M \abs{Du}^2 
        &= \int_M \bbac{\sum_{j=1}^{n}\abs{\nabla_{e_j} u}^2 + \inn{u,\calr u}} + \int_{\partial M} \bbinn{u,\bbket{\frac{n-1}{2}H^{TM} - D^\partial} u};
    \end{split}
\end{align}
refer to \cite{gilkey93}, \cite[Lem.~4.1]{cz24}, see also \eqref{part2}.

On $\bbz_2$-graded Clifford modules, the bracket of two operators is understood as a superbracket. Following \cite{zhang17,cz24}, we will assume
\begin{set}\label{setdd}
    Let $(M,g^{TM})$ be a complete Riemannian manifold with compact or empty boundary, and let $(S,\nabla,c)$ be a $\bbz_2$-graded Clifford module on $\ket{M,g^{TM}}$. Suppose $\vartheta, V\in \Gamma\ket{\en{S}}$ are odd self-adjoint bundle endomorphisms that anticommute with $c(X)$ for all $X\in TM$, and satisfy
    \begin{align}\label{c6}
        \vartheta^2=\id,\quad [\nabla_X,\vartheta] = 0 \quad \text{and} \quad 
        [V,\vartheta] = V\vartheta + \vartheta V = 0.
    \end{align}
    For constants $\delta>0$ and $T\in\bbr$ together with smooth functions $h\colon M\to [\delta,\infty)$ and $\psi\colon M\to\bbr$, we define a \emph{deformed Dirac operator}
    \begin{equation}\label{defd1}
        \calb_{h,\psi,T} = hDh + \psi \vartheta + TV\colon \Gamma(S)\to \Gamma(S).
    \end{equation}
    As in \textup{\cite[(2-1)]{cz24}}, we take a locally constant function $s\colon \partial M \to \{-1,1\}$ and define
    \begin{equation}\label{bc}
        \chi = s\,c(\nu)\,\vartheta \in \Gamma\ket{\en{S|_{\partial M}}}.
    \end{equation}
\end{set}
\begin{rmk}\label{admissible}
    Operators of the form $D+\psi\vartheta$ used in \cite[(3-1)]{cz24} are called \emph{Callias-type operators} \cite{ca78}, and the pair $(S,\vartheta)$ is called a relative Dirac bundle \cite[Def.~2.2]{cz24}. In \eqref{defd1}, the deformation term $TV$ plays the role of $\beta^{-1}f(\frac{\rho}{R})\widehat{c}(\sigma)$ in \cite[(2.21)]{zhang17}, which comes from the analytic localization point of view. The strategy of using two deformations was previously utilized in \cite[(2.18)]{sz22} and \cite[(2.11)]{swz22}. For more discussions on deformed Dirac operators, see \cite{zhang23}. The use of a choice of sign $s\colon \partial M \to \{-1,1\}$ goes back to \cite[Thm.~B]{freed}. 
\end{rmk}
Following \cite[\S2]{cz24}, the bundle endomorphism $\chi$ is an even self-adjoint involution ($\chi^2 =\id$) such that, for all $X\in T\partial M$, we have
\begin{equation}\label{chicommute}
    c(\nu)\chi = -\chi c(\nu) \quad \text{and} \quad c(X)\chi = \chi c(X).
\end{equation}
By \eqref{bddir}--\eqref{parti} and \cref{setdd}, we obtain the identity \cite[(2-11)]{cz24} on $\partial M$:
\begin{equation}\label{bd0}
    D^\partial\chi = -\chi D^\partial.
\end{equation}
Meanwhile, we have $V\chi = \chi V$ under \cref{setdd}. It follows from \eqref{chicommute} that
\begin{equation}\label{bd01}
    \bket{c(\nu)V}\chi = -\chi \bket{c(\nu)V}.
\end{equation}

Define the following subspace of compactly supported smooth sections:
\begin{equation}\label{secbdcon}
    \Gamma_{\vartheta,s}(S) = \{u\in \Gamma_\cm(S):\chi (u|_{\partial M}) = u|_{\partial M}\}.
\end{equation}
Following \cite[(2-12)]{cz24}, for $u\in \Gamma_{\vartheta,s}(S)$, the formulas \eqref{bd0}--\eqref{bd01} imply
\begin{align}\label{bd1}
    \inn{u|_{\partial M},D^\partial (u|_{\partial M})} &= 0,\\
    \inn{u|_{\partial M},c(\nu)V (u|_{\partial M})} &= 0.\label{bd2}
\end{align}

For $u_1,u_2\in \Gamma_\cm(S)$, we have the Green's formula for the inward unit normal $\nu$:
\begin{equation}\label{green}
    \int_M\inn{Du_1,u_2} = \int_M\inn{u_1,Du_2} + \int_{\partial M}\inn{u_1,c(\nu)u_2}.
\end{equation}
    For $\psi\in C^\infty(M)$, we identify the differential $\di \psi$ with the gradient of $\psi$ with respect to $g^{TM}$. Since $D$ and $V$ are odd, we have the superbracket $[D,V] = DV + VD$. 
    
    We have the following analogue of \cite[Prop.~4.2]{cz24}:

\begin{prop}\label{spe2}
    In \cref{setdd}, for $u\in \Gamma_{\vartheta,s}(S)$, we have 
    \begin{align}
            \label{spectral}
            \int_M |\calb_{h,\psi,T} \,u|^2
            =
            &\int_M  \babs{\ket{hDh + \psi\vartheta} u}^2 + \int_M \Big\langle u, \bbket{T^2V^2 + Th^2[D,V]}u \Big\rangle
        \\
        \begin{split}\label{spectral2}
            =&\int_M  h^2|D(hu)|^2 + \int_{\partial M} s\psi |hu|^2 \\
            & \qquad+ \int_M \Big\langle u, \bbket{\psi^2 + h^2c(\di\psi)\vartheta + T^2 V^2 + Th^2[D,V]}u \Big\rangle.
        \end{split}
    \end{align}
\end{prop}
\begin{proof}By \eqref{defd1}, \eqref{green}, and the assumptions in \cref{setdd}, we have
    \begin{align}\label{span}
        \begin{split}
            \int_M |\calb_{h,\psi,T} \,u|^2
            =& \int_M \!\babs{\ket{hDh + \psi\vartheta} u}^2 + \int_M\!\bbinn{u, \bbket{T^2V^2 + [hDh +\psi\vartheta, TV]}u}\\
            &\qquad + \int_{\partial M}\binn{hu,c(\nu)TV(hu)}.
        \end{split}
    \end{align}
    By \eqref{c6}, \eqref{bd2}, and $[hDh,TV] = Th^2[D,V]$, we obtain \eqref{spectral}. Similarly, we have
    \begin{align}\label{span2}
        \begin{split}
            \int_M  \babs{\ket{hDh + \psi\vartheta} u}^2
            =& \int_M  \abs{hDhu}^2 + \int_M\bbinn{u, \bbket{\psi^2\vartheta^2 + [hDh,\psi\vartheta]}u}\\
            &\qquad + \int_{\partial M}\binn{hu,c(\nu)\psi\vartheta(hu)}.
        \end{split}
    \end{align}
    By \eqref{span}--\eqref{span2}, $[hDh,\psi\vartheta] = h^2c(\di\psi)\vartheta$, and $\chi (u|_{\partial M}) =u|_{\partial M}$, we obtain \eqref{spectral2}. 
\end{proof}

\subsection{Boundary Value Problems and Index Theory}\label{2.2}

Assume \cref{setdd}. Let $\rml^2(M,S)$ and $\rml^2_{\mathrm{loc}}(M,S)$ denote the spaces of square-integrable and locally square-integrable sections, respectively. The Sobolev space $\rmh^1_{\mathrm{loc}}(M,S)$ consists of all sections $u\in \rml^2_{\mathrm{loc}}(M,S)$ such that $\nabla u \in \rml^2_{\mathrm{loc}}(M,T^*M\otimes S)$ in the distributional sense. Since $\partial M$ is compact, the restriction $\Gamma(S)\to \Gamma(S|_{\partial M}), u\mapsto u|_{\partial M}$ extends by the trace theorem to a continuous linear operator
\begin{equation}
    \scrr\colon \rmh^1_{\mathrm{loc}}(M,S) \to \rmh^{1/2}(\partial M ,S|_{\partial M}).
\end{equation}

By \cite[p.~1183]{cz24}, the Dirac operator $D$ is essentially self-adjoint on the domain $\Gamma_{\vartheta,s}(S)$ (the crucial ingredient is that $c(\nu)$ anticommutes with $\chi$) and its closure has domain given by the Sobolev space
\begin{align}\label{bdcon2}
    \begin{split}
        \rmh^1_{\vartheta,s}(M,S) &= \big\{u\in \rmh^1_{\mathrm{loc}}(M,S) \cap \rml^2(M,S):\big.\\
        &\qquad\qquad\qquad\big.D u\in \rml^2(M,S), \chi \scrr u = \scrr u\big\};
    \end{split}
\end{align}
refer to \cite[\S7.2]{bb12} and \cite[Lem.~7.3]{bb12}. Endow $\rmh^1_{\vartheta,s}(M,S)$ with the graph norm
\begin{equation}\label{graphnorm}
    \norm{u}_{\rmh^1_{\vartheta,s}(M,\,S)} = \bac{\norm{u}^2 + \norm{Du}^2}^{1/2}.
\end{equation}
In particular, the subspace $\Gamma_{\vartheta,s}(S)$ is dense in $\rmh^1_{\vartheta,s}(M,S)$; see also \cite[Thm.~4.10]{bb16}.

Further assume that $h,\psi$ are bounded and $c(\di h), V$ are uniformly bounded under \cref{setdd}. By \eqref{defd1}, we have
\begin{equation}\label{defd2}
    \calb_{h,\psi,T} 
    = h^2D + hc(\di h) + \psi\vartheta + TV.
\end{equation}
Then the closure of the operator $\calb_{h,\psi,T}$ on $\Gamma_{\vartheta,s}(S)$ also has $\rmh^1_{\vartheta,s}(M,S)$ as its domain. Meanwhile, the graph norm on $\rmh^1_{\vartheta,s}(M,S)$ induced by $\calb_{h,\psi,T}$ is equivalent to \eqref{graphnorm}. 

Let  $\calb_{h,\psi,T,s}$ denote the operator $\calb_{h,\psi,T}$ on the domain $\rmh^1_{\vartheta,s}(M,S)$. Since $\vartheta,V$ are self-adjoint, the operator $\calb_{h,\psi,T,s}$ is also self-adjoint by \eqref{defd1}.

The boundary Dirac operator $D^\partial$ from \eqref{bddir} is an adapted boundary operator\footnote{Let $\partial M\times [0,\epsilon)$ be a geodesic collar neighborhood of $\partial M$. Let   $t\in [0,\epsilon)$ denote the normal coordinate so that $\frac{\partial}{\partial t}$ is the inward unit normal vector field. After identifying
$S|_{\partial M\times\{t\}}$ with $S|_{\partial M}$ by parallel transport along the normal geodesics, $B_{h,\psi,T}=hDh + \psi \vartheta + TV$ is of the form \begin{equation}
h^2 c(\tfrac{\partial}{\partial t}) \,\bket{\tfrac{\partial}{\partial t} + D^\partial + A_t}\notag
\end{equation}
on $\partial M\times [0,\epsilon)$, where $\{A_t\}_{t\in [0,\epsilon)}$ is a smooth family of differential operators on $\Gamma(S|_{\partial M})$ of order at most one, with $A_0$ of order zero; see \cite[Lem.~1.4]{bb12}.} for $\calb_{h,\psi,T}$ in the sense of \cite[p.~4]{bb12}. Then \eqref{bd0} and \cite[Cor.~7.23]{bb12} imply that $\chi\scrr u = \scrr u$ is a local elliptic boundary condition for $\calb_{h,\psi,T}$. This is a chiral boundary condition; see \cite[Ex.~7.26]{bb12} and \cite[Ex.~4.20]{bb16} for further discussions.

\begin{prop}[cf.\ {\cite[Thm.~3.4]{cz24}}]\label{fredholm}
    Under \cref{setdd}, further assume that $h,\psi$ are bounded and $c(\di h), V$ are uniformly bounded. 
    If there exists a compact subset $L_0\subseteq M$ and a constant $C_0>0$ such that
    \begin{equation}\label{positive}
        \psi^2 + h^2c(\di\psi)\vartheta + T^2 V^2 + Th^2[D,V] \ge C_0 \quad \text{on } M\setminus L_0
    \end{equation}
    as an inequality of self-adjoint bundle endomorphisms, then the operator
    \begin{equation}\label{bhpsis}
        \calb_{h,\psi,T,s}\colon \rmh^1_{\vartheta,s}(M,S) \to \rml^2(M,S)
    \end{equation}
    is Fredholm and self-adjoint when viewed as an unbounded operator on $\rml^2(M,S)$ with domain $\rmh^1_{\vartheta,s}(M,S)$.
\end{prop}
\begin{proof}
    For $u\in\Gamma_\cm(S)$ supported in $M\setminus (L_0 \cup \partial M)$, by \eqref{spectral2} and \eqref{positive}, we have
    \begin{equation}
            \int_{M}\abs{\calb_{h,\psi,T}\,u}^2
            \ge
            C_0 \int_{M}\abs{u}^2.
    \end{equation}
    Since $L_0\cup \partial M$ is compact, the operator $\calb_{h,\psi,T}$ is coercive at infinity in the sense of \cite[Def.~8.2]{bb12}. Then $\calb_{h,\psi,T,s}$ is Fredholm by \cite[Cor.~8.6]{bb12}. 
\end{proof}

Recall that $S = S^+\oplus S^{-}$ is $\bbz_2$-graded and $\chi$ is even. Let $\rmh^1_{\vartheta,s}(M,S^\pm)$ be the Sobolev spaces defined as in \eqref{bdcon2}. Since $\calb_{h,\psi,T,s}$ is odd and self-adjoint, it restricts to the following operators:
\begin{equation}\label{bpm}
    \calb^\pm_{h,\psi,T,s}\colon \rmh^1_{\vartheta,s}(M,S^\pm) \to \rml^2(M,S^\mp),
\end{equation}
which are adjoint to each other and satisfy
\begin{equation}\label{bpm0}
    \calb_{h,\psi,T,s} =
    \begin{pmatrix}
     0 & \calb^-_{h,\psi,T,s}\\
     \calb^+_{h,\psi,T,s} & 0
    \end{pmatrix}.
\end{equation}
Under \cref{fredholm}, we define the \emph{index} of the operator $\calb_{h,\psi,T,s}$ by
\begin{equation}\label{index}
    \ind \calb_{h,\psi,T,s} = \ind \calb^+_{h,\psi,T,s} = \dim \ker \calb^+_{h,\psi,T,s} - \dim \ker \calb^-_{h,\psi,T,s}\in\bbz.
\end{equation}

\subsection{An Index Formula for Odd-Dimensional Compact Bands}\label{2.3}

Following \cite[Ex.~2.6]{cz24}, we consider the case where $W$ is an odd-dimensional connected compact band with $\partial W = \partial_-W \sqcup \partial_+W$. Recall that $\partial_\pm W$ are nonempty unions of connected components. Let $W^\infty$ be the noncompact manifold without boundary obtained by attaching two cylindrical ends to $W$, that is
\begin{equation}
    W^\infty = W^-\cup_{\partial_- W} W \cup_{\partial_+ W} W^+,
\end{equation}
where $W^\pm = \partial_\pm W \times [1,\infty)$. Let $g^{TW^\infty}$ be a complete Riemannian metric on $W^\infty$.

Let $(\mu,\nabla^\mu,c^{\,\mu})$ be a Clifford module on $(W^\infty, g^{TW^\infty})$. Suppose $\xi = \xi^+\oplus\xi^-$ is a $\bbz_2$-graded Hermitian vector bundle over $W^\infty$ equipped with an even Hermitian connection $\nabla^\xi$ and an odd self-adjoint bundle endomorphism $V^\xi\in\Gamma\ket{\en{\xi}}$. For the $\bbz_2$-grading $\tau^\xi$ on $\xi$, we obtain a Clifford module by \cref{twi}:
\begin{equation}\label{twicli1}
    (\mu\otimes \xi,\nabla^{\mu\otimes \xi},c^{\,\mu} \otimes \tau^\xi).
\end{equation}
    By \cite[Ex.~2.6]{cz24}, the Clifford module $\mu\otimes\xi$ induces a $\bbz_2$-graded Clifford module
\begin{equation}\label{dirdou1}
    S = \ket{\mu \otimes \xi} \oplus \ket{\mu \otimes \xi},
\end{equation}
where we regard the first summand as $S^+$ and the second summand as $S^-$. Its Hermitian connection and Clifford action are defined by
    \begin{equation}\label{cldouble1}
        \nabla = 
        \begin{pmatrix}
            \nabla^{\mu \otimes \xi} & 0 \\
            0 & \nabla^{\mu \otimes \xi}
        \end{pmatrix}
        \quad
        \text{and}
        \quad
        c = 
        \begin{pmatrix}
            0 & c^{\,\mu} \otimes \tau^\xi \\
            c^{\,\mu} \otimes \tau^\xi & 0
        \end{pmatrix}.
    \end{equation}
    Then the associated Dirac operator is given by
    \begin{equation}\label{dir2}
        D = \begin{pmatrix}
        0 & D^{\mu \otimes \xi} \\
        D^{\mu \otimes \xi}  & 0
    \end{pmatrix}.\end{equation}
    Meanwhile, we have two odd self-adjoint bundle endomorphisms
    \begin{equation}\label{oddend}
        V = \begin{pmatrix}
                0 & \id_\mu\otimes V^{\xi} \\
                \id_\mu\otimes V^{\xi} & 0
            \end{pmatrix}
        \quad\text{and}\quad 
        \vartheta = \begin{pmatrix}
        0 & - \sqrt{-1} \\
        \sqrt{-1} & 0
        \end{pmatrix},
    \end{equation}
    which satisfy the requirements in \cref{setdd}.
    
    Following \cite[Cor.~3.10]{cz24}, define a choice of signs $s\colon\partial W\to \{-1,1\}$ by setting
    \begin{equation}
        s\equiv \pm 1\qquad\text{on } \partial_\pm W.
    \end{equation}
    Let $\delta>0$ and $T\in\bbr$ be two constants. Let $h\colon W\to [\delta,\infty)$ and $\psi\colon W\to \bbr$ be smooth functions. Since $W$ is compact, the condition \eqref{positive} automatically holds. So we obtain a self-adjoint Fredholm operator
    \begin{equation}\label{bdproblem}
        \calb^W_{h,\psi,T,s} = hDh + \psi \vartheta + TV\colon \rmh^1_{\vartheta,s}(W,S) \to \rml^2(W,S)
    \end{equation}
    by \cref{fredholm}. The superscript $W$ indicates that $\calb^W_{h,\psi,T,s}$ is defined on $W$.

    By \eqref{bdclif}, the Clifford module $(\mu,\nabla^\mu,c^{\,\mu})$ on $W^\infty$ induces a Clifford module $\mu_\partial = \mu|_{\partial W}$ on $\partial W$. Consider the $\bbz_2$-grading
    \begin{equation}\label{z2grading}
        \mu_\partial = \mu^+_\partial \oplus \mu^-_\partial
    \end{equation}
    determined by the involution $\tau^{\mu_\partial} \coloneqq c^{\,\mu}(\nu)/\sqrt{-1}$, where $\nu$ is the inward unit normal.

    Consider the Clifford module\footnote{Here ``$\widehat{\otimes}$'' is the notation for the $\bbz_2$-graded tensor product; cf.\ \cite{qu85}.}
        $\mu_\partial\,\widehat{\otimes}\,\xi$ on the closed manifold $\partial W$, which is the Clifford module $\mu_\partial\,\otimes\,\xi$ with the Clifford action $c^{\,\mu_\partial}\otimes \id_\xi$ while carrying a $\bbz_2$-grading $\tau^{\mu_\partial} \otimes \tau^\xi$; compare with \cref{twi}. There is a splitting of the Dirac operator $D^{\mu_\partial\widehat{\otimes}\xi}$ on $\mu_\partial\,\widehat{\otimes}\,\xi$ given by
    \begin{equation}\label{z2grading3}
        D^{\mu_\partial\widehat{\otimes}\xi} = \begin{pmatrix}
            0  & D^{\mu_\partial\widehat{\otimes}\xi,-}\\
            D^{\mu_\partial\widehat{\otimes}\xi,+} & 0 
        \end{pmatrix},
    \end{equation}
    where the restriction operators $D^{\mu_\partial\widehat{\otimes}\xi,\pm}$ are formally adjoint to each other and
    \begin{align}
        D^{\mu_\partial\widehat{\otimes}\xi,+}\colon & \Gamma\bket{(\mu^+_\partial \otimes \xi^+) \oplus (\mu^-_\partial \otimes \xi^-)} \to \Gamma\bket{(\mu^-_\partial \otimes \xi^+) \oplus (\mu^+_\partial \otimes \xi^-)},\\
        D^{\mu_\partial\widehat{\otimes}\xi,-}\colon & \Gamma\bket{(\mu^-_\partial \otimes \xi^+) \oplus (\mu^+_\partial \otimes \xi^-)} \to \Gamma\bket{(\mu^+_\partial \otimes \xi^+) \oplus (\mu^-_\partial \otimes \xi^-)}.\label{z2grading5}
    \end{align}
\begin{lem}[cf.\ {\cite[Cor.~3.10]{cz24}}]\label{band2}
    Let $W$ be an odd-dimensional connected compact band. Under the above settings, for any constants $\delta>0, T\in\bbr$ and any smooth functions $h\colon W\to [\delta,\infty)$, $\psi\colon W\to \bbr$, we have
    \begin{equation}\label{partitionfor}
        \ind \calb^W_{h,\psi,T,s}  =  \ind \ket{D^{\mu_\partial \widehat{\otimes} \xi}|_{\partial_-W}}.
    \end{equation}
\end{lem}
\begin{proof}
    First, we take a smooth function $\psi_\infty\colon W^\infty\to \bbr$ such that
    \begin{equation}\label{psipm1}
        \psi_\infty \equiv \pm 1 \qquad\text{on } W^\pm.
    \end{equation}    
    Given constants $\delta>0,T\in\bbr$ and smooth functions $h\colon W\to [\delta,\infty)$ and $\psi\colon W\to \bbr$, the operator $\calb^W_{h,\psi,T,s}$ is homotopic to $\calb^W_{1,\psi_\infty,0,s}$ through a continuous family of self-adjoint Fredholm operators with fixed domain $\rmh^1_{\vartheta,s}(W,S)$ by \cref{fredholm}. So
    \begin{equation}\label{homotopy}
        \ind \calb^W_{h,\psi,T,s}
        = \ind \calb^W_{\psi_\infty,s},
    \end{equation}
    where we write $\calb^W_{\psi_\infty} = \calb^W_{1,\psi_\infty,0}$ and $\calb^W_{\psi_\infty,s} = \calb^W_{1,\psi_\infty,0,s}$ for simplicity.

    By \eqref{psipm1}, the formula \eqref{positive} holds for $\psi_\infty$ outside $W$ when $h\equiv 1$ and $T=0$. By \cref{fredholm}, we obtain a self-adjoint Fredholm operator (cf.\ \cite[Cor.~3.10]{cz24})
    \begin{equation}\label{binfty}
        \calb^{W^\infty}_{\psi_\infty} = D + \psi_\infty\vartheta = 
        \begin{pmatrix}
            0 & D^{\mu\otimes\xi} - \sqrt{-1}\psi_\infty\\
            D^{\mu\otimes\xi} + \sqrt{-1}\psi_\infty & 0
        \end{pmatrix}
    \end{equation}
    on the noncompact manifold $W^\infty$ without boundary. Applying the splitting theorem \cite[Thm.~3.6]{cz24}, \cite[Thm.~8.17]{bb12} to $W^\infty$ along $\partial W = \partial_-W\sqcup \partial_+W$, we have
    \begin{equation}\label{splittingindex}
        \ind \calb^{W^\infty}_{\psi_\infty} = \ind \calb^{W^-}_{-1,-1} + \ind \calb^W_{\psi_\infty,s} + \ind \calb^{W^+}_{1,1}.
    \end{equation}
    For $u\in \Gamma_{\vartheta,\pm 1}(S|_{W^\pm})$, the formula \eqref{spectral2} implies
    \begin{equation}\label{square1}
        \int_{W^\pm}\babs{\calb^{W^\pm}_{\pm 1,\pm 1}\,u}^2
        \ge \int_{W^\pm}|u|^2 + \int_{\partial W^\pm}|u|^2.
    \end{equation}
    It follows from \eqref{square1} that (this is exactly \cite[Lem.~3.7]{cz24})
    \begin{equation}\label{zeroindex}
        \ind \calb^{W^\pm}_{\pm 1,\pm 1} = 0.
    \end{equation}

    Recall that the Clifford action on $\mu\otimes\xi$ is defined to be $c^{\,\mu}\otimes\tau^\xi$. For the splitting $\mu\otimes\xi = (\mu\otimes\xi^+)\oplus(\mu\otimes\xi^-)$, the Dirac operator $D^{\mu\otimes\xi}$ can be written as
    \begin{equation}\label{z2grading2}
        D^{\mu\otimes\xi} = D^{\mu\otimes\xi^+} \oplus (-D^{\mu\otimes\xi^-}),
    \end{equation}
    where $D^{\mu\otimes\xi^\pm}$ are the Dirac operators on the Clifford modules $\mu\otimes\xi^\pm$ with Clifford actions $c^{\,\mu}\otimes \id_{\xi^\pm}$; compare with \cref{twi}. It follows from \eqref{binfty} and \eqref{z2grading2} that
    \begin{align}\label{partindex}
        \begin{split}
            \ind \calb^{W^\infty}_{\psi_\infty}
            &= 
            \ind\ket{D^{\mu\otimes\xi} + \sqrt{-1}\psi_\infty}\\
            & = \ind\ket{D^{\mu\otimes\xi^+} + \sqrt{-1}\psi_\infty} + \ind\ket{-D^{\mu\otimes\xi^-} + \sqrt{-1}\psi_\infty}\\
            & = \ind\ket{D^{\mu\otimes\xi^+} + \sqrt{-1}\psi_\infty} + \ind\ket{D^{\mu\otimes\xi^-} - \sqrt{-1}\psi_\infty}\\
            &= 
            \ind\ket{D^{\mu\otimes\xi^+} + \sqrt{-1}\psi_\infty} - \ind\ket{D^{\mu\otimes\xi^-} + \sqrt{-1}\psi_\infty}.
        \end{split}
    \end{align}
    Let $D^{\mu_\partial \otimes \xi^\pm}$ be the Dirac operators on the Clifford modules $\mu_\partial \otimes \xi^\pm$ with Clifford actions $c^{\,\mu_\partial}\otimes \id_{\xi^\pm}$. Applying the partitioned index theorem \cite[Cor.~1.9]{an93} to the Clifford modules $\mu\otimes\xi^\pm$ separately, we have\footnote{\cite[Cor.~1.9]{an93} implies that $\ind\ket{D^{\mu\otimes\xi^\pm} + \sqrt{-1}\psi_\infty} = \ind (D^{\mu_\partial\otimes\xi^\pm}|_{\partial_+W})$, where the latter index is determined by the $\bbz_2$-grading associated with the \emph{outward} unit normal; see the discussions preceding \cite[Thm.~1.5]{an93}. By the cobordism theorem \cite[Cor.~21.6]{bw}, we have $\ind (D^{\mu_\partial\otimes\xi^\pm}|_{\partial_+W}) = \ind (D^{\mu_\partial\otimes\xi^\pm}|_{\partial_-W})$, where the latter index is determined by \eqref{z2grading}.}
    \begin{equation}\label{partition2}
        \ind\ket{\ket{D^{\mu \otimes \xi^\pm} + \sqrt{-1}\psi_\infty}|_{W^\infty}} = \ind (D^{\mu_\partial \otimes \xi^\pm}|_{\partial_-W}).
    \end{equation}
    By \eqref{homotopy}, \eqref{splittingindex}, \eqref{zeroindex} and \eqref{partindex}--\eqref{partition2}, we have
    \begin{align}\label{ind10}
        \begin{split}
            \ind \calb^W_{h,\psi,T,s} &= \ind \calb^{W^\infty}_{\psi_\infty}\\
            &=
            \ind (D^{\mu_\partial \otimes\xi^+}|_{\partial_-W}) - \ind (D^{\mu_\partial \otimes\xi^-}|_{\partial_-W})\\
            &= \ind \ket{D^{\mu_\partial \widehat{\otimes}\xi}|_{\partial_-W}},
        \end{split}
    \end{align}
    where the last equality follows from \eqref{z2grading3}--\eqref{z2grading5}.
\end{proof}

\section{A Lichnerowicz-Type Formula for Almost Isometric Foliations}\label{3}

In \cref{3.1.1}, we first establish a Bochner--Weitzenböck identity for the partial sum $\sum_{j=1}^{q}c(e_j)\nabla_{e_j}$, analogous to the classical formula for $D = \sum_{j=1}^{n}c(e_j)\nabla_{e_j}$: $D^2 = -\Delta +\calr$. We then analyze its asymptotic behavior in the adiabatic limit for almost isometric foliations in \cref{3.1.2}. The main results of this section are \cref{part,partcurva}, as well as \cref{part3,partcur1}.

\subsection{Partial Sum of the Dirac Operator}\label{3.1.1}

We assume the following setting (where $F$ is not necessarily integrable):
\begin{set}\label{3.1}
    Let $(M,g^{TM})$ be a Riemannian manifold with the orthogonal splitting
    \begin{equation}\label{osm0}
        TM = F\oplus F^\perp,
        \quad
        g^{TM} = g^F \oplus g^{F^\perp}.
    \end{equation}
    Let $p\colon TM\to F$ and $p^\perp \colon TM\to F^\perp$ be the orthogonal projections. Set $n = \dim M$ and $q= \rk F$. Let $\{{e}_j\}_{j=1}^n$ denote a local orthonormal frame of $TM$ such that $\{{e}_j\}_{j=1}^q$ spans $F$ and $\{{e}_j\}_{j=q+1}^n$ spans $F^\perp$. Let $(S,\nabla,c)$ be a real or complex Clifford module on $(M,g^{TM})$.
    
    If $\partial M\ne \varnothing$, further assume $F$ is transverse to $\partial M$ and $F^\perp|_{\partial M} \subseteq T\partial M$. Suppose $e_1=\nu$ is the inward unit normal on $\partial M$ with respect to $g^{TM}$. Then $\nu \in \Gamma(F|_{\partial M})$.
\end{set}

Recall that $R= \nabla^2$ is the curvature of the connection $\nabla$ on $S$. On $M$, we have the following well-defined partial sums (distinguished by the subscript $q$):
\begin{align}\label{partterm}
    \begin{split}
        &D_q \coloneqq \sum_{j=1}^{q}c(e_j)\nabla_{e_j},\quad
        \Delta_q \coloneqq \sum_{j=1}^{q}\bbac{\nabla_{e_j}\nabla_{e_j} - \nabla_{p\nabla^{TM}_{e_j}e_j}}\\
        \text{and}\quad&
        \calr_q \coloneqq \frac{1}{2}\sum_{i,j=1}^{q}c(e_i)\,c(e_j)\,R(e_i,e_j).
    \end{split}
\end{align}

By \eqref{bdclif}, we have the Clifford module $(S|_{\partial M},\nabla^\partial,c^\partial)$ on $\partial M$. For the local orthonormal frame $\{\nu,{e}_2,\dots,e_{q}\}$ of $F|_{\partial M}$, we have the well-defined partial sums
\begin{equation}\label{partbd}
    D^\partial_q \coloneqq \sum_{j=2}^{q}c^\partial(e_j)\nabla_{e_j}^\partial
    \quad\text{and}\quad
    H^{TM}_q \coloneqq -\frac{1}{q-1}\sum_{j=2}^{q} \binn{\nabla^{TM}_{e_j}\nu,e_j}
\end{equation}
on $\partial M$. It follows from \eqref{a1}--\eqref{a3}, \eqref{bdclif} and $\inn{\nu,\nabla^{TM}_{e_j}\nu}=0$ that
\begin{align}\label{bdcnu}
    \begin{split}
        D^\partial_q c(\nu) 
        &= 
        \sum_{j=2}^{q}c^\partial(e_j)\bbket{\nabla_{e_j} + \frac{1}{2}c(\nabla^{TM}_{e_j}\nu)\,c(\nu)}c(\nu)\\
        &= 
        \sum_{j=2}^{q}c^\partial(e_j)\bbket{c(\nu)\nabla_{e_j} + c(\nabla^{TM}_{e_j}\nu) - \frac{1}{2}c(\nabla^{TM}_{e_j}\nu)}\\
        &= 
        \sum_{j=2}^{q}c^\partial(e_j)\,c(\nu)\bbket{\nabla_{e_j} - \frac{1}{2}c(\nu)\,c(\nabla^{TM}_{e_j}\nu)}
        =
        - c(\nu)D^\partial_q.
    \end{split}
\end{align}

    If we further assume \cref{setdd}, then we have $D^\partial_{q}\vartheta = \vartheta D^\partial_{q}$ by \eqref{c6}.     
    Thus, $D^\partial_{q} \chi = -\chi D^\partial_{q}$. As in \eqref{bd1}, we have for any $u\in \Gamma_{\vartheta,s}(S)$ that
    \begin{equation}\label{zero1}
        \binn{u|_{\partial M},D^\partial_{q} (u|_{\partial M})}= 0.
     \end{equation}

    We adopt the convention that the integral $\int_{\partial M}$ vanishes if $\partial M$ is empty, and the empty sum $\sum_{j=n+1}^{n}$ is zero. The following proposition generalizes \eqref{lic}.
\begin{prop}\label{part}
    Under \cref{3.1}, we have for all $u\in \Gamma_\cm(S)$ that
    \begin{align}\label{part2}
        \begin{split}
            \int_M |D_q u|^2 =  \int_M \bbbac{\sum_{j=1}^{q}&|\nabla_{e_j} u|^2 + \inn{u,\calr_q u}} + I_1 + I_2 +I_3\\
            &+ \int_{\partial M} \bbinn{u,\bbket{\frac{q-1}{2}H^{TM}_q - D^\partial_q} u} + I_4,
        \end{split}
    \end{align}
    where 
    \begin{align}
        I_1 &=\int_M\bbinn{u,c\bbket{\sum_{j=1}^{q} p^\perp\nabla^{TM}_{e_j}e_j - \sum_{k=q+1}^{n}p\nabla^{TM}_{e_k}e_k} D_qu},\label{aa1}\\
        I_2 &= \int_M\bbinn{u,\frac{1}{2}\sum_{j,k=1}^{q}c(e_j)c(e_k)\nabla_{p^\perp[e_j,e_k]}u - \sum_{j=1}^{q}\sum_{k=q+1}^{n}c(e_j)c(e_k)\nabla_{p\nabla^{TM}_{e_j}e_k}u},\label{aa2}\\
        I_3 & = -\sum_{k=q+1}^{n} \int_M\binn{u,\nabla_{p\nabla^{TM}_{e_k}e_k}u},\label{aa3}\\
        I_4 &= \frac{1}{2}\sum_{j=2}^{q}\,\sum_{k=q+1}^{n}\int_{\partial M}\binn{u,c(e_j)c(e_k)u}\binn{\nabla^{TM}_{e_j}\nu,e_k}.\label{aa4}
    \end{align}
\end{prop}

The proof is similar to that of \eqref{lic} and the details can be found in \cref{app}. The primary difference is the absence of symmetry, which means some terms no longer cancel and there are the remainder terms $I_1$ to $I_4$.

For $X\in \Gamma(F)$, let $P^q_X$ be the \emph{Penrose operator along $F$} defined by
\begin{equation}\label{pen}
    P^q_X = \nabla_X + \frac{1}{q}c(X)D_q.
\end{equation}
The formula \eqref{part2} can be refined by the following lemma; see \cref{part3}.
\begin{lem}\label{lem2}
    Under \cref{3.1}, we have
    \begin{equation}\label{spe3}
        \sum_{j=1}^{q}|\nabla_{e_j}u|^2 =  \sum_{j=1}^{q}|P^q_{e_j}u|^2 + \frac{1}{q}|D_q u|^2.
    \end{equation}
\end{lem}
\begin{proof}
    This is similar to the case of $q=n$ (cf.\ \cite[(5.8)]{bhm15}). We have
    \begin{equation}
            \sum_{j=1}^{q}|P^q_{e_j}u|^2 =\sum_{j=1}^{q}|\nabla_{e_j}u|^2 - \frac{2}{q}|D_q u|^2 + \frac{1}{q}|D_q u|^2 = \sum_{j=1}^{q}|\nabla_{e_j}u|^2 - \frac{1}{q}|D_q u|^2
    \end{equation}
    by directly expanding $|P^q_{e_j}u|^2= |\nabla_{e_j}u + \frac{1}{q}c(e_j)D_qu|^2$.
\end{proof}

For the Levi-Civita connection $\nabla^{TM}$ of $g^{TM}$, let $R^{TM} = (\nabla^{TM})^2$ be the Riemannian curvature tensor. We define $R^{TM}_{ijkl} = -\inn{R^{TM}(e_i,e_j)\,e_k,e_l}$. When $M$ is spin and $S = S(TM)$ is the Hermitian bundle of spinors, recall that
\begin{align}\label{cur-3}
    \begin{split}
        \calr_q &= \frac{1}{2}\sum_{i,j=1}^{q}c(e_i)\,c(e_j)\,R^{S(TM)}(e_i,e_j)\\
        &= -\frac{1}{8}\sum_{i,j=1}^{q}\sum_{k,l=1}^{n} R^{TM}_{ijkl}\,c(e_i)\,c(e_j)\,c(e_k)\,c(e_l).
    \end{split}
\end{align}
When $q=n$ and $S = S(TM)$, Lichnerowicz \cite{lic63} proved 
\begin{equation}\label{lichnerowicz}
\calr_n = \frac{k^{TM}}{4},
\end{equation}
where $k^{TM}$ is the scalar curvature of $g^{TM}$. The following proposition extends \eqref{lichnerowicz}.

\begin{prop}\label{partcurva}
    Under \cref{3.1}, we have 
    \begin{align}\label{subcur}
        \begin{split}
            -\frac{1}{8}\sum_{i,j=1}^{q} \sum_{k,l=1}^{n} R^{TM}_{ijkl}&\,c(e_i)\,c(e_j)\,c(e_k)\,c(e_l)\\
            =
            \frac{1}{4}\sum_{i,j=1}^{q}&R^{TM}_{ijij} -\frac{1}{2}\sum_{i,j=1}^{q}\,\sum_{l=q+1}^{n}R^{TM}_{ijil}\,c(e_j)\,c(e_l)\\
            &- \frac{1}{8}\sum_{i,j=1}^{q}\,\sum_{k,l=q+1}^{n} R^{TM}_{ijkl}\,c(e_i)\,c(e_j)\,c(e_k)\,c(e_l).
        \end{split}
    \end{align}
\end{prop}

A detailed proof can be found in Appendix \ref{app}; see \eqref{distinct}--\eqref{a16}. Again, the presence of the additional terms in \eqref{subcur} is due to the loss of symmetry. 

\subsection{Adiabatic Limits on Almost Isometric Foliations}\label{3.1.2}
On almost isometric foliations in the sense of Connes \cite[\S3]{connes} (see \cref{aliso}), it turns out that \eqref{part2} and \eqref{subcur} exhibit well-behaved adiabatic limits. We first recall some definitions. 

Suppose $M$ is a smooth manifold with boundary. Let $F\subseteq TM$ be an integrable subbundle transverse to $\partial M$ equipped with a smooth metric $g^F$. 

For any $x\in M$, let $\scrl_x$ denote the unique leaf passing through $x$, which is a smooth immersed submanifold such that $F|_{\scrl_x} = T\scrl_x$. 
The \emph{leafwise scalar curvature} $k^F \in C^\infty(M)$ with respect to $g^F$ is defined by
\begin{equation}\label{k1}
    k^F(x) =  k^{T\scrl_x}(x)
\end{equation}
for $x\in M$, where $k^{T\scrl_x}$ is the scalar curvature of $\scrl_x$ with respect to $g^F|_{\scrl_x}$. 

Since $F$ is transverse to $\partial M$, it induces a foliation on $\partial M$ with the same codimension \cite[Ex.~1.2.21]{cc}. For $x\in\partial M$, the connected components of $\scrl_x \cap \partial M$ are leaves in $\partial M$, and the immersed submanifold $\scrl_x$ has boundary $\scrl_x\cap \partial M$. We define the \emph{leafwise boundary mean curvature} $H^F\in C^\infty(\partial M)$ with respect to $g^F$ by
\begin{equation}\label{h1}
    H^F(x) = H^{T\scrl_x}(x)
\end{equation}
for $x\in\partial M$, where $H^{T\scrl_x}(x)$ is the mean curvature of $\scrl_x$ at $x$ with respect to the inward unit normal; see \eqref{h100}. 

Next, we take a splitting $TM = F \oplus F^\perp$, and let $p^\perp\colon TM\to F^\perp$ be the natural projection. Then $F^\perp$ is isomorphic to the normal bundle $TM/F$. Following \cite{bo68}, we define the Bott connection $\nabla^{F^\perp}$ to be any connection on $F^\perp$ such that for $X\in \Gamma(F)$ and $Y\in \Gamma(F^\perp)$, we have
\begin{equation}\label{bot}
    \nabla^{F^\perp}_X Y = p^\perp [X,Y].
\end{equation} 

Let $G$ denote the holonomy groupoid of the foliated manifold $(M,F)$, which consists of equivalence classes $[\gamma]$ of piecewise smooth curves $\gamma\colon [0,1]\to M$ with images contained in a single leaf; see \cite{wink83}. Two piecewise smooth curves $\gamma,\bar{\gamma}$ are equivalent if $\gamma(0) = \bar{\gamma}(0), \gamma(1) = \bar{\gamma}(1)$, and they have the same holonomy element $F^\perp_{\gamma(0)} \to F^\perp_{\gamma(1)}$ defined by the parallel transport with respect to the Bott connection $\nabla^{F^\perp}$. The action of $[\gamma]\in G$ on $F^\perp_{\gamma(0)}$ is defined to be the holonomy element $F^\perp_{\gamma(0)}\to F^\perp_{\gamma(1)}$ along $\gamma$ with respect to $\nabla^{F^\perp}$.

Consider a further splitting $F^\perp = \fp \oplus \ffp$. Let $g^{TM}$ be a Riemannian metric on $M$ defined by the orthogonal splitting
\begin{equation}\label{osm}
    TM = F\oplus \fp \oplus \ffp,\quad
    g^{TM} = g^F \oplus g^\fp \oplus g^\ffp.
\end{equation}
Set $q=\rk F$, $m=\rk(F\oplus \fp)$ and $n=\dim M$. Recall the following definition.

\begin{defn}[{\cite[\S3]{connes}\,\cite[Def.~A.1]{lz}}]\label{aliso}
    We say the foliated manifold $(M,F)$ is \emph{almost isometric} with respect to $\fp\oplus \ffp$ and $g^\fp \oplus g^\ffp$ (or with respect to \eqref{osm}) if the action of the holonomy groupoid $G$ on $\fp \oplus \ffp$ takes the form
    \begin{equation}\label{aliso1}
        \begin{pmatrix}
            O_1 & 0\\
            A & O_2
        \end{pmatrix},
    \end{equation}
    where $O_1 \in \ortho(m-q)$ and $O_2 \in \ortho(n-m)$ are orthogonal matrices, and $A$ is an $(n-m)\times (m-q)$ matrix.
\end{defn}

\begin{exam}\label{aliso2}
    If $(M,F)$ is almost isometric with respect to \eqref{osm}, then $(M \times \bbs^1, \pro^*_M F)$ is also almost isometric with respect to each of the following splittings:
    \begin{align}
        \pro_M^*\fp\oplus (\pro_M^*\ffp\oplus T\bbs^1) 
        &\quad\text{and}\quad 
        \pro_M^*g^\fp \oplus (\pro_M^*g^\ffp\oplus \di\theta^2), \\
        (\pro_M^*\fp\oplus T\bbs^1) \oplus \pro_M^*\ffp 
        &\quad\text{and}\quad
        (\pro_M^*g^\fp\oplus \di\theta^2) \oplus \pro_M^*g^\ffp,
    \end{align}
    where $\pro_M \colon M\times\bbs^1 \to M$ is the natural projection and $\theta\in \bbr/2\pi\bbz$. This is because $\frac{\partial}{\partial \theta}\in T\bbs^1$ is parallel along any curve in a single leaf of $\pro_M^*F$ with respect to the Bott connection on $\fp\oplus \ffp\oplus T\bbs^1$.
\end{exam}

Next, we will consider a family of rescaled Riemannian metrics $g^{TM}_\bv$ on $M$.

\begin{set}\label{3.3}
    Assume \eqref{osm}. Suppose $F,\ffp \subseteq TM$ are integrable subbundles and $(M,F)$ is almost isometric with respect to \eqref{osm}. If $\partial M\ne\varnothing$, further assume that $F$ is transverse to $\partial M$ and $(\fp\oplus\ffp)|_{\partial M} \subseteq T\partial M$. Let $\wnu\in \Gamma(F|_{\partial M})$ be the inward unit normal on $\partial M$ with respect to $g^{TM}$.
    
    Following \textnormal{\cite[(1.11)]{zhang17}}, let $g^{TM}_\bv$ be the rescaled Riemannian metrics defined by
    \begin{equation}\label{split1}
        g^{TM}_{\beta,\varepsilon} = \beta^2 g^{F} \oplus \dfrac{g^{\fp}}{\varepsilon^2} \oplus g^{\ffp}\qquad (0<\beta,\varepsilon\le 1).
    \end{equation}
    Let $\{\we_j\}_{j=1}^n$ be a local orthonormal frame on $(M,g^{TM})$ such that $\{\we_j\}_{j=1}^q$, $\{\we_j\}_{j=q+1}^m$ and $\{\we_j\}_{j=m+1}^n$ span the subbundles $F$, $\fp$ and $\ffp$ respectively. We define a local orthonormal frame $\{e_j\}_{j=1}^n$ associated with $g^{TM}_{\beta,\varepsilon}$ by
    \begin{equation}\label{scale}
        e_j =
        \begin{cases}
            \beta^{-1}\we_j, & \text{if } 1\le j \le q,\\
            \varepsilon\we_j, & \text{if } q+1\le j \le m,\\
            \we_j, & \text{if } m+1\le j \le n.
        \end{cases}
    \end{equation}
    On $\partial M$, we further require that $\we_1 = \wnu$, and define $\nu = \beta^{-1}\wnu$.
\end{set}

Following \cite{zhang17}, we add "$\bv$" as a subscript or superscript to indicate geometric objects associated with the metric $g^{TM}_\bv$. The omission of "$\bv$" corresponds to the standard case where $\beta = \varepsilon = 1$.

Write $\inn{\cdot,\cdot}_\bv$ for $g^{TM}_\bv$, and let $\nabla^{TM,\bv}$ be its Levi-Civita connection. The connection matrix $\omega^{\bv} = (\omega^{\bv}_{jk})$ of $\nbv$ is a matrix of $1$-forms defined by
\begin{equation}\label{ome2}
    \omega^{\bv}_{jk}(e_i) = \binn{\nbv_{e_i}e_j,e_k}_\bv.
\end{equation}
Changing the inner products in \cite[(1.19--1.27)]{zhang17} from $g^{TM}$ to $g^{TM}_\bv$, we have (the key ingredients  are the almost isometric property and the integrability of $F,\ffp$)
\begin{lem}[{\cite[Lem.~1.3]{zhang17}}]\label{esti}
    Assume \cref{3.3}. Decompose the square matrix $\omega^{\bv}$ into blocks according to the indices $\{1,\dots, q\}, \{q+1,\dots,m\}$ and $\{m+1,\dots,n\}$. For small $\beta,\varepsilon>0$, the block matrix $\omega^{\bv}(e_i)$ has asymptotic estimates
\begin{equation*}
\begin{minipage}{0.32\linewidth}
\centering
$1\le i \le q$
\begin{equation*}
    \begin{pNiceMatrix}[margin, columns-width=auto]
            O(\frac{1}{\beta})  & O(\varepsilon) & O(1)\\
            O(\varepsilon)  & O(\frac{1}{\beta})  & O(\frac{\varepsilon}{\beta})\\
            O(1) & O(\frac{\varepsilon}{\beta}) & O(\frac{1}{\beta})
            \CodeAfter
            \tikz \draw [dashed] (2-|1) -- (2-|4) (3-|1) -- (3-|4) (1-|2) -- (4-|2) (1-|3) -- (4-|3) ;
        \end{pNiceMatrix}
\end{equation*}
\end{minipage}
\hfill
\begin{minipage}{0.33\linewidth}
\centering
$q+1\le i \le m$
\begin{equation*}
    \begin{pNiceMatrix}[margin, columns-width=auto]
            O(\varepsilon)  & O{\scriptstyle(\beta\varepsilon^2)} & O(\frac{\varepsilon}{\beta})\\
            O{\scriptstyle(\beta\varepsilon^2)} & O(\varepsilon) & O(1)\\
            O(\frac{\varepsilon}{\beta}) & O(1) & O(\varepsilon)
            \CodeAfter
            \tikz \draw [dashed] (2-|1) -- (2-|4) (3-|1) -- (3-|4) (1-|2) -- (4-|2) (1-|3) -- (4-|3) ;
        \end{pNiceMatrix}
\end{equation*}
\end{minipage}
\hfill
\begin{minipage}{0.32\linewidth}
\centering
$m+1\le i \le n$
\begin{equation*}
    \begin{pNiceMatrix}[margin, columns-width=auto]
            O(1) & O(\frac{\varepsilon}{\beta}) & 0 \\
            O(\frac{\varepsilon}{\beta}) & O(1)  & O(\varepsilon) \\
            0 & O(\varepsilon) & O(1)
            \CodeAfter
            \tikz \draw [dashed] (2-|1) -- (2-|4) (3-|1) -- (3-|4) (1-|2) -- (4-|2) (1-|3) -- (4-|3) ;
        \end{pNiceMatrix}
\end{equation*}
\end{minipage}
\end{equation*}
    holding uniformly on any compact subset. In particular, we have $\omega^{\bv}_{jk}(e_i) = O(\frac{1}{\beta})$ for $1\le i \le q$, and $\omega^{\bv}_{jk}(e_i) = O(1+\frac{\varepsilon}{\beta})$ for $q+1\le i \le n$.
\end{lem}

\begin{prop}
    Under \cref{3.3}, if $\partial M\ne \varnothing$, then we have
    \begin{equation}\label{meancur}
        H^{TM}_{\bv;q}
        = \frac{H^{F}}{\beta}
        \quad \text{and} \quad
        H^{TM}_{\bv}
        =
        \frac{(q-1)H^{F}}{(n -1)\beta} + O(\beta\varepsilon^2)
    \end{equation}
    holding uniformly on any compact subset of $\partial M$ for sufficiently small $\beta, \varepsilon > 0$.
\end{prop}
\begin{proof}
    Note that $p\nabla^{TM,\beta,\varepsilon}p$ is the Levi-Civita connection on each leaf for the metric $\beta^2g^F$, and the leafwise boundary mean curvature associated with $\beta^2g^F$ is $\frac{H^F}{\beta}$. Then
    \begin{equation}\label{meancur1}
        H^{TM}_{\bv;q}
        = -\frac{1}{q-1}\sum_{j=2}^{q}\ttbinn{\nabla^{TM,\bv}_{e_j}\nu,e_j}
        = \frac{H^{F}}{\beta}.
    \end{equation}
    By \cref{esti} and \eqref{ome2}, the formula
    \begin{align}\label{meancur2}
        \begin{split}
            H^{TM}_{\bv}
            &=\frac{(q-1)H^{F}}{(n -1)\beta} -\dfrac{1}{n-1}\sum_{j=q+1}^{n}\omega^\bv_{1j}(e_j)
            =
            \frac{(q-1)H^{F}}{(n -1)\beta} + O(\beta\varepsilon^2)
        \end{split}       
    \end{align}
    holds uniformly on any compact subset of $\partial M$ for sufficiently small $\beta, \varepsilon > 0$.    
\end{proof}

For convenience, we quote several results from \cite[\S1.2--1.3]{zhang17} (cf.\ \cite[\S1$\frac{7}{8}$]{g96} and \cite[\S6.5.2]{g23}). For $b=1,2$, let $\nabla^{F^\perp_b,\bv}$ be the Euclidean connections on $F^\perp_b$ defined by
\begin{equation}\label{defcon}
    \nabla^{F^\perp_b,\bv} = p_b^\perp\nabla^{TM,\bv}p_b^\perp.
\end{equation}

\begin{prop}\label{curvature10}
    Assume \cref{3.3}. Suppose $\bv>0$ are sufficiently small.
    
   \begin{enumerate}[label={\rm{(\arabic*)}}]
    \item \textup{\cite[(1.30)--(1.31)]{zhang17}} It holds uniformly on any compact subset of $M$ that
    \begin{align}\label{cur2}
        k^{TM,\bv}= \frac{k^F}{\beta^2} + O\bbbket{1+\frac{\varepsilon^2}{\beta^2}},
        \quad
        \sum_{i,j=1}^{q}\rbv_{ijij} = \frac{k^F}{\beta^2} + O(1).
    \end{align}
    \item \textup{\cite[(1.45)--(1.46)]{zhang17}} Suppose $q+1 \le k,l \le m$ for $b=1$ and $m+1 \le k,l \le n$ for $b=2$. It holds uniformly on any compact subset of $M$ that
    \begin{equation}\label{cur31}
        R^{F^\perp_b,\bv}_{ijkl} = 
        \begin{cases}
            O(\varepsilon^2),& \text{if } 1\le i,j\le q,\\
            O(\beta^{-1}),& \text{otherwise},
        \end{cases}
    \end{equation}
    where $R^{F^\perp_b,\bv}_{ijkl} = -\binn{R^{F^\perp_b,\bv}(e_i,e_j)\,e_k,e_l}_\bv$ and $R^{F^\perp_b,\bv}= \bket{\nabla^{F^\perp_b,\bv}}^2$.
   \end{enumerate}
\end{prop}

The following theorem combined with \cref{partcur1} is a  Lichnerowicz-type formula with coefficient $\frac{q}{q-1}$ for almost isometric foliations in the adiabatic limit; see \eqref{est9} and compare with \cite{fr80}, \cite[\S5.2]{bhm15}. 
\begin{thm}\label{part3}
    Assume \cref{3.3}. For each pair of $\bv>0$, let $(S_\bv,\nabla^\bv,c_\bv)$ be a real or complex Clifford module on $\ket{M, g^{TM}_\bv}$ with Dirac operator $D_\bv$. Recall
    \begin{equation}
        P^{\bv;q}_{e_j} = \nabla^\bv_{e_j} + \frac{1}{q}c_\bv(e_j)D_{\bv;q}.
    \end{equation}
    Let $L_0\subseteq M$ be an arbitrary compact subset. Then the following formula holds uniformly for sufficiently small $\bv>0$ and all $u\in \Gamma_\cm(S_\bv)$ supported in $L_0$:
    \begin{align}\label{spe5}
        \begin{split}
            \int_M|D_\bv u|^2 
            &=
            \int_M \bbbac{\frac{q+O(\beta)}{q-1}\sum_{j=1}^{q}\babs{P^{\bv;q}_{e_j}u}^2 + \sum_{j=q+1}^{n}\babs{\nabla^\bv_{e_j}u}^2}\\
            &\quad+
            \int_M \bbbinn{u, \bbbket{\calr^\bv + \frac{1+O(\beta)}{q-1}\calr^\bv_{q} + O\bket{\beta^{-1}}} u}\\
            &\quad\quad+ \int_{\partial M}\bbbinn{u, \bbbket{\frac{q H^{F}}{2\beta} - D^\partial_\bv - \frac{1+O(\beta)}{q-1}D^\partial_{\bv;q} + O(1)}u}.
        \end{split}
    \end{align}
    Here $\inn{\cdot,\cdot}$ and $|\cdot|$ denote the Hermitian metric and its norm on $S_\bv$ for simplicity. The integrals are associated with the Riemannian measures induced by $g^{TM}_\bv$.
\end{thm}
\begin{proof} 
    Applying \cref{part} to $F^\perp = \fp\oplus\ffp$ and $g^{TM}_\bv$, we obtain
    \begin{align}\label{rem1}
        \begin{split}
            I^\bv_1
            &=\bbbac{\sum_{j,l=1}^{q}\,\sum_{k=q+1}^{n} - \sum_{k,l=1}^{q}\,\sum_{j=q+1}^{n}}
            \int_M \omega^\bv_{jk}(e_j)\cdot\binn{u,\cbv(e_k)\,\cbv(e_l)\nnbv_{e_l}u}.
        \end{split}
    \end{align}
    By \cref{esti}, we have $\omega^\bv_{jk}(e_j) = O(1)$ in \eqref{rem1}. So there exists a constant $C > 0 $ depending on $L_0$ but independent of $\bv$ such that \eqref{remainder1}--\eqref{remainder4} hold:
    \begin{equation}\label{remainder1}
        \babs{I^\bv_1}
        \le C\sum_{l=1}^{q}\int_M|u|\cdot\babs{\nnbv_{e_l}u}
        \le C\beta\sum_{l=1}^{q}\int_M\babs{\nnbv_{e_l}u}^2 + \frac{C}{\beta}\int_M|u|^2.
    \end{equation}
    Since $F$ is integrable, we have $p^\perp[e_j,e_k]=0$ when $1\le j,k\le q$. So 
    \begin{align}\label{rem2}
        \begin{split}
            \babs{I^\bv_2}
            &= \bbbabs{-\sum_{j,l=1}^{q}\,\sum_{k=q+1}^{n}\int_M \omega^\bv_{kl}(e_j)\cdot\binn{u,\cbv(e_j)\,\cbv(e_k)\nnbv_{e_l}u}}\\
            &\le C\beta\sum_{l=1}^{q}\int_M\babs{\nnbv_{e_l}u}^2 + \frac{C}{\beta}\int_M|u|^2.
        \end{split}
    \end{align}
    Similarly, we have
    \begin{align}
        \begin{split}
            \babs{I^\bv_3} 
            &= \bbbabs{\sum_{k=q+1}^{n}\,\sum_{l=1}^{q}\int_M\omega^\bv_{kl}(e_k)\cdot\binn{u,\nabla^\bv_{e_l}u}} \\
            &\le 
            C\beta\sum_{l=1}^{q}\int_M\babs{\nnbv_{e_l}u}^2 + \frac{C}{\beta}\int_M|u|^2,
        \end{split}
    \end{align}
    as well as 
    \begin{equation}\label{remainder4}
        \babs{I^\bv_4} =\bbbabs{\frac{1}{2}\sum_{j=2}^{q}\,\sum_{k=q+1}^{n}\int_{\partial M} \omega^\bv_{1k}(e_j)\cdot\inn{u,\cbv(e_j)\,\cbv(e_k)u}}
        \le 
        C\int_{\partial M}|u|^2.
    \end{equation}
    It follows from \eqref{part2} and \eqref{remainder1}--\eqref{remainder4} that
    \begin{align}\label{spe4}
        \begin{split}
            \int_M |D_{\bv;q} u|^2 = \bket{1+O(\beta)} \sum_{j=1}^{q} &\int_M \babs{\nabla^\bv_{e_j} u}^2 + \int_M\bbinn{u, \bbket{\calr^\bv_{q} + O\bket{\beta^{-1}}} u} \\ 
              + &\int_{\partial M} \bbinn{u,\bbket{\frac{q-1}{2}H^{TM}_{\bv;q} - D^\partial_{\bv;q} + O(1)}u}.
        \end{split}
    \end{align}
    
    In this case, we integrate \eqref{spe3} over $M$ and then multiply by $q$ to obtain
    \begin{equation}\label{spe31}
        q\sum_{j=1}^{q}\int_M\babs{\nabla^\bv_{e_j}u}^2 =  q\sum_{j=1}^{q}\int_M\babs{P^{\bv;q}_{e_j}u}^2 + \int_M|D_{\bv;q} u|^2.
    \end{equation}
    Inserting \eqref{spe4} into \eqref{spe31}, by $\frac{1}{q-1+O(\beta)} = \frac{1+ O(\beta)}{q-1}$, we have
    \begin{align}\label{med}
        \begin{split}
            \sum_{j=1}^{q}\int_M \babs{\nabla^\bv_{e_j} u}^2 &= \frac{1+O(\beta)}{q-1}\int_M\bbac{q\sum_{j=1}^{q}  \babs{P^{\bv;q}_{e_j} u}^2 + \binn{u, \bket{\calr^\bv_{q} + O\ket{\beta^{-1}}} u}} \\ 
            & \qquad+\frac{1+O(\beta)}{q-1}\int_{\partial M} \bbinn{u,\bbket{\frac{q-1}{2}H^{TM}_{\bv;q} - D^\partial_{\bv;q} + O(1)}u}.
        \end{split}
    \end{align}
    Apply \eqref{lic} to $D_\bv$, and then insert \eqref{med} into it to conclude the proof.
\end{proof}

We turn to \cref{partcurva}. The curvature matrix  $\varOmega^\bv = \di\omega^\bv - \omega^\bv \wedge \omega^\bv$ of $\nbv$ is a matrix $\varOmega^\bv = \bket{\varOmega^\bv_{ij}}$ of $2$-forms. Recall $R^{TM,\bv}_{klij} = - \varOmega^\bv_{ij}(e_k,e_l)$. So
    \begin{align}\label{cur6}
        \begin{split} 
            R^{TM,\bv}_{klij}&= - e_k\bket{\omega^\bv_{ij}(e_l)} + e_l\bket{\omega^\bv_{ij}(e_k)} + \omega^\bv_{ij}\bket{[e_k,e_l]}\\
            &\qquad+ \sum_{b=1}^{n}\omega^\bv_{ib}(e_k)\cdot\omega^\bv_{bj}(e_l) - \sum_{b=1}^{n}\omega^\bv_{ib}(e_l)\cdot\omega^\bv_{bj}(e_k).
        \end{split}
    \end{align}
    By \eqref{ome2} and the formula \eqref{levi} for the Levi-Civita connection $\nbv$, we have
\begin{equation}\label{kos}
        2\cdot\omega^\bv_{jk}(e_i) = \binn{[e_i,e_j],e_k}_\bv - \binn{[e_i,e_k],e_j}_\bv - \binn{[e_j,e_k],e_i}_\bv.
\end{equation}
    Applying \eqref{split1}--\eqref{scale} to \eqref{kos},  we observe that $\omega^\bv_{jk}(e_i)$ and $\we_l\bket{\omega^\bv_{jk}(e_i)}$ have the same order of asymptotic estimates for all $1\le i,j,k,l\le n$.

\begin{thm}\label{partcur1}
    Assume \cref{3.3}. For each pair of $\bv>0$, let $(S_\bv,\nabla^\bv,c_\bv)$ be a real or complex Clifford module on $\ket{M, g^{TM}_\bv}$. For $a \in \{q,n\}$, we have
    \begin{equation}\label{cur4}
        -\frac{1}{8}\sum_{i,j=1}^{a}\, \sum_{k,l=1}^{n} R^{TM,\bv}_{ijkl}\,c_\bv(e_i)\,c_\bv(e_j)\,c_\bv(e_k)\,c_\bv(e_l) = \frac{k^F}{4\beta^2} + O\bbbket{\frac{1}{\beta}+\frac{\varepsilon}{\beta^2}}
    \end{equation}
    holding uniformly on any compact subset of $M$ when $\bv>0$ are small.
\end{thm}
\begin{proof}
    \textbf{Case 1: $a = n$.}  Applying \cref{partcurva} to the case $F = TM$, we obtain
    \begin{align}
        \begin{split}
            -\frac{1}{8}\sum_{i,j=1}^{n}\, \sum_{k,l=1}^{n} R^{TM,\bv}_{ijkl}\,c_\bv(e_i)\,c_\bv(e_j)\,c_\bv(e_k)\,c_\bv(e_l)
            =
            \frac{k^{TM,\bv}}{4}.
        \end{split}
    \end{align}
    Then the result follows from \eqref{cur2}, which is \cite[(1.30)]{zhang17}.
    
    \textbf{Case 2: $a =q$.}  \cref{partcurva} implies that 
    \begin{align}\label{cur5}
        \begin{split}
            &-\frac{1}{8}\sum_{i,j=1}^{q}\, \sum_{k,l=1}^{n} R^{TM,\bv}_{ijkl}\,c_\bv(e_i)\,c_\bv(e_j)\,c_\bv(e_k)\,c_\bv(e_l)\\
            &\qquad=
            \frac{1}{4}\sum_{i,j=1}^{q}R^{TM,\bv}_{ijij} -\frac{1}{2}\sum_{i,j=1}^{q}\,\sum_{l=q+1}^{n}R^{TM,\bv}_{ijil}\,c_\bv(e_j)\,c_\bv(e_l)\\
            &\qquad\qquad- \frac{1}{8}\sum_{i,j=1}^{q}\,\sum_{k,l=q+1}^{n} R^{TM,\bv}_{ijkl}\,c_\bv(e_i)\,c_\bv(e_j)\,c_\bv(e_k)\,c_\bv(e_l).
        \end{split}
    \end{align}

    Suppose now $1\le i,j\le q$ and $q+1\le k,l \le n$. By \cref{esti}, the first two terms on the right-hand side of \eqref{cur6} are $O(1+\frac{\varepsilon}{\beta})$. Together with \eqref{split1}--\eqref{scale}, the third term on the right-hand side of \eqref{cur6} satisfies
    \begin{align}\label{cur7}
        \begin{split}
            \omega^\bv_{ij}\bket{[e_k,e_l]} =& \bbbac{\sum_{b=1}^{q} + \sum_{b=q+1}^{m} + \sum_{b=m+1}^{n}}\, \bbket{\omega^\bv_{ij}(e_b) \cdot\ttbinn{[e_k,e_l],e_b}}\\
            =& O\bket{\tfrac{1}{\beta}}\,O(\beta) + O(\varepsilon)\,O\bket{\tfrac{1}{\varepsilon}} + O(1)\,O(1) = O(1).
        \end{split}
    \end{align}
    The last two terms on the right-hand side of \eqref{cur6} are
    \begin{equation}
        O\bbbket{1+\frac{\varepsilon}{\beta}}\cdot O\bbbket{1+\frac{\varepsilon}{\beta}} = O\bbbket{1+\frac{\varepsilon}{\beta}+\frac{\varepsilon^2}{\beta^2}} = O\bbbket{1 + \frac{\varepsilon^2}{\beta^2}}
    \end{equation} 
    since $\frac{\varepsilon}{\beta} \le \frac{1}{2}\bket{1+\frac{\varepsilon^2}{\beta^2}}$. When $1\le i,j\le q$ and $q+1\le k,l\le n$, we conclude that
    \begin{equation}\label{cur8}
        \rbv_{ijkl} = \rbv_{klij} =  O\bbbket{1 +\frac{\varepsilon^2}{\beta^2}}.
    \end{equation}

    Suppose now $1\le i,j \le q$ and $q+1\le l \le n$. For the second term on the right-hand side of \eqref{cur5}, it follows from \eqref{cur6} that
    \begin{align}\label{cur9}
        \begin{split}
            R^{TM,\bv}_{ijil} = R^{TM,\bv}_{ilij} =& - e_i\bket{\omega^\bv_{ij}(e_l)} + e_l\bket{\omega^\bv_{ij}(e_i)} + \omega^\bv_{ij}\bket{[e_i,e_l]}\\
            & + \sum_{b=1}^{n}\omega^\bv_{ib}(e_i)\cdot\omega^\bv_{bj}(e_l) - \sum_{b=1}^{n}\omega^\bv_{ib}(e_l)\cdot\omega^\bv_{bj}(e_i).
        \end{split}
    \end{align}
    By \cref{esti} and \eqref{split1}--\eqref{scale}, the third term on the right-hand side of \eqref{cur9} is
    \begin{align}\label{cur10}
        \begin{split}
            &\omega^\bv_{ij}\bket{[e_i,e_l]}\\
            &\quad =\bbbac{\sum_{b=1}^{q} + \sum_{b=q+1}^{m} + \sum_{b=m+1}^{n}}\,\bbket{\omega^\bv_{ij}(e_b)\cdot\ttbinn{[e_i,e_l],e_b}}\\
            &\quad =
        \begin{cases}
            O\bket{\frac{1}{\beta}}\, O(\varepsilon)  + O(\varepsilon)\, O\bket{\frac{1}{\beta}} + O(1) \,O\big(\frac{\varepsilon}{\beta}\big)   = O\big(\frac{\varepsilon}{\beta}\big), & \text{if } q+1 \le l \le m,\\
            O\bket{\frac{1}{\beta}} \,O(1) + O(\varepsilon) \,O\big(\frac{1}{\beta\varepsilon}\big) + O(1) \,O\big(\frac{1}{\beta}\big) = O\big(\frac{1}{\beta}\big), & \text{if } m+1 \le l \le n.\\
        \end{cases}
        \end{split}
    \end{align}
    Similar to \eqref{cur8}, by \cref{esti} and \eqref{cur9}--\eqref{cur10}, we have
    \begin{align}\label{cur11}
        \begin{split}
            \rbv_{ijil} &= \frac{O(1)}{\beta} + O\bbbket{\frac{1}{\beta}} + O\bbbket{\frac{1}{\beta}} +  O\bbbket{\frac{1}{\beta}\bbket{1+\frac{\varepsilon}{\beta}}} + O\bbbket{\bbket{1+\frac{\varepsilon}{\beta}}\frac{1}{\beta}}\\
            &= O\bbbket{\frac{1}{\beta}+ \frac{\varepsilon}{\beta^2}}.
        \end{split}
    \end{align}
    Since \eqref{cur2}, \eqref{cur8} and \eqref{cur11} hold uniformly on any compact subset when $\bv>0$ are small, the result follows from \eqref{cur5}.
\end{proof}

\section{Connes Fibration and Leafwise Band Width}\label{4}

The Connes fibration introduced in \cite[\S5]{connes} has the advantage of admitting an almost isometric structure, which is essential in \cref{part3,partcur1}. In \cref{fiber}, we recall the construction of the Connes fibration with spin vertical tangent bundle. We prove \cref{widmain,widmain2} in \cref{4.2}, and then discuss further examples of foliated bands with infinite vertical $\widehat{A}$-cowaist in \cref{4.3}.

\subsection{Connes Fibrations over Foliated Bands}\label{fiber}
Let $W$ be a compact band equipped with an integrable subbundle $F$ transverse to $\partial W$. Let $g^F$ be a metric on $F$. Suppose $\dim W\ge 2$ and $F$ has codimension $a\ge 1$.

Recall that $\partial W = \partial_-W\sqcup\partial_+W$, where $\partial_\pm W$ are nonempty unions of connected components. As in \eqref{glue}, we obtain a noncompact, foliated manifold
\begin{equation}\label{vinfty}
    W^\infty = W^- \cup_{\partial_-W} W \cup_{\partial_+W} W^+,
\end{equation}
where $W^\pm = \partial_\pm W \times [1,+\infty)$ are half-cylinders. Let $F$ still denote the integrable subbundle on $W^\infty$. We extend $g^F|_W$ to a complete metric $g^F$ on $W^\infty$.

Following \cite[\S5]{connes} (cf.\ \cite[\S2.1]{zhang17}), we obtain a smooth fibration\footnote{Throughout this paper, the term \emph{smooth fibration} refers to a smooth, locally trivial fiber bundle.}
\begin{equation}\label{cm}
    \bar{\pi}\colon \calw^\infty\to W^\infty,
\end{equation} 
called the \emph{Connes fibration}\footnote{This specific construction was previously utilized in \cite[p.~457]{g23} and \cite[p.~3]{sy25}.}, such that for each $x\in W^\infty$, the fiber $\calw^\infty_{x} = \bar{\pi}^{-1}(x)$ is a product of four copies of the space of Euclidean metrics on the vector space $T_xW^\infty/F_x$. The vertical tangent bundle $T^V\calw^\infty$ is spin and of even rank since it is a direct sum of four isomorphic subbundles. We obtain a noncompact band
\begin{equation}
    \calw \coloneqq \bar{\pi}^{-1}(W)
\end{equation}
with $\partial_\pm\calw = \bar{\pi}^{-1}(\partial_\pm W)$. Meanwhile, $\calw^\infty\setminus\calw$ is diffeomorphic to $\partial\calw\times(1,\infty)$.

Let $\calh_a$ denote the noncompact homogeneous space $\gl(a,\bbr)/\ortho(a)$. The fibers of $\calw^\infty$ are diffeomorphic to $\calh_a\times\calh_a\times\calh_a\times\calh_a$, which is simply connected. Note that $\calh_a$  carries a $\gl(a,\bbr)$-invariant complete metric of nonpositive sectional curvature, and the product metric on $\calh_a\times\calh_a\times\calh_a\times\calh_a$ inherits these properties. Let $g^{T^V\calw^\infty}$ denote the induced smooth metric on the vertical tangent bundle $T^V\calw^\infty$. The Cartan--Hadamard theorem guarantees that for each $x\in W^\infty$, any two points $q_1,q_2$ in the fiber $\calw^\infty_{x}$ can be joined by a unique geodesic with respect to $g^{T^V\calw^\infty}|_{\calw^\infty_{x}}$.

Since $F$ is transverse to $\partial W$, we can take a splitting
\begin{equation}\label{bdreq}
    TW^\infty = F \oplus F^\perp\quad \text{with}\quad\fz|_{\partial W}  \subset T\partial W.
\end{equation}
Then $F^\perp$ is isomorphic to the normal bundle $TW^\infty/F$. The Bott connection on $F^\perp$ defined in \eqref{bot} lifts to a connection on the fibration $\calw^\infty$. Let $T^H\calw^\infty\subset T\calw^\infty$ denote the corresponding horizontal tangent bundle, which is isomorphic to $\bar{\pi}^*(TW^\infty)$. Then the splitting $TW^\infty = F \oplus F^\perp$ lifts to a splitting $T^H\calw^\infty = \calf \oplus \f$, where $\calf$ and $\f$ are isomorphic to $\bar{\pi}^*F$ and $\bar{\pi}^*F^\perp$ respectively. 

Since $F$ is integrable and the Bott connection on $F^\perp$ is leafwise flat, the subbundle $\calf$ is also integrable. The metric $g^F$ lifts to a metric $g^\calf = \bar{\pi}^*g^F$ on $\calf$. For any $v\in \calw^\infty$, the fiber $\calf_{1,v}^\perp$ is identified with $F^\perp_{\bar{\pi}(v)}$ under the projection $\bar{\pi}\colon \calw^\infty \to W^\infty$. By definition, each point $v\in \calw^\infty$ is a quadruple of Euclidean metrics on $F^\perp_{\bar{\pi}(v)}\cong T_{\bar{\pi}(v)}W^\infty/F_{\bar{\pi}(v)}$. The first metric in this quadruple determines a metric on $\calf^\perp_{1,v}$. Then the subbundle $\f$ carries an induced smooth metric $g^\f$. 

In what follows, we adopt the notation $\calf_2^\perp = T^V\calw^\infty$. By \eqref{bdreq}, we have
\begin{equation}
    (\f\oplus\ff)|_{\partial \calw} \subset T\partial \calw.
\end{equation}
Endow $\calw^\infty$ with the Riemannian metric $g^{T\calw^\infty}$ defined by the orthogonal splitting
\begin{equation}
    \label{splittingmetric}
    T\calw^\infty = \calf\oplus\calf_1^\perp\oplus\calf_2^\perp,\quad g^{T\calw^\infty} = g^\calf \oplus g^{\calf_1^\perp} \oplus g^{\calf_2^\perp}.
\end{equation}
Let $\nu$ denote the inward unit normal to $\partial \calw$ with respect to $g^{T\calw^\infty}$. It follows that $\nu\in \Gamma(\calf|_{\partial \calw})$. Furthermore, according to \cite[Lem.~5.2]{connes}, the foliated manifold $(\calw^\infty,\calf)$ is almost isometric with respect to \eqref{splittingmetric} in the sense of \cref{aliso}.

Choose four metrics on $\fz$. This determines an embedded section 
\begin{equation}\label{embedsec}
    \iota\colon W^\infty\hookrightarrow \calw^\infty
\end{equation}
of the Connes fibration $\bar{\pi}\colon \calw^\infty \to W^\infty$. For any $v\in \calw^\infty \setminus \iota(W^\infty)$, recall that there exists a unique geodesic from $v$ to $\iota(\bar{\pi}(v))$ in the fiber $\calw^\infty_{\bar{\pi}(v)}$. Let $\sigma(v) \in \ff|_v$ \phantomsection\label{sigmadef}
be the unit vector tangent to this geodesic at $v$ pointing toward $\iota(\bar{\pi}(v))$. Let
\begin{equation}\label{rho}
    \rho(v)= \dist_{\calw^\infty_{\bar{\pi}(v)}}(v,\iota(\bar{\pi}(v)))
\end{equation}
denote the length of this geodesic, and define $\rho|_{\iota(W^\infty)} = 0$. Then $\sigma$ is a smooth vector field defined on $\calw^\infty \setminus \iota(W^\infty)$, and $\rho$ is a smooth function on $\calw^\infty \setminus \iota(W^\infty)$.

Let $p_2^\perp \colon T\calw^\infty \to \ff$ be the orthogonal projection, and let $\nabla^{T\calw^\infty}$ denote the Levi-Civita connection of $g^{T\calw^\infty}$. Then $\nabla^\ff \coloneqq p_2^\perp \nabla^{T\calw^\infty} p_2^\perp$ is a Euclidean connection on $\ff$. It is independent of $g^\calf\oplus g^\f$ since $\ff$ is integrable.

Since $W$ is compact, the following key lemma of \cite{zhang17} still holds (the key ingredient is that the restriction of $g^{\ff}$ to each fiber is $\gl(a,\bbr)$-invariant and has nonpositive sectional curvature).
\begin{lem}[{\cite[Lem.~2.1]{zhang17}}]\label{key3}
    There exists a constant $C>0$, depending only on the embedding $\iota\colon  W \to \calw$, such that for any $X\in \Gamma(\calf)$ with $|X|\le 1$, the following pointwise inequalities hold on $\calw \setminus \iota(W)$:
    \begin{equation}
        |X(\rho)|\le C, \qquad \bbabs{\nabla^{\calf^\perp_2}_X\sigma} \le \frac{C}{\rho}.
    \end{equation}
\end{lem}

The noncompactness of the fibers of $\calw^\infty$ causes difficulties. So we consider
\begin{equation}
    \calw^\infty_r \coloneqq  \{v\in \calw^\infty: \rho(v)\le r\}
    \quad \text{and}\quad 
    \ws \coloneqq  \{v\in \calw: \rho(v)\le r\}
\end{equation}
for $r>0$. In \cite[(2.33)]{zhang17}, pseudodifferential operators on the double of $\ws$ were used. However, the boundary value problems for pseudodifferential operators are considerably more difficult than those for differential operators\footnote{Since $\ws$ has corners, the approach in Yu--Zhang \cite{yz18} involving Atiyah--Patodi--Singer boundary value problems does not apply directly.}. In order to construct globally defined differential operators, we will carefully extend the structures on $\ws$ to its double.

Let $\inter M$ denote the interior of a manifold $M$. Given $r> 0$, take a sufficiently small $\epsilon>0$ such that the collar neighborhood $\inter\calw^\infty_{r+\epsilon}\setminus \calw^\infty_{r-\epsilon}$ is isomorphic to $\rho^{-1}(r)\times (-\epsilon,\epsilon)$ as smooth fibrations over $W^\infty$. Let $\calv$ be another copy of $\wsin$. By gluing $\calv$ to $\wsin$ along $\rho^{-1}(r)$ and utilizing the natural reflection on the above collar neighborhood, we construct a smooth double manifold $\wwin$ (\cref{doubleofv2}). View $\inter\calw^\infty_{r+\epsilon}$ as a subset of $\wwin$. Then \eqref{embedsec} induces
\begin{equation}\label{iotaembed}
    \iota\colon W^\infty \hookrightarrow \wwin.
\end{equation} 
The projection $\bar{\pi}\colon \wsin\to W^\infty$ naturally extends to a proper smooth submersion
\begin{equation}\label{dcm}
    \widetilde{\pi}_r\colon\wwin \to W^\infty.
\end{equation}
It is a smooth fibration by Ehresmann's Fibration Theorem \cite{ehr51}.
Each fiber of $\wwin$ is the double of the corresponding fiber in $\wsin$. Let $\ww$ be the compact band
\begin{equation}\label{band1}
    \ww \coloneqq \widetilde{\pi}_r^{-1}(W)
\end{equation} 
with $\partial_\pm \ww = \widetilde{\pi}_r^{-1}(\partial_\pm W)$. The subset $\wwin \setminus \ww$ is diffeomorphic to $\partial\ww \times (1,\infty)$.

\begin{figure}[htbp]
    \centering
    \includegraphics[width=\textwidth]{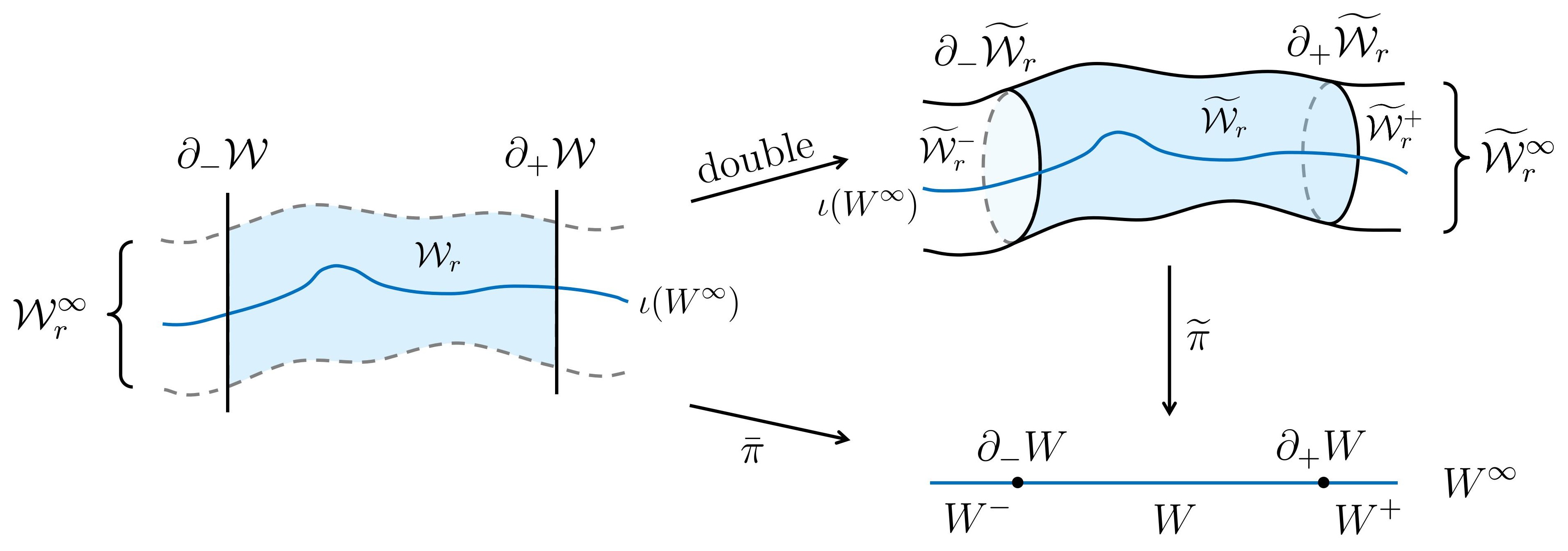}
    \caption{The Connes fibration of $W^\infty$ and the double of $\win$.}
    \label{doubleofv2}
\end{figure} 

The splitting $T^H\calw^\infty|_{\wsin} = \calf|_{\wsin}\oplus\f|_{\wsin}$ can be extended to a splitting $T^H\wwin = \hfz \oplus \hf$ which lifts $TW^\infty = F \oplus F^\perp$. In general, $\hfz$ might not be integrable outside $\wsin$. Let $g^{\hfz}\coloneqq\widetilde{\pi}_r^*g^F$ be the pullback metric on $\hfz$. Let $\hff\coloneqq T^V\wwin$ denote the vertical tangent bundle, which is also spin. Extend $g^{\f} \oplus g^{\ff}$ to a smooth metric $g^{\hf} \oplus g^{\hff}$ on $\wwin$. Endow $\wwin$ with the Riemannian metric $g^{T\wwin}$ defined by
\begin{equation}\label{spmet2}
    T\wwin = \hfz \oplus \hf \oplus \hff, \quad g^{T\wwin} = g^{\hfz} \oplus g^{\hf} \oplus g^{\hff}.
\end{equation}
By construction, we have $T\wwin|_{\wsin} = T\calw^\infty|_\wsin$ and $g^{T\wwin}|_{\wsin} = g^{T\calw^\infty}|_{\wsin}$.

Since $\ff$ is spin and of even rank, we have the $\bbz_2$-graded Hermitian (dual) bundle of spinors $S(\ff)^* = S_+(\ff)^*\oplus S_-(\ff)^*$. It is equipped with the Hermitian connection $\nabla^{S(\ff)^*} = \nabla^{S_+(\ff)^*} \oplus \nabla^{S_-(\ff)^*}$ induced by $\nabla^{\ff}$ and a bundle homomorphism $\widehat{c}\colon \ff \to \en{S(\ff)^*}$ such that for $X,Y\in \Gamma(\ff)$, the endomorphism $\widehat{c}(X)$ is self-adjoint and
\begin{equation}\label{hatc}
    \widehat{c}(X)^2 = |X|^2,\quad \bcu{\nabla_X^{S(\ff)^*},\widehat{c}(Y)} = \widehat{c}\bket{\nabla_X^\ff Y}.
\end{equation}

Let ${f}_0\colon[0,1]\to[0,1]$ be a smooth function such that ${f}_0\equiv 0$ on $\bcu{0,\frac{1}{4}}$ and ${f}_0\equiv 1$ on $\bcu{\frac{1}{2},1}$. Then ${f}_0\big(\frac{\rho}{r}\big)\colon \wsin\to[0,1]$ extends to a globally defined smooth function
\begin{equation}\label{funf}
    f\colon\wwin \to [0,1]
\end{equation}
such that $f\equiv 1$ on $\wwin\setminus\calw^\infty_{\frac{r}{2}}$. 
    
Given a Hermitian vector bundle $E$ over $W^\infty$, let $\bar{\pi}^*{E}$ denote the pullback Hermitian vector bundle over $\calw^\infty$. Following \cite[\S1b]{bz93} (cf.\ \cite[\S2.3]{zhang17}, \cite[\S3]{zhang23}), we take a Hermitian vector bundle $\caly$ over $\wsin$ such that\footnote{Such a vector bundle $\caly$ always exists; see, e.g., \cite[IV.~Thm.~3.3]{hir76}.} $\bket{S_+\ket{\ffin}^* \otimes \bar{\pi}^*E} \oplus \caly$ is a trivial vector bundle over $\wsin$.
Then $\bket{S_-\ket{\ffin}^* \otimes \bar{\pi}^*{E}} \oplus \caly$ is also a trivial vector bundle over $\wsin \setminus \iota(W^\infty)$, under the identification $\big(\widehat{c}(\sigma)\otimes \id_{\bar{\pi}^*E}\big) \oplus \id_{\caly}$, where $\sigma$ is the unit vector field defined before \eqref{rho}.

Extending the above trivial vector bundles to $\wwin$, we can construct a $\bbz_2$-graded Hermitian vector bundle $\xi = \xi^+ \oplus \xi^-$ over $\wwin$ satisfying (compare \cite[(2.30)]{zhang17})
\begin{equation}\label{xixi}
    \xi^\pm = \big(S_\pm(\ff)^* \otimes \bar{\pi}^*{E}\big) \oplus \caly\quad \text{on } \wsin,
\end{equation}
and construct an odd self-adjoint bundle endomorphism $V^{\xi}\in \Gamma(\en{\xi})$ such that
\begin{equation}\label{ww}
    \begin{cases}
        V^{\xi}|_{\xi^\pm} =  \big(f\,\widehat{c}(\sigma) \otimes \id_{\bar{\pi}^*{E}}\big) \oplus \id_{\caly}, &\text{on } \wsin,\\
        (V^{\xi})^2=\id, &\text{on } \wwin\setminus \calw^\infty_{\frac{r}{2}}.
    \end{cases}
\end{equation}

Take Hermitian connections $\nabla^{\caly}$ and $\nabla^E$ on $\caly$ and $E$, respectively. Let $\nabla^{\bar{\pi}^*{E}}$ be the pullback Hermitian connection on $\bar{\pi}^*E$. Take an even Hermitian connection $\nabla^{\xi}$ on the Hermitian vector bundle $\xi\to\wwin$ such that
\begin{equation}\label{metcon}
    \nabla^{\xi}|_{\xi^\pm} = \bket{\nabla^{S_\pm(\ffin)^*}\otimes \id_{\bar{\pi}^*{E}} + \id_{S_\pm(\ffin)^*}\otimes\nabla^{{\bar{\pi}^*{E}}}} \oplus \nabla^{\caly} \quad\text{on } \wsin.
\end{equation}
For $X\in\Gamma(T\wsin)$, the second formula in \eqref{hatc} implies that 
\begin{equation}\label{changefor}
    \bcu{\nabla^{\xi}_X,V^{\xi}}\big|_{\xi^\pm} = \bbket{\bbket{X(f) \, \widehat{c}(\sigma) + f\,\widehat{c}\bket{\nabla^\ff_X\sigma}} \otimes \id_{\bar{\pi}^*E}} \oplus 0 \quad\text{on } \wsin.
\end{equation}

Given a Clifford module $(\mu,\nabla^\mu,c^{\,\mu})$ on $(\wwin,g^{T\wwin})$, we obtain a Clifford module $(\mu\otimes\xi, \nabla^{\mu\otimes\xi},c^{\,\mu}\otimes\tau^\xi)$ by \cref{twi}, 
where $\tau^\xi = \pm \id|_{\xi^\pm}$. Let $D^{\mu\otimes\xi}$ denote the Dirac operator on $\mu\otimes\xi$. We write $V^{\mu\otimes\xi} = \id_\mu \otimes V^\xi$. Since $V^\xi$ is odd, we have
\begin{equation}\label{changefor2}
    D^{\mu\otimes\xi}\,V^{\mu\otimes\xi} + V^{\mu\otimes\xi}\,D^{\mu\otimes\xi}
    =
    \sum_{j=1}^{n} {c^{\,\mu}(e_j) \otimes \bket{\tau^\xi\bcu{\nabla^{\xi}_{e_j},V^{\xi}}}}
\end{equation} 
on $\wsin$, where $n = \dim \wsin$ and $\{e_j\}_{j=1}^n$ is any orthonormal basis on $\wsin$ with respect to $g^{T\calw^\infty}|_{\wsin}$. By \eqref{changefor}--\eqref{changefor2}, we have on $\wsin$ that (cf.\ \cite[(2.26)]{zhang17})
\begin{align}\label{estmm}
    \begin{split}
        \babs{D^{\mu\otimes\xi} \, V^{\mu\otimes\xi} + V^{\mu\otimes\xi} \, D^{\mu\otimes\xi}}
        &\le \sum_{j=1}^{n}\babs{\bcu{\nabla^{\xi}_{e_j},V^{\xi}}}\\
        &\le \sum_{j=1}^{n}\bbabs{e_j(f) \,\widehat{c}(\sigma) + f\,\widehat{c}\bket{\nabla^\ff_{e_j}\sigma}}\\
        &\le \sum_{j=1}^{n}\bbbac{\bbabs{{f}'_0\bbket{\frac{\rho}{r}}} \cdot \frac{\abs{e_j(\rho)}}{r} + f\babs{\nabla^{\ff}_{e_j}\sigma}}.
    \end{split}
\end{align}
Meanwhile, we have on $\calw^\infty_r\setminus\iota(W^\infty)$ that (cf.\ \cite[(2.25)]{zhang17})
\begin{align}\label{estrr}
        \begin{split}
            |c^{\,\mu}\otimes\tau^\xi(\di\rho)| = \bbbabs{\sum_{j=1}^{n} c^{\,\mu}\bket{e_j(\rho)\cdot e_j}\otimes\tau^\xi} 
            \le \sum_{j=1}^{n} |e_j(\rho)|.
        \end{split}
    \end{align}

\subsection{Width Estimates and Proof of Theorem \ref{widmain}}\label{4.2}

Under the conditions of \cref{widmain}, we further assume $F$ has codimension $\ge 1$. Let $W^\infty$ be the noncompact, foliated manifold defined in \eqref{vinfty} obtained by attaching cylindrical ends to $W$. Then we have the following smooth fibrations defined in \eqref{cm} and \eqref{dcm}:
\begin{equation}
    \bar{\pi}\colon \calw^\infty\to W^\infty,
    \quad
    \widetilde{\pi}_r\colon \wwin\to W^\infty\; (r>0).
\end{equation}
For $0<\bv\le 1$, we rescale the metrics \eqref{splittingmetric} and \eqref{spmet2} as:
\begin{equation}\label{osm31}
    g^{T\calw^\infty}_\bv \coloneqq \beta^2 g^\calf \oplus \frac{g^{\calf_1^\perp}}{\varepsilon^2} \oplus g^{\calf_2^\perp},\quad
    g^{T\wwin}_\bv \coloneqq \beta^2 g^{\hfz} \oplus \frac{g^{\hf}}{\varepsilon^2} \oplus g^{\hff}.
\end{equation}
Let $g^{T\calw}_\bv$, $g^{T\ww}_\bv$ be the restriction of $g^{T\calw^\infty}_\bv$, $g^{T\wwin}_\bv$, respectively. Let $|\cdot|_\bv$ denote the induced fiberwise norms. We will omit the subscripts when $\beta=\varepsilon = 1$.

\begin{lem}\label{dist}
    Given $r>0$ and a metric $g^{F^{\perp}}$ on $F^\perp$, we define $g_\beta = g^F \oplus \frac{1}{\beta^2}g^{F^\perp}$.
    Let $\wid_\beta(W) \coloneqq \dist_\beta(\partial_-W,\partial_+W)$ denote the width of the band $(W,g_\beta)$. Then there exists a number $a_r>0$ such that the width of the band $(\ww,g^{T\ww}_\bv)$ satisfies
    \begin{equation}\label{dis7}
        \wid_\bv\ket{\ww} \ge \beta \cdot \wid_\beta(W)
    \end{equation}
    for all $0<\beta\le 1$ and $0<\varepsilon<\min\{a_r^{-1},1\}$.
\end{lem}
\begin{proof}
    For the differential $\di\widetilde{\pi}_r \colon T\ww\to TW$, since $\ww$ is compact, we have
    \begin{equation}
        a_r 
        \coloneqq 
        \sup_{0\ne X\in \hf|_{\ww}}
        \bbbac{\frac{|\di\widetilde{\pi}_r(X)|}{|X|}}<\infty.
    \end{equation}
    By \cref{key6}, there exists a geodesic $\gamma\colon [0,1] \to \ww$ in $(\ww,g^{T\ww}_\bv)$ realizing $\wid_\bv\ket{\ww}$. 
    When $0<\varepsilon < \min\{a_r^{-1},1\}$, it follows from $g^{\hfz} = \widetilde{\pi}_r^*g^F$ that
    \begin{equation}\label{dis621}
        \frac{|\di \widetilde{\pi}_r (\dot{\gamma})|_\beta }{|\dot{\gamma}|_\bv} 
        = \bbbbac{
            \frac{\babs{
                \di \widetilde{\pi}_r(p(\dot{\gamma}))}^2
                + 
                \beta^{-2}\babs{\di \widetilde{\pi}_r(p_1^\perp(\dot{\gamma}))}^2
                }{
                \beta^2\babs{p(\dot{\gamma})}^2 
                + 
                \varepsilon^{-2}\babs{p_1^\perp(\dot{\gamma})}^2 
                + 
                \babs{p_2^\perp(\dot{\gamma})}^2}
                }^{\frac{1}{2}}  
        \le \frac{1}{\beta},
    \end{equation}
    where $p\colon T\ww\to \hfz$, $p_b^\perp\colon T\ww\to \widetilde{\calf}_b^\perp$ $(b=1,2)$ are orthogonal projections. Then 
    \begin{align}\label{dis62}
            \wid_\beta(W) 
            \le 
            \int_{0}^{1}|\di \widetilde{\pi}_r (\dot{\gamma})|_\beta
            \le
            \frac{1}{\beta} \int_{0}^{1}|\dot{\gamma}|_\bv
            =
            \frac{1}{\beta}\cdot \wid_\bv\ket{\ww}.
    \end{align}
    The result follows.
\end{proof}

To prove \cref{widmain}, we combine \cite{zhang17,cz24}. By the assumption of infinite vertical $\widehat{A}$-cowaist, we will construct deformed Dirac operators on $\ww$ such that the local elliptic boundary value problems \eqref{bdproblem} have nonvanishing indices. Now suppose to the contrary that $\wid_F(W)>t_+-t_-$. Then there is enough room to choose suitable $\psi$ such that the boundary value problems \eqref{bdproblem} is ``invertible'' in the adiabatic limit when $r>0$ is sufficiently large, which is a contradiction.

\begin{proof}[Proof of \cref{widmain}]
    If $\dim W$ is even, the odd-dimensional foliated band $(W\times\bbs^1,\pro^*_W F)$ has infinite vertical $\widehat{A}$-cowaist by \cref{ahat} and satisfies the curvature bounds \eqref{scalmean} for the pullback metric $\pro_W^*g^F$ on $\pro_W^*F$. Meanwhile, we have
    \begin{equation}\label{bandfol1}
        \wid_F(W) = \wid_{\pro^*_W F}(W\times \bbs^1).
    \end{equation}
    So we can assume $\dim W$ is odd\footnote{\label{stable}Indeed, the same argument implies that \cref{widmain} holds when the foliated band $(W,F)$ has \emph{infinite stable vertical $\widehat{A}$-cowaist}, that is, there exists an integer $k\ge 0$ such that $k+\dim W$ is odd and the foliated band $(W\times\bbt^k,\pro_W^*F)$ has infinite vertical $\widehat{A}$-cowaist in the sense of \cref{ahat}\,(1).} without loss of generality. If $F=TW$, then \cref{widmain} reduces to \cite[Thm.~7.6]{cz24}. We further assume $2\le \rk F < \dim W$. 
    
    Suppose to the contrary that $\wid_F(W) > t_+ - t_-$. The proof has three steps.
    
    {\bf Step 1:} Recall the splitting \eqref{bdreq}. Fix a metric $g^{F^\perp}$ on $F^\perp$ and define $g_\beta = g^F \oplus \frac{1}{\beta^2}g^{F^\perp}$. By \cref{key6}, we have 
    \begin{equation}\label{limitwid}
        \lim_{\beta\to 0^+} \wid_\beta(W) = \wid_F(W).
    \end{equation}
    By \cref{dist} and \eqref{limitwid}, there exists $\beta_0,l>0$ such that
    \begin{equation}\label{dist2}
        \wid_\bv(\ww) > 2\beta l > \beta(t_+ - t_-)
    \end{equation}
    for all $0<\beta<\beta_0$ and $0<\varepsilon<\min\{a_r^{-1},1\}$. Set $t_0 = \frac{1}{2}(t_+ + t_-)$. Since $t_\pm \in (-\frac{\pi}{\alpha q}, \frac{\pi}{\alpha q})$, we can further require that the number $l>0$ in \eqref{dist2} satisfies
    \begin{equation}\label{control1}
        -\frac{\pi}{\alpha q}<t_0 -l <t_- < t_+< t_0+l<\frac{\pi}{\alpha q}.
    \end{equation}

    By \cite[Lem.~7.2]{cz24} and \eqref{dist2}, there exists a 1-Lipschitz smooth function
    \begin{equation}
        w\colon \bket{\ww, g^{T\ww}_\bv} \to \bket{[-\beta l,\beta l], \di t^2}
    \end{equation}
    such that $w(\pmww) = \pm \beta l$. In what follows, we identify the differentials, like $\di w$, with the gradient with respect to $g^{T\ww}_\bv$. 

    As in \cite[(7-3)]{cz24}, by \eqref{control1}, take $\eta_0 <1$ sufficiently close to $1$ such that
    \begin{equation}\label{tan2}
        \pm \,\eta_0 \cdot \tan\!\bbbket{\frac{\alpha q(t_0 \pm \eta_0 l)}{2}} > \pm \tan\!\bbket{\frac{\alpha qt_{\pm}}{2}}.
    \end{equation}
    Let $\psi\colon \ww \to \bbr$ be the smooth function  defined by
    \begin{equation}\label{psi}
    \psi = \frac{1}{\beta}\,\psi_{\eta_0}\bbket{\frac{w}{\beta}},
    \quad \text{where }
    \psi_{\eta_0}(t) = \frac{\alpha q\eta_0}{2}\tan\!\bbbket{\frac{\alpha q(t_0 + \eta_0 t)}{2}}.
    \end{equation}
    Since $|\di w|_\bv \le 1$, we have on $\ww$ that 
    \begin{align}\label{psi1}
        \begin{split}
            \psi^2 - |\di\psi|_\bv 
            &\ge
            \frac{1}{\beta^2}\bbbac{\psi^2_{\eta_0}\bbket{\frac{w}{\beta}} - \psi'_{\eta_0}\bbket{\frac{w}{\beta}}} 
            = -\frac{\alpha^2q^2\eta_0^2}{4\beta^2}.
        \end{split}
    \end{align}
    By \eqref{scalmean}, \eqref{tan2} and the fact that $H^\calf = \bar{\pi}^*H^F$ on $\partial\calw$, we have on $\partial_\pm \ww \cap \calw_r$ that
    \begin{equation}\label{psi2}
        \frac{qH^{\calf}}{2\beta} \pm \psi \ge \frac{1}{\beta} \bbbac{\mp \frac{\alpha q}{2}\tan\!\bbket{\frac{\alpha q t_{\pm}}{2}} \pm \psi_{\eta_0}(\pm l)}> 0.
    \end{equation}

    {\bf Step 2:}  
    Since the odd-dimensional foliated band $(W,F)$ has infinite vertical $\widehat{A}$-cowaist, for each $\varpi>0$, there is a Hermitian vector bundle ${E}_\varpi$ over $W$ such that
    \begin{equation}\label{ind3}
        \babs{R^{{E}_\varpi}_F} < \varpi \,\text{ on } W
        \quad\text{and}\quad
        \binn{\widehat{A}(T\partial_-W)\ch({E}_\varpi), [\partial_- W]} \ne 0,
    \end{equation}
    where the norm $\babs{R^{{E}_\varpi}_F}$ is defined in \eqref{normcur} and is associated with $g^F$.
    
    We extend $E_\varpi$ to a Hermitian vector bundle over $W^\infty$, which we still denote by $E_\varpi$. Taking $E = {E}_\varpi$ in \eqref{xixi}, we obtain a $\bbz_2$-graded Hermitian vector bundle
    \begin{equation}
        \xi_\varpi = \xi_\varpi^+\oplus \xi_\varpi^-
    \end{equation}
    over $\wwin$ and odd self-adjoint $V^{\xi_\varpi}\in \Gamma(\en{\xi_\varpi})$ such that
    \begin{align}
        \xi^\pm_\varpi &= \big(S_\pm(\ff)^* \otimes \bar{\pi}^*{E_\varpi}\big) \oplus \caly_\varpi \quad \text{on } \wsin,\label{xibv}\\
        V^{\xi_\varpi}|_{\xi^\pm_\varpi} &=  \big(f\,\widehat{c}(\sigma) \otimes \id_{\bar{\pi}^*{E_\varpi}}\big) \oplus \id_{\caly_\varpi} \quad \text{on } \wsin,\label{wbv}
    \end{align}
    where $f\colon \wwin \to \bbr$ is defined in \eqref{funf}. Meanwhile, by \eqref{ww}, we have
    \begin{equation}\label{vbv2}
        (V^{\xi_\varpi})^2=\id \quad \text{on } \wwin\setminus \calw^\infty_{\frac{r}{2}}.
    \end{equation}

    Suppose $\dim W =m$ and $\dim \calw = n$. Recall that $\rk \calf = q$. Let $\{\we_j\}_{j=1}^n$ denote a local orthonormal frame on $(\calw,g^{T\calw})$ such that $\{\we_j\}_{j=1}^q$, $\{\we_j\}_{j=q+1}^m$, $\{\we_j\}_{j=m+1}^n$ spans $\calf, \f, \ff$ respectively. On $\partial \calw$, we further require that $\we_1 = \wnu$ is the inward unit normal. We define $\nu = \beta^{-1}{\wnu}$, and define $\{e_j\}_{j=1}^n$ as in \eqref{scale}:
\begin{equation}\label{scale1}
        e_j =
        \begin{cases}
            \beta^{-1}\we_j, & \text{if } 1\le j \le q,\\
            \varepsilon\we_j, & \text{if } q+1\le j \le m,\\
            \we_j, & \text{if } m+1\le j \le n.
        \end{cases}
    \end{equation}

    By our constructions, $T\wwin$ is spin and $n =\dim\ww$ is odd. There are two choice for the Hermitian bundle of spinors. We take
    \begin{equation}\label{mudef}
        \mu_\bv = S^1_\bv(T\wwin)
    \end{equation}
    to be the Hermitian bundle of spinors that satisfies
    \begin{equation}\label{vol}
        \bket{\sqrt{-1}\,}^{\frac{n+1}{2}}\,c^{\,\mu_\bv}(e_1)\, \cdots \,c^{\,\mu_\bv}(e_{n}) = -\id_{\mu_\bv},
    \end{equation}
    where $c^{\,\mu_\bv}$ denotes the Clifford action. For the Clifford module $\mu_{\bv}|_{\partial_-\ww}$ defined as in \eqref{bdclif}, there is an identification
    \begin{equation}\label{oddevenclif}
        \mu_{\bv}|_{\partial_-\ww} = S_\bv(T\partial_-\ww)
    \end{equation}
    such that the $\bbz_2$-grading $c^{\,\mu_\bv}(\nu)/\sqrt{-1}$ on $\mu_{\bv}|_{\partial_-\ww}$ is identified with the canonical $\bbz_2$-grading on the Hermitian bundle of spinors $S_\bv(T\partial_-\ww)$. As in \eqref{dirdou1}--\eqref{cldouble1}, we obtain a Clifford module $(S_\bv,\nabla^\bv,c_\bv)$ on $\bket{\wwin,g^{T\wwin}_\bv}$ with $\bbz_2$-grading
    \begin{equation}\label{s1}
        S_\bv = \ket{\mu_\bv\otimes\xi_\varpi} \oplus \ket{\mu_\bv\otimes\xi_\varpi}
    \end{equation}
    where the Clifford action on $\mu_\bv\otimes\xi_\varpi$ is defined by $c^{\,\mu_\bv}\otimes\tau^{\xi_\varpi}$. Let $D_\bv$ denote the Dirac operator on $S_\bv$ defined as in \eqref{dir2}.

    For each $0<\delta< 1$, let $h_{\delta,0} \colon [0,1] \to [\delta,1]$ be a smooth function such that ${h}_{\delta,0}  \equiv 1$ on $\big[0,\frac{3}{4}\big]$, while ${h}_{\delta,0} \equiv\delta$ on $\big[\frac{7}{8},1\big]$. Recall that the function $\rho$ is defined in \eqref{rho}. Then the function ${h}_{\delta,0}  \bket{\frac{\rho}{r}}\colon \ws \to [\delta,1]$ extends to a globally defined smooth function $h_\delta\colon \ww \to [\delta,1]$ such that $h_\delta\equiv \delta$ on $\ww\setminus\calw_{\frac{7r}{8}}$; see \cref{relation}.
    \begin{figure}[ht]
        \centering
        \includegraphics[width=0.38\textwidth]{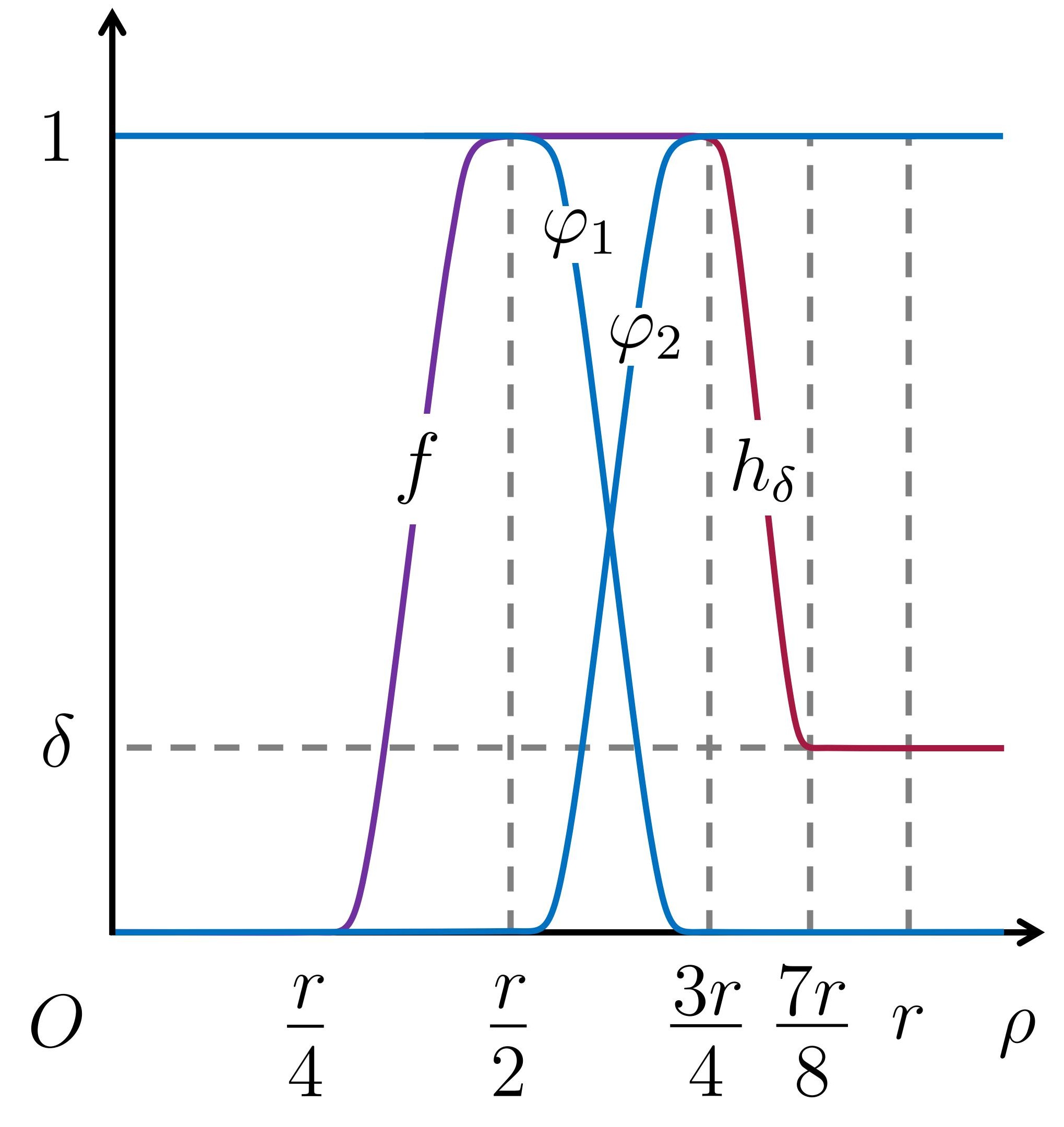}
        \caption{Cut-off functions.} 
        \label{relation}
    \end{figure} 
    
    For $T=\beta^{-1}$ and the function $\psi\colon \ww\to \bbr$ defined in \eqref{psi}, consider the following deformed Dirac operator defined as in \eqref{defd1}:
    \begin{equation}\label{bdef}
        \calb_{h_\delta,\psi,\beta^{-1}}
        = h_\delta D_\bv h_\delta + \psi\vartheta + \beta^{-1}V.
    \end{equation}
    For the choice of signs $s\equiv \pm 1$ on $\pmww$, we obtain a self-adjoint Fredholm operator
    \begin{align}\label{bts}
        \calb_{h_\delta,\psi,\beta^{-1},s}: \rmh^1_{\vartheta,s}(\ww,S_\bv)\to \rml^2(\ww,S_\bv)
    \end{align}
    as in \eqref{bdproblem}. By \cref{band2} and the Atiyah-Singer index theorem \cite{as68} (cf.\ \cite[III.~Thm.~13.8]{lm89}),  \eqref{ind3} and \eqref{oddevenclif}, we have
    \begin{align}\label{ind11}
        \begin{split}
            \ind \calb_{h_\delta,\psi,\beta^{-1},s} 
            &= \ind \bket{D^{S_\bv(T\partial_-\ww)\widehat{\otimes}\xi_\varpi}}\\
            &= \binn{\widehat{A}(T\partial_-\ww)\bket{\ch(\xi^+_\varpi) - \ch(\xi^-_\varpi)},[\partial_-\ww]}\\
            &= \binn{\widehat{A}(T\partial_- W)\ch (E_\varpi),[\partial_-W]}\ne 0;
        \end{split}
    \end{align}
    cf.\ \cite[(2.34)]{zhang17} and \cite[(1.5)]{zhang20}.

    {\bf Step 3:} For simplicity, let $|\cdot|$ and $\|\cdot\|$ denote the fiberwise Hermitian norm and the $\rml^2$-norm on $S_\bv|_\ww$ with respect to $g^{T\ww}_\bv$, respectively. Following \cite[p.~115]{bl91}, let ${\varphi}_0\colon[0,1]\to[0,1]$ be a smooth function such that ${\varphi}_0\equiv 0$ on $\bcu{0,\frac{1}{2}}$ and ${\varphi}_0\equiv 1$ on $\bcu{\frac{3}{4} ,1}$. Let $\varphi_1,\varphi_2\colon\ww \to \bbr$ be smooth functions such that
    \begin{equation}
        \varphi_1 = \frac{1- {\varphi}_0\bket{\frac{\rho}{r}}}{\bbket{{\varphi}_0\bket{\frac{\rho}{r}}^2 + \bket{1- {\varphi}_0\bket{\frac{\rho}{r}}}^2}^{\frac{1}{2}}},
        \quad
        \varphi_2 = \frac{{\varphi}_0\bket{\frac{\rho}{r}}}{\bbket{{\varphi}_0\bket{\frac{\rho}{r}}^2 + \bket{1- {\varphi}_0\bket{\frac{\rho}{r}}}^2}^{\frac{1}{2}}},
    \end{equation}
    on $\ws$, and $\varphi_1,\varphi_2$ are constant outside $\calw_{\frac{3r}{4}}$ (\cref{relation}). Then we have 
    \begin{equation}\label{phisq}
        \varphi_1^2 + \varphi_2^2 \equiv 1.
    \end{equation} 

    For $y,z\in \Gamma_\cm(S_\bv)$ we have 
    \begin{align}\label{normineq}
        \frac{1}{2}\|y + z\|^2 \le \frac{1}{2}(\|y\| + \|z\|)^2 \le \|y\|^2 + \|z\|^2.
    \end{align}
Recall that $\supp\ket{\di\varphi_1},\supp\ket{\di\varphi_2} \subset \calw_{\frac{3r}{4}}$ and $h_\delta\equiv 1$ on $\calw_{\frac{3r}{4}}$. For $u\in \Gamma_\cm(S_\bv)$ and $b=1,2$, it follows from \eqref{bdef}, \eqref{normineq} and $D_\bv(\varphi_b u) = \varphi_b D_\bv u + c_\bv(\di \varphi_b)u$ that
\begin{align}\label{estib12}
    \begin{split}
        \frac{1}{2}\|\mathcal{B}_{h_\delta,\psi,\beta^{-1}}(\varphi_b u)\|^2 
        &\le 
        \|\varphi_b \mathcal{B}_{h_\delta,\psi,\beta^{-1}}\,u\|^2 + \|c_\bv(\di\varphi_b)u\|^2.
    \end{split}
\end{align}
By \eqref{phisq} and \eqref{estib12}, we have (cf.\ \cite[(2.46)]{zhang17})
\begin{align}\label{est7}
\begin{split}
    \|\mathcal{B}_{h_\delta,\psi,\beta^{-1}}\,u\|^2
    &= \|\varphi_1 \mathcal{B}_{h_\delta,\psi,\beta^{-1}}\,u\|^2+ \|\varphi_2 \mathcal{B}_{h_\delta,\psi,\beta^{-1}}\,u\|^2\\
    &\geq 
    \frac{1}{2}\| \mathcal{B}_{h_\delta,\psi,\beta^{-1}}(\varphi_1 u)\|^2+ \frac{1}{2}\| \mathcal{B}_{h_\delta,\psi,\beta^{-1}}(\varphi_2 u)\|^2\\
    &\qquad\qquad -\|c_\bv({\rm d}\varphi_1)u\|^2 -\|c_\bv({\rm d}\varphi_2)u\|^2.
\end{split}
\end{align}

    Now we analyze the asymptotic behavior of \eqref{est7} when $\bv>0$ are small. 
    
    Combining \cref{key3} with \eqref{cldouble1}--\eqref{oddend}, \eqref{estmm}--\eqref{estrr}, \eqref{scale1} and \eqref{s1}, there exists $C>0$ independent of $\bv,\varpi,r>0$ such that we have (cf.\ \cite[(2.27)]{zhang17})
    \begin{equation}\label{asymp1}
        |c_\bv(\di \varphi_b )| \le \frac{C}{\beta r} + O_r(1)\quad \text{and} \quad |[D_\bv,V]| \le \frac{C}{\beta r} + O_r\ket{1}
    \end{equation}
    on $\ws$, where the bounding constant in the $O_r(1)$ term depends on $r$.
    
    There is an orthogonal splitting $S_\bv|_{\calw_r} = S_1|_{\calw_r} \oplus S_2|_{\calw_r}$ for \eqref{s1}, where
    \begin{align}
        S_1|_{\calw_r} &= \oplus_{j=1}^2 \bket{\mu_\bv\otimes S(\ff)^*\otimes\bar{\pi}^*E_\varpi},\\
        S_2|_{\calw_r} &= \oplus_{j=1}^4 \ket{\mu_\bv\otimes \caly_\varpi}.
    \end{align}
    Next, suppose $u\in \Gamma_{\vartheta,s}(S_\bv|_\ww)$, then so are $\varphi_1 u$ and $\varphi_2 u$. There is a splitting
    \begin{equation}\label{splitvarphi1}
        \varphi_1 u = u_1 + u_2,
    \end{equation}
    where $u_b\in \Gamma_{\vartheta,s}(S_b|_{\calw_r})$ since $\chi(S_b|_{\calw_r}) \subseteq S_b|_{\calw_r}$ for $b=1,2$. Moreover, we have 
    \begin{align}
        &\|\varphi_1u\|^2 = \|u_1\|^2 + \|u_2\|^2\label{bu12},\\
        &\| \calb_{h_\delta,\psi,\beta^{-1}}(\varphi_1 u) \|^2 = \| \calb_{h_\delta,\psi,\beta^{-1}}\,u_1 \|^2 + \| \calb_{h_\delta,\psi,\beta^{-1}}\,u_2 \|^2\label{bu13}.
    \end{align}
    Note that we have $V^2u_2 = u_2$ by \eqref{ww} and $[D_\bv,V]u_2 = 0$ by \eqref{changefor}--\eqref{changefor2}. Then it follows from \eqref{spectral} that
    \begin{equation}\label{bu2}
        \| \calb_{h_\delta,\psi,\beta^{-1}}\,u_2 \|^2
        = 
        \int_{\calw_r} \bbac{|(D_\bv+\psi\vartheta)u_2|^2 
        + 
        \beta^{-2}|u_2|^2} \ge \beta^{-2}\|u_2\|^2.
    \end{equation}
    Set $\calr^\bv_n = \calr^\bv$. By \eqref{partterm} and \eqref{mudef}, for $a\in\{q,n\}$, we have on $\ws$ that
    \begin{align}\label{cur13}
        \begin{split}
            \calr^\bv_a|_{\mu_\bv\otimes\xi_\varpi}
            &= \frac{1}{2}\sum_{i,j=1}^{a}\bket{c^{\,\mu_\bv}(e_i)\otimes \tau^{\xi_\varpi}} \,\bket{c^{\,\mu_\bv}(e_j) \otimes \tau^{\xi_\varpi}} \,R^{\mu_\bv\otimes\xi_\varpi}(e_i,e_j)\\
            &=-\frac{1}{8}\sum_{i,j=1}^{a}\,\sum_{k,l=1}^{n} R^{T\ws,\bv}_{ijkl}\,c^{\,\mu_\bv}(e_i)\,c^{\,\mu_\bv}(e_j)\,c^{\,\mu_\bv}(e_k)\,c^{\,\mu_\bv}(e_l)\otimes\id\\
            &\qquad\qquad+ \frac{1}{2}\sum_{i,j=1}^{a} c^{\,\mu_\bv}(e_i)\,c^{\,\mu_\bv}(e_j) \otimes R^{\xi_\varpi}(e_i,e_j),
        \end{split}
    \end{align}   
    where $\{e_j\}_{j=1}^n$ is defined in \eqref{scale1}. Recall that the foliated manifold $(\calw,\calf)$ is almost isometric with respect to the metric $g^{T\calw}_\bv$ defined in \eqref{osm31} and
    \begin{equation}\label{cur43}
        R^{S(\ff)^*}(e_i,e_j)
            = \frac{1}{4}\sum_{k,l=m+1}^{n} R^{\ff}_{ijkl}\,\widehat{c}(e_k)\,\widehat{c}(e_l).
    \end{equation}
    Recall the number $\varpi$ in \eqref{ind3}. By \eqref{cur31}, \eqref{xibv} and \eqref{cur43}, we have on $\ws$ that
    \begin{align}\label{cur44}
        \begin{split}
            R^{\xi_\varpi}(e_i,e_j)|_{S(\ff)^*\otimes\bar{\pi}^*E_\varpi} &= R^{S(\ff)^*}(e_i,e_j) \otimes \id + \id \otimes R^{\bar{\pi}^*E_\varpi}(e_i,e_j)\\
            &=
            \begin{cases}
                O_r(\varepsilon^2) + O\bket{\frac{\varpi}{\beta^2}}, & \text{if } 1\le i,j \le q,\\
               O_{r,\varpi}\bket{\frac{1}{\beta}}, & \text{otherwise}.
            \end{cases}
        \end{split}
    \end{align}
    By \eqref{cur13}, \eqref{cur44} and  \cref{partcur1}, we have 
    \begin{equation}\label{cur12}
        \calr^\bv_a u_1 = \bbbac{\frac{k^\calf}{4\beta^2} + O_{r,\varpi}\bbbket{\frac{1}{\beta} + \frac{\varepsilon}{\beta^2}} + O\bbbket{\frac{\varpi}{\beta^2}}}u_1.
    \end{equation} 

    Note that $\supp (u_1) \subset \calw_{\frac{3r}{4}}$. By \eqref{bd1}, \eqref{zero1}, \eqref{cur12} and \cref{part3}, we obtain
    \begin{align}\label{est9}
        \begin{split}
            \int_\ws |D_\bv u_1|^2 
            &\ge
            \int_\ws \bbbinn{u_1, \bbbket{\calr^\bv + \frac{1+O_r(\beta)}{q-1}\calr^\bv_{q} + O_r\bbbket{\frac{1}{\beta}}\!} u_1}\\
            &\qquad\qquad\qquad+ \int_{\partial\ww \cap \ws} \bbbac{\frac{q H^{\calf}}{2\beta} + O_r(1)}|u_1|^2
            \end{split}\\
            \begin{split}
            &=
            \int_\ws \bbbac{\frac{q}{q-1}\cdot\frac{k^\calf}{4\beta^2} + O_{r,\varpi}\bbbket{\frac{1}{\beta} + \frac{\varepsilon}{\beta^2}} + O\bbbket{\frac{\varpi}{\beta^2}}\!}|u_1|^2\\
            &\qquad\qquad\qquad+ \int_{\partial\ww \cap \ws} \bbbac{\frac{qH^{\calf}}{2\beta} + O_r(1)}|u_1|^2.
        \end{split}
    \end{align}
    Recall that $h_\delta\equiv 1$ on $\supp (u_1)$ and $k^\calf = \bar{\pi}^*k^F \ge \alpha^2 q(q-1)$ on $\calw_r$ by \eqref{scalmean}.
    Combining \eqref{spectral2} with \eqref{psi1}, \eqref{asymp1} and \eqref{est9}, we have (cf.\ \cite[(2.22)]{zhang17})
    \begin{align}\label{spe61}
        \begin{split}
            \|\calb_{h_\delta,\psi,\beta^{-1}}\,u_1\|^2
            &\ge 
            \int_{\calw_r}\!|D_\bv u_1|^2 + \!\int_{\calw_r}\!\!\bbbac{\psi^2  -|\di \psi|_\bv + \frac{V^2}{\beta^2} - \frac{|[D_\bv,V]|}{\beta}}|u_1|^2\\
            &\qquad\qquad+ \int_{\partial {\ww}\cap \calw_r} s\psi |u_1|^2 \\
            &\ge
            \bbbac{\frac{\alpha^2q^2 (1-\eta_0^2)}{4\beta^2} + O_{r,\varpi}\bbbket{\frac{1}{\beta} + \frac{\varepsilon}{\beta^2}}\! + O\bbbket{\frac{\varpi}{\beta^2}+ \frac{1}{\beta^2 r}}\!} \|u_1\|^2\\
            &\qquad\qquad+ 
            \int_{\partial \ww\cap \calw_r}\bbbac{\frac{q H^\calf}{2\beta} + s\psi + O_r(1)}|u_1|^2.
        \end{split}
     \end{align}

    Finally, note that $[D_\bv,V]\in \Gamma(\en{S_\bv})$ is bounded on the compact manifold $\ww$. Combining \eqref{spectral} with \eqref{vbv2}, \eqref{asymp1} and $h_\delta \equiv \delta$, we have
    \begin{align}\label{spe62}
        \begin{split}
        \|\calb_{h_{\delta},\psi,\beta^{-1}}(\varphi_2 u)\|^2
        &\ge
        \int_{\ww} \bbbac{\frac{V^2}{\beta^2} - \frac{h_\delta^2}{\beta}|[D_\bv,V]|} |\varphi_2 u|^2\\
        &\ge 
        \bbbac{\frac{1}{\beta^2} + O\bbbket{\frac{1}{\beta^2 r}} + O_r\bbbket{\frac{1}{\beta}}\!}\int_{\calw_r}|\varphi_2 u|^2 \\
        &\qquad\qquad\quad\;\; +  \bbbac{\frac{1}{\beta^2} + \delta^2\, O_{\bv,\varpi,r}\ket{1}}
        \int_{\ww\setminus \calw_r}|\varphi_2 u|^2;
        \end{split}
    \end{align}
    cf.\ \cite[(2.23)~and~(2.39)]{zhang17}.

    Recall that $s \equiv \pm 1$ on $\pmww$ and $\alpha>0, \eta_0<1$. Combining \eqref{est7} with \eqref{psi2}, \eqref{phisq}, \eqref{asymp1}, \eqref{bu12}--\eqref{bu2} and \eqref{spe61}--\eqref{spe62}, we conclude that there exist sufficiently large $r>0$, sufficiently small $\varpi>0$, sufficiently small $ 0<\beta = \varepsilon <\min\{\beta_0,a_r^{-1}\}$, sufficiently small $\delta>0$ and some $C_0>0$ such that
    \begin{equation}\label{posi}
        \|\calb_{h_\delta,\psi,\beta^{-1}} \,u\|^2 \ge C_0\|u\|^2
    \end{equation}
    for all $u\in \Gamma_{\vartheta,s}(S_\bv)$. Since $\Gamma_{\vartheta,s}(S_\bv)$ is dense in $\rmh^1_{\vartheta,s}(\ww,S_\bv)$, the formula \eqref{posi} also holds for all $u\in \rmh^1_{\vartheta,s}(\ww,S_\bv)$. For the chosen $\delta,\bv, \varpi, r>0$, it follows that
    \begin{equation}\label{ind40}
        \ind \calb_{h_\delta,\psi,\beta^{-1},s} = 0,
    \end{equation}
    which contradicts \eqref{ind11}. We conclude $\wid_F(W)\le t_+- t_-$. 
    
    The formula \eqref{diswid} is a direct consequence of \eqref{width}. 
\end{proof}

\cref{widmain2} follows directly from \cref{widmain}:

\begin{proof}[Proof of \cref{widmain2}]
    (1) Since $W$ is compact and $H^F$ is smooth, there exist $t_\pm$ such that $-\frac{\pi}{\alpha q} < t_- < t_+ < \frac{\pi}{\alpha q}$ and \eqref{scalmean} holds. By \cref{widmain}, we have
    \begin{equation}\label{eq1}
        \wid_F(W) \le t_+-t_- < \frac{2\pi}{\alpha q}.
    \end{equation}
    
    (2) Extend $g^F$ to a metric $g$ on $W$. There exists $l \in (0, \frac{\pi}{\alpha q})$ such that
    \begin{equation}\label{assump1}
        \wid_F(W) \ge \wid_g(W) > l;
    \end{equation}
    see \eqref{noless}. Suppose to the contrary that 
    \begin{equation}\label{contr2}
        \inf_{\partial_- W} H^F + \inf_{\partial_+ W} H^F\ge 0.
    \end{equation}
    There exists $t_-\in (-\frac{\pi}{\alpha q},\frac{\pi}{\alpha q})$ such that $\inf_{\partial_- W} H^F = \alpha\tan\ket{\frac{\alpha q t_-}{2}}$. By \eqref{contr2}, we have 
    \begin{equation}\label{eq2}
        \inf_{\partial_+ W} H^F\ge -\alpha\tan\!\bbket{\frac{\alpha q t_-}{2}} > -\alpha\tan\!\bbbket{\frac{\alpha q (t_-+l)}{2}}.
    \end{equation}
    Then it follows from \cref{widmain} that $\wid_F(W)\le l$, which contradicts \eqref{assump1}.
\end{proof}

\subsection{\texorpdfstring{$K$}{K}-Cowaist and \texorpdfstring{$\widehat{A}$}{A hat}-Cowaist of Foliations}\label{4.3}
We now extend \cite[Ex.~7.4--7.5]{cz24} to foliations and provide examples of foliated bands with infinite vertical $\widehat{A}$-cowaist. Recall the norm $\bla{R^E_{F_0}}_\infty$ is defined in \eqref{normcur}.

\begin{defn}[{\cite[p.~195]{g96}\,\cite[Def.~1.1]{sw25}}]
    Let $(N,F_0)$ be a closed, oriented, foliated manifold equipped with a metric $g^{F_0}$ on $F_0$.
\begin{enumerate}[label={\rm{(\arabic*)}}]
    \item If $\dim N$ is even, the \emph{$K$-cowaist} of $(N,F_0)$ is defined as
    \begin{equation}\label{kcw2}
        K\text{-}\mathrm{cw}_2(N,F_0) = \bbket{\inf_{(E,\nabla^E)}\bbac{\bla{R^E_{F_0}}_\infty}}^{-1}\in (0,\infty],
    \end{equation}
    where the infimum is taken over all pairs $(E, \nabla^E)$ consisting of a Hermitian vector bundle $E$ over $N$ with a nontrivial Chern character number\footnote{That is, there exist positive integers $j_1,\dots,j_a$ such that $\binn{\ch_{j_1}(E)\cdots \ch_{j_a}(E),[N]}\ne 0$, where $\ch_{j_i}(E)$ $(i=1,\dots,a)$ denotes the $j_i$-th Chern character form of $E$. Since Chern character numbers and Chern numbers can be expressed as linear combinations of each other, we can instead require that $E$ has a nontrivial Chern number.} and a Hermitian connection $\nabla^E$ on $E$.    
    \item If $\dim N$ is even, the \emph{$\widehat{A}$-cowaist} of $(N,F_0)$ is defined as
    \begin{equation}\label{acw2}
        \widehat{A}\text{-}\mathrm{cw}_2(N,F_0) = \bbket{\inf_{(E,\nabla^E)}\bbac{\bla{R^E_{F_0}}_\infty}}^{-1}\in (0,\infty],
    \end{equation}
    where the infimum is taken over all pairs $(E, \nabla^E)$ consisting of a Hermitian vector bundle $E$ over $N$ satisfying $\binn{\widehat{A}(TN)\ch(E),[N]}\ne 0$ and a Hermitian connection $\nabla^E$ on $E$.
    \item If $\dim N$ is odd, let $\pro_N\colon N\times\bbs^1 \to N$ be the natural projection. We define
    \begin{align}
        K\text{-}\mathrm{cw}_2(N,F_0) &= K\text{-}\mathrm{cw}_2\bket{N\times\bbs^1,\pro_{N}^*F_0},\label{kcw2_odd}\\
        \widehat{A}\text{-}\mathrm{cw}_2(N,F_0) &= \widehat{A}\text{-}\mathrm{cw}_2\bket{N\times\bbs^1,\pro_N^*F_0}.\label{acw2_odd}
    \end{align}
\end{enumerate}
\end{defn}

When $F_0 = TN$, \eqref{kcw2} is the $K$-area defined in \cite[\S4]{g96} and \eqref{acw2} is the $\widehat{A}$-area defined in \cite[Def.~1.6]{cz24}. For discussions of the relation between $K$-cowaist, $\widehat{A}$-cowaist and positive scalar curvature, see \cite{g96,g17,bh23,g23,w23,bh24,cz24,shi25a}, etc. For foliated versions, refer to \cite[\S9$\frac{2}{3}$]{g96} and \cite{sw25}.

The notion of a closed, oriented, foliated manifold $(N,F_0)$ having \emph{infinite} $K$-cowaist or \emph{infinite} $\widehat{A}$-cowaist is independent of the choice of $g^{F_0}$. If $g^{F_0}$ is the restriction of a Riemannian metric $g^{TN}$ on $N$ to $F_0$, then $\bla{R^E_{F_0}}_\infty \le \bla{R^E_{TN}}_\infty$. Thus, 
\begin{equation}
    K\text{-cw}_2(N,F_0) \ge K\text{-cw}_2(N,TN)
    \quad\text{and}\quad
    \widehat{A}\text{-cw}_2(N,F_0) \ge \widehat{A}\text{-cw}_2(N,TN).
\end{equation}

\begin{exam}\label{ex0}
    Let $(N,F_0)$ be a closed, oriented, foliated manifold. If $(N,F_0)$ has infinite $\widehat{A}$-cowaist, then the foliated band $(N\times [-1,1],F)$ constructed in \eqref{product} has infinite vertical $\widehat{A}$-cowaist in the sense of \cref{ahat}. 

    A repetition of the argument in \cite[\S5$\frac{3}{8}$]{g96} (cf.\ \cite[Thm.~1.1]{w23}, \cite[Prop.~1.3]{sw25}) shows that a closed, oriented, foliated manifold with infinite $K$-cowaist also has infinite $\widehat{A}$-cowaist. Moreover, if $TN$ is spin and $(N,F_0)$ has infinite $\widehat{A}$-cowaist, then\footnote{\label{subtle}There is a subtle point when $TN$ is spin, $\dim N$ and $\rk F_0$ are odd. In this case, we have $\widehat{A}\text{-cw}_2(N\times\bbs^1,\pro_N^*F_0) \ge \widehat{A}\text{-cw}_2(N\times\bbs^1,\pro_N^*F_0\oplus T\bbs^1)$. Without considering coverings, our requirement $\widehat{A}\text{-cw}_2(N\times\bbs^1,\pro_N^*F_0) = \infty$ is slightly weaker than the requirement $\widehat{A}\text{-cw}_2(N\times\bbs^1,\pro_N^*F_0\oplus T\bbs^1) = \infty$ in \cite[Thm.~1.2]{sw25}. However, the result of \cite[Thm.~1.2]{sw25} still holds with the help of \cref{aliso2} (compare the proof of \cref{widmain1}).} $F_0$ admits no metric of positive leafwise scalar curvature by \cite[Thm.~1.2]{sw25}. 

    If $(N,F_0)$ has infinite $\widehat{A}$-cowaist and $M$ is an oriented, closed manifold satisfying $\widehat{A}(M)\ne 0$, then the foliated manifold $(N\times M,F_0\oplus TM)$ has infinite $\widehat{A}$-cowaist.
\end{exam}

\begin{exam}\label{ex1}
    Let $(N,F_0)$ be an oriented, closed, foliated manifold. Suppose $\dim N = n$ and $F_0$ is equipped with a metric $g^{F_0}$. Fix a metric $g^{T\bbs^n}$ on $\bbs^n$. We say the foliated manifold $(N,F_0)$ is \emph{compactly $\Lambda^2$-enlargeable} if for every $\epsilon>0$, there exists a \emph{finite} covering $\widehat{\pi}\colon\widehat{N}_\epsilon\to N$ and a smooth map $\Phi\colon \widehat{N}_\epsilon\to\bbs^n$ of nonzero degree, which is $(\epsilon,\Lambda^2)$-contracting along $\widehat{\pi}^*F_0 \subseteq T\widehat{N}_\epsilon$; that is, for any $X,Y\in (\widehat{\pi}^*F_0)_x$ and $x\in \widehat{N}_\epsilon$,
    \begin{equation}
        |\di\Phi (X)\wedge\di\Phi (Y)|\le \epsilon|X\wedge Y|,
    \end{equation}
    where the norms are induced by $g^{T\bbs^n}$ and the pullback metric $\widehat{\pi}^*g^{F_0}$ respectively. 
    
    Since $N$ and $\bbs^n$ are compact, this definition is independent of the choice of $g^{F_0}$ and $g^{T\bbs^n}$. For more discussions on enlargeability, see \cite{gl80a,gl83}. For foliated versions, see \cite{zhang20,sz22,swz22}.

    The \emph{$\widehat{A}$-degree} \cite[Def.~2.6]{gl80a} of a smooth map $\Psi$ between two connected, oriented, compact manifolds is defined as the number $\widehat{A}(\Psi^{-1}(x))$ for any regular value $x$ of $\Psi$ in the interior of the target (with the convention $\widehat{A}(\varnothing) = 0$). Suppose $(N,F_0)$ is a closed, oriented, foliated manifold that is compactly $\Lambda^2$-enlargeable, and let $\pro\colon N\times[-1,1]\to N$ be the natural projection. Assume $(W,F)$ is a compact, oriented, foliated band that admits a smooth map $\Psi\colon W\to N\times [-1,1]$ of nonzero $\widehat{A}$-degree satisfying
    \begin{equation}
    \Psi(\partial_\pm W) \subseteq N\times\{\pm 1\}
    \quad\text{and}\quad
    \di \Psi(F)\subseteq \pro^*F_0\oplus T[-1,1].
    \end{equation}
    Then $(W,F)$ has infinite vertical $\widehat{A}$-cowaist by the argument in \cite[Prop.~6.8]{cz24}.
\end{exam}

\section{Sub-Dirac Operators and Spin Foliations}\label{5}

Liu and Zhang \cite[\S2b]{lz} introduce the sub-Dirac operators associated with spin subbundles of even rank. These operators are Dirac operators on certain Clifford modules and play a crucial role in \cite{zhang17}. We study sub-Dirac operators associated with spin subbundles of odd rank in \cref{5.1}, and then we establish \cref{widmain} for spin foliations in \cref{5.2}.

\subsection{Sub-Dirac Operators Associated with Spin Subbundles}\label{5.1}

Let $(M,g^{TM})$ be a Riemannian manifold with the orthogonal splitting
\begin{equation}\label{osm-1}
    TM = F_0\oplus F_1 \oplus F_2,
    \quad
    g^{TM} = g^{F_0} \oplus g^{F_1}\oplus g^{F_2}.
\end{equation}
Suppose $q= \rk F_0$, $m = \rk (F_0 \oplus F_1)$ and $n = \dim M$. Assume that $F_0,F_2$ are spin and $F_1$ is orientable. Assume that $\rk F_0$ is odd and $\rk F_1, \rk F_2$ are even. Note that $F_0,F_1,F_2$ are not necessarily integrable in this subsection.

For the Levi-Civita connection $\nabla^{TM}$ of $g^{TM}$ and the orthogonal projections $p_b\colon TM\to F_b$ $(b=0,1,2)$, we define the Euclidean connection $\nabla^{F_b}$ on $F_b$ by
\begin{equation}\label{con2}
    \nabla^{F_b} = p_b\nabla^{TM}p_b.
\end{equation}

For $b \in\{0,2\}$, let $S(F_b)$ be a Hermitian bundle of spinors associated with $(F_b,g^{F_b})$, which is equipped with the Clifford action $c\colon F_b\to \en{S(F_b)}$ and the Hermitian connection $\nabla^{S(F_b)}$ induced by the connection $\nabla^{F_b}$ defined in \eqref{con2}. Let $\Lambda^*(F_1)$ be the complexified exterior algebra bundle of the dual bundle $F_1^*$ equipped with the induced Hermitian metric and the Hermitian connection $\nabla^{\Lambda^*(F_1)}$ induced by $\nabla^{F_1}$. For any $X\in F_1$, we define $c,\widehat{c} \colon F_1 \to \en{\Lambda^*(F_1)}$ by
    \begin{equation}\label{clif1}
        c(X)= (X^*\wedge) - \,i_{X}, \quad \widehat{c}(X)= (X^*\wedge) + \,i_{X}, 
    \end{equation}
    where $X^*\wedge$ and $i_{X}$ denote the exterior multiplication by $X^*$ and the interior multiplication by $X$ on $\Lambda^*(F_1)$, respectively.

We fix orientations on $F_0, F_1$ and $F_2$. Let $\{e_j\}_{j=1}^n$ be a positively oriented orthonormal basis of $TM$ such that $\{e_j\}_{j=1}^q,\{e_j\}_{j=q+1}^m,\{e_j\}_{j=m+1}^n$ span $F_0,F_1,F_2$ respectively. Since $\rk F_1$ and $\rk F_2$ are even, there are $\bbz_2$-gradings $\Lambda^*(F_1) = \Lambda^*_+(F_1)\oplus \Lambda^*_-(F_1)$ and $S(F_2) = S_+(F_2) \oplus S_-(F_2)$ defined by
\begin{align}\label{vol1}
    \tau_1 &= \bket{\sqrt{-1}\,}^{\frac{m-q}{2}}\,c(e_{q+1})\, \cdots \,c(e_{m}),\\
    \tau_2 &= \bket{\sqrt{-1}\,}^{\frac{n-m}{2}}\,c(e_{m+1})\, \cdots \,c(e_{n}),
\end{align}
which anticommute with the Clifford actions $c(\cdot)$. Since $\rk F_0$ is odd, there are two choices for $S(F_0)$. As in \eqref{vol}, we choose the one that satisfies
\begin{equation}\label{choice1}
    \bket{\sqrt{-1}\,}^{\frac{q+1}{2}}\,c(e_{1})\, \cdots \,c(e_{q}) = -\id
\end{equation}
and write $S^1(F_0)$ for this $S(F_0)$ for clarity.

Following \cite[(2.11)]{lz} (cf.\ \cite[(1.55)]{zhang17}), define the Hermitian vector bundle
\begin{equation}\label{mu0}
    \mu=S^1(F_0) \otimes \Lambda^*(F_1)\otimes S(F_2).
\end{equation}
For  $X_b\in F_b$ $(b=0, 1, 2)$, we define the Clifford action $c^{\,\mu} \colon TM\to \en{\mu}$ by
\begin{align}
        c^{\,\mu}(X_0) &= c(X_0) \otimes \tau_1 \otimes \tau_2,\label{b1}\\
        c^{\,\mu}(X_1) &= \id \otimes \,c(X_1) \otimes \tau_2,\\
        c^{\,\mu}(X_2) &= \id \otimes \id \otimes \,c(X_2),\label{b3}
\end{align}
which satisfies \eqref{a1}. Let $\overline{\nabla}^{\mu}$ be the induced tensor product Hermitian connection on $\mu$. Let $A$ be the $\en{TM}$-valued $1$-form on $M$ defined by
\begin{equation}
    A = \nabla^{TM} - \nabla^{F_0} - \nabla^{F_1} - \nabla^{F_2}.
\end{equation}
We define a Hermitian connection on $\mu$ by (see \cite[(2.17)]{lz}, \cite[(1.60)]{zhang17})
\begin{equation}\label{modcon}
    \nabla^{\mu}_X = \overline{\nabla}^{\mu}_X + \frac{1}{4}\sum_{i,j=1}^n\inn{A(X)e_i,e_j}\,c^{\,\mu}(e_i)\,c^{\,\mu}(e_j)
\end{equation}
for any $X\in TM$. Following \cite{lz}, we define the \emph{sub-Dirac operator}
\begin{equation}\label{subdir}
    D^\mu = \sum_{j=1}^{n}c^{\,\mu}(e_j) \nabla^\mu_{e_j}.
\end{equation}

\begin{rmk}[cf.\ {\cite[Thm.~2.3]{lz}, \cite[Rem.~1.8]{zhang17}}]\label{rem3}
    If $F_1$ is spin in addition, then $TM = F_0 \oplus F_1 \oplus F_2$ is also spin. We have the following identifications of $\bbz_2$-graded vector bundles (see \cite[Prop.~10.16]{lm89}, \cite[(1.63)]{zhang17})
    \begin{equation}\label{iso}
        \Lambda^*_+(F_1) \oplus \Lambda^*_-(F_1) = \bket{S_+(F_1) \otimes S(F_1)^*} \oplus \bket{S_-(F_1) \otimes S(F_1)^*}.
    \end{equation}
    Then \eqref{mu0} can be identified with the twisted Clifford module
    \begin{equation}
        \mu' \coloneqq S^1(TM) \otimes S(F_1)^*.
    \end{equation}
    The connection $\nabla^\mu$ is identified with the tensor product connection on $\mu'$, and the Clifford action $c^{\,\mu}$ is identified with the Clifford action on $\mu'$ given by $c\otimes \id$. Hence, the sub-Dirac operator $D^\mu$ is identified with the twisted Dirac operator on $\mu'$. 
\end{rmk}

We now return to the general case. Although $F_1$ might not be spin, it is spin locally after choosing a local trivialization. Since the local formula \eqref{a3} holds for ${\mu'}$, it also holds for $\mu$. Hence, $(\mu,\nabla^\mu,c^{\,\mu})$ is also a Clifford module on $(M,g^{TM})$. Moreover, the curvature $R^{\mu} = (\nabla^{\mu})^2$ satisfies (cf.\ \cite[(1.67)]{zhang17})
\begin{equation}\label{cur1}
    R^\mu(e_i,e_j) = -\frac{1}{4} \sum_{k,l=1}^{n}R^{TM}_{ijkl}\,c^{\,\mu}(e_k) \,c^{\,\mu}(e_l) + \frac{1}{4} \sum_{k,l=q+1}^{m}R^{F_1}_{ijkl}\,\widehat{c}^{\,\mu}(e_k)\,\widehat{c}^{\,\mu}(e_l),
\end{equation}
where $\widehat{c}^{\,\mu}(e_k) \coloneqq \id \otimes \,\widehat{c}(e_k)\otimes \id$ $(q+1 \le k\le m)$.

In what follows, suppose $\partial M \neq \varnothing$. We further assume $F_0$ is transverse to $\partial M$ and  $\ket{F_1 \oplus F_2}\big|_{\partial M} \subseteq T\partial M$. Let $\nu$ be the inward unit normal on $\partial M$. It follows that $\nu \in \Gamma(F_0|_{\partial M})$ and $T\partial M = (F_0\cap T\partial M)\oplus F_1|_{\partial M} \oplus F_2|_{\partial M}$, where $ F_0\cap T\partial M$ is a subbundle of even rank. As in \eqref{oddevenclif}, there is an identification
\begin{equation}\label{identi}
    S^1(F_0)|_{\partial M} = S(F_0\cap T\partial M).
\end{equation}
The Clifford action on $S^1(F_0)|_{\partial M}$ is given by $c^\partial(X) \coloneqq c(X)\,c(\nu)$ for $X\in F_0\cap T\partial M$; cf.\ \eqref{bdclif}. Let $\tau^\partial$ be the canonical $\bbz_2$-grading on $S(F_0\cap T\partial M)$. By \eqref{choice1}, we have 
\begin{equation}\label{identi2}
    c(\nu) / \sqrt{-1} = \bket{\sqrt{-1}\,}^{\frac{q-1}{2}} c^\partial (e_2)\, \cdots \,c^\partial(e_q) = \tau^\partial
\end{equation}
for any positively oriented orthonormal basis $\{\nu,e_2,\dots,e_q\}$ of $F_0$ such that $\{e_j\}_{j=2}^q$ spans $F_0\cap T\partial M$.

The Clifford module $\mu_\partial = \mu|_{\partial M}$ defined by \eqref{bdclif} has Clifford action $c^{\,\mu_\partial}(X) = c^{\,\mu}(X)\,c^{\,\mu}(\nu)$ for $X\in T\partial M$. On the other hand, the Clifford module
\begin{equation}
    \widetilde{\mu} = S(F_0\cap T\partial M) \,\widehat{\otimes}\, \Lambda^*(F_1) \,\widehat{\otimes}\, S(F_2)
\end{equation} 
has the following $\bbz_2$-grading by \eqref{b1} and \eqref{identi2}:
\begin{equation}\label{grading}
    \tau^\partial \otimes \tau_1 \otimes \tau_2 = c^{\,\mu}(\nu)/\sqrt{-1}.
\end{equation} 
For $X_0\in F_0\cap T\partial M$ and $X_b\in F_b$ $(b=1,2)$, the Clifford action on $\widetilde{\mu}$ is given by
\begin{align}
    c^{\,\widetilde{\mu}}(X_0) &= c^\partial(X_0) \otimes \id \otimes \id,\label{b4}\\
    c^{\,\widetilde{\mu}}(X_1) &= \tau^\partial \otimes c(X_1) \otimes \id,\\
    c^{\,\widetilde{\mu}}(X_2) &= \tau^\partial \otimes \tau_1 \otimes c(X_2).\label{b6}
\end{align}

The following lemma identifies the Clifford actions on $\mu_\partial = \mu|_{\partial M}$ and $\widetilde{\mu}$.
\begin{lem}\label{isomod}
    We define $A_1\in\Gamma(\en{\Lambda^*(F_1)})$ and $A_2\in \Gamma(\en{S(F_2)})$ by 
    \begin{equation}
        A_b = \frac{1+\tau_b}{2} - \sqrt{-1}\cdot\frac{1-\tau_b}{2}\qquad (b =1,2).
    \end{equation}
    Under the identification \eqref{identi}, define $\Phi = \id \otimes\, A_1 \otimes A_2\colon \mu_\partial \to \widetilde{\mu}$. Then $\Phi$ is an isomorphism of Clifford modules on $\partial M$; that is, for all $X\in T\partial M$,
    \begin{equation}\label{iso1}
        \Phi\circ c^{\,\mu_\partial}(X) = c^{\,\widetilde{\mu}}(X) \circ \Phi.
    \end{equation}
\end{lem}
\begin{proof}
    Observe that $A_b$ is an isomorphism with inverse
    \begin{equation}
        A_b^{-1} = \frac{1+\tau_b}{2} + \sqrt{-1}\cdot\frac{1-\tau_b}{2}.
    \end{equation}
    So $\Phi$ is an isomorphism. It suffices to check \eqref{iso1} for $X_0\in F_0\cap T\partial M$ and $X_b\in F_b|_{\partial M}$ $(b=1,2)$. In view of \eqref{identi}, we have\begin{equation}
        \Phi\circ c^{\,\mu_\partial}(X_0) = c^{\,\widetilde{\mu}}(X_0) \circ \Phi.
    \end{equation}
    For $b = 1,2$, since $\tau_b$ anticommutes with $c(X_b)$, we have
    \begin{align}\label{acommute}
        \begin{split}
            \sqrt{-1} \, A_b \, c(X_b) \, \tau_b &= \sqrt{-1} \, \bbbket{\frac{1+\tau_b}{2} - \sqrt{-1}\cdot\frac{1-\tau_b}{2}} \, c(X_b) \, \tau_b\\
            & = c(X_b) \, \bbbket{\sqrt{-1}\cdot \frac{1-\tau_b}{2} + \frac{1+\tau_b}{2}} \, \tau_b = c(X_b) \, A_b.
        \end{split}
    \end{align}
    By \eqref{identi2}, \eqref{acommute} and $A_1 \, \tau_1 = \tau_1 \, A_1$, we have
    \begin{align}
        \begin{split}
            \Phi\circ c^{\,\mu_\partial}(X_1) 
            &= 
            (\id\otimes \,A_1 \otimes A_2) \, 
            \ket{\id \otimes \,c(X_1) \otimes \tau_2} \, 
            \ket{c(\nu) \otimes \tau_1 \otimes \tau_2}\\
            &= \sqrt{-1}\,\tau^\partial \otimes A_1\,c(X_1)\,\tau_1 \otimes A_2\\
            &= \tau^\partial \otimes c(X_1)\,A_1 \otimes A_2 = c^{\,\widetilde{\mu}}(X_1) \circ \Phi,
        \end{split}
    \end{align}
    \begin{align}
        \begin{split}
            \Phi\circ c^{\,\mu_\partial}(X_2) &= 
            (\id\otimes \,A_1 \otimes A_2) \, 
            \ket{\id \otimes \id \otimes \,c(X_2)} \, 
            \ket{c(\nu) \otimes \tau_1 \otimes \tau_2}\\
            &= \sqrt{-1}\,\tau^\partial \otimes  A_1 \, \tau_1 \otimes A_2\,c(X_2)\,\tau_2\\
            &= \tau^\partial \otimes \tau_1 \, A_1 \otimes c(X_2) \, A_2 = c^{\,\widetilde{\mu}}(X_2) \circ \Phi.
        \end{split}
    \end{align}
    We conclude the proof.
\end{proof}

\subsection{Leafwise Band Width Estimate for Spin Foliations}\label{5.2}
When $TW$ is nonspin but $F$ is spin, \cref{widmain} still holds for the following modified definition of infinite vertical $\widehat{A}$-cowaist.

\begin{defn}\label{ahat1}
    Let $(W,F)$ be a compact, oriented, foliated band with $m\coloneqq \dim W$ and $q\coloneqq \rk F$. We define the auxiliary odd-dimensional band $W^\lozenge$ by
    \begin{equation}\label{band0}
    W^\lozenge = 
    \begin{cases}
        W, &\text{if } q \text{ is odd and } m-q \text{ is even},\\
        W\times \bbs^1 \times \bbs^1, &\text{if } q \text{ is even and } m-q \text{ is odd},\\
        W\times \bbs^1, &\text{otherwise}.
    \end{cases}
    \end{equation}
    We say $(W,F)$ has infinite \emph{modified} vertical $\widehat{A}$-cowaist if $\ket{W^\lozenge,\pro_W^*F}$ has infinite vertical $\widehat{A}$-cowaist in the sense of \cref{ahat}\,(1), where $\pro_W\colon W^\lozenge\to W$.
\end{defn}
\begin{thm}\label{widmain1}
    \cref{widmain,widmain2} still hold when $W$ is nonspin but $W$ is oriented, $F$ is spin and $(W,F)$ has infinite modified vertical $\widehat{A}$-cowaist. 
\end{thm}
\begin{proof}
We have $\rk F < \dim W$ since $F$ is spin but $W$ is nonspin. Let $W^\lozenge / W$ be the torus such that $W^\lozenge = W \times (W^\lozenge / W)$.  Define $\calw^\lozenge = \calw \times (W^\lozenge / W)$, and similarly define $\calw_r^\lozenge,\ww^\lozenge$. The fibration maps \eqref{cm} and \eqref{dcm} extend to smooth maps
\begin{equation}
    \bar{\pi}\colon \calw^\lozenge \to W^\lozenge
    \quad\text{and}\quad
    \widetilde{\pi}_r\colon \ww^\lozenge \to W^\lozenge,
\end{equation} which act as identities on the $\bbs^1$ factors.

Suppose to the contrary that $\wid_F(W) > t_+ - t_-$. For simplicity, we will omit the pullback maps in the pullback bundles. For example, write $F$ for $\pro_W^*F$. Note that $\wid_{F}(W^\lozenge) = \wid_F(W)$. Repeat Step 1 in the proof of \cref{widmain}. The smooth function \eqref{psi} extends naturally to $\psi\colon \ww^\lozenge\to\bbr$.

{\bf Step 1:} Recall the splitting \eqref{bdreq}. We define subbundles $Q,\gp\subset TW^\lozenge$ by
\begin{equation}\label{tan}
    Q = 
    \begin{cases}
        F, & \text{if } q \text{ is odd},\\
        F \oplus T\bbs^1, & \text{if } q \text{ is even},
    \end{cases}
    \;
    \gp = 
    \begin{cases}
        F^\perp, & \text{if } m-q \text{ is even},\\
        F^\perp\oplus T\bbs^1, & \text{if } m-q \text{ is odd}.
    \end{cases}
\end{equation}
Then $\rk Q$ is odd and $\rk Q^\perp$ is even. Meanwhile, we have
\begin{equation}\label{bdsplit1}
    TW^\lozenge = Q\oplus \gp
    \quad\text{and}\quad
    \gp|_{\partial W^\lozenge}\subset T\partial W^\lozenge.
\end{equation}
Recall the splitting \eqref{spmet2}. We define subbundles $\hgz,\hg\subset T\ww^\lozenge$ by
\begin{equation}\label{q2}
    \hgz = 
    \begin{cases}
        \hfz, & \text{if } q \text{ is odd},\\
        \hfz \oplus T\bbs^1, & \text{if } q \text{ is even},
    \end{cases}
    \;
    \hg = 
    \begin{cases}
        \hf, & \text{if } m-q \text{ is even},\\
        \hf \oplus T\bbs^1, & \text{if } m-q \text{ is odd}.
    \end{cases}
\end{equation}
Then $\rk \hgz$ is odd and $\rk\hg$ is even. Meanwhile, there are isomorphisms $\hgz = \widetilde{\pi}_r^*Q$ and $\hg=\widetilde{\pi}_r^*Q^\perp$. Unless otherwise specified, the pullback bundles are equipped with the pullback metrics. Equip $T\bbs^1$ with the flat metric $\di \theta^2$, where $\theta \in \bbr/2\pi\bbz$. 

For $\bv>0$, let $g^{T\ww^\lozenge}_\bv$ be the metric on $\ww^\lozenge$ defined by the orthogonal splitting
\begin{equation}\label{met31}
    T\ww^\lozenge = \hgz \oplus \hg \oplus \hff, 
    \quad 
    g^{T\ww^\lozenge}_\bv = g^{\hgz}_\beta \oplus \frac{g^{\hg}}{\varepsilon^2} \oplus g^{\hff}
\end{equation}
induced from \eqref{spmet2}, where
\begin{equation}
    g^{\hgz}_\beta \coloneqq \begin{cases}
        \beta^2 g^{\hfz}, & \text{if } q \text{ is odd},\\
        \beta^2 g^{\hfz}\oplus \di\theta^2, & \text{if } q \text{ is even},
    \end{cases}
    \;
    g^{\hg} \coloneqq \begin{cases}
        g^{\hf}, & \text{if } m-q \text{ is even},\\
        g^{\hf}\oplus \di\theta^2, & \text{if } m-q \text{ is odd}.
    \end{cases}
\end{equation} 
Since $\hgz$ and $\hff$ are spin and $\hg$ is orientable, we obtain a Clifford module
\begin{equation}\label{sub2}
    \mu_\bv = S^1_\bv(\hgz) \otimes \Lambda_\bv^*(\hg) \otimes S(\hff)
\end{equation}
on $\bket{\ww^\lozenge,g^{T\ww^\lozenge}_\bv}$ as in \eqref{mu0}. By \eqref{bdsplit1} and \cref{isomod}, there is an identification 
\begin{equation}\label{clifsub2}
    \mu_\bv|_{\partial_- \ww^\lozenge} = S_\bv(\hgz\cap T\partial_-\ww^\lozenge) \,\widehat{\otimes}\, \Lambda_\bv^*(\hg) \,\widehat{\otimes}\, S(\hff).
\end{equation}

We rearrange \eqref{met31} on $\calw^\lozenge_r$ as follows. Define subbundles $\gb,\ggb\subset T\calw^\lozenge$ by\footnote{That is, we keep $\hg|_{\calw^\lozenge_r}$ unchanged and transfer any $T\bbs^1$ factor from $\hgz|_{\calw^\lozenge_r}$ to $\hff|_{\calw^\lozenge_r}$.}
\begin{equation}\label{q1}
    \gb = 
    \begin{cases}
        \f, & \text{if } m-q \text{ is even},\\
        \f \oplus T\bbs^1, & \text{if } m-q \text{ is odd},
    \end{cases}
    \;
    \ggb = 
    \begin{cases}
        \ff, & \text{if } q \text{ is odd},\\
        \ff \oplus T\bbs^1, & \text{if } q \text{ is even}.
    \end{cases}
\end{equation}
 For $\bv>0$, let $g^{T\calw^\lozenge}_\bv$ be the metric on $\calw^\lozenge$ defined by the orthogonal splitting
\begin{equation}\label{met3}
    T\calw^\lozenge = \calf \oplus \gb \oplus \ggb,
    \quad
    g^{T\calw^\lozenge}_\bv = \beta^2 g^\calf \oplus \frac{g^{\gb}}{\varepsilon^2} \oplus g^{\ggb},
\end{equation}
which is induced from \eqref{splittingmetric}. Then \eqref{met3} coincides with \eqref{met31} on $\calw_r^\lozenge$. 

Since $(W,F)$ has infinite modified vertical $\widehat{A}$-cowaist, the compact, oriented, foliated band $(W^\lozenge,F)$ has infinite vertical $\widehat{A}$-cowaist in the sense of \cref{ahat}\,(1). For every $\varpi>0$, there exists a Hermitian vector bundle $E_\varpi$ over $W^\lozenge$ such that
    \begin{equation}\label{ind21}
        \babs{R^{E_\varpi}_{F}} < \varpi \text{ on } W^\lozenge
        \quad \text{and} \quad
        \binn{\widehat{A}(T\partial_-W^\lozenge)\ch(E_\varpi), [\partial_- W^\lozenge]} \ne 0.
    \end{equation}
Let $\phi(\gp)$ be a polynomial in the exterior and symmetric powers of $\gp$ with integral coefficients, and use the same notation for its complexification. Take
    \begin{equation}\label{evr}
        E = \phi(\gp) \otimes E_\varpi
    \end{equation}
in \eqref{xixi}. Let $\nabla^{\gb,\bv}$ be the Euclidean connection on $\gb$ defined as in \eqref{defcon} with respect to the splitting \eqref{met3}. By \eqref{tan} and \eqref{q1}, the bundle $\bar{\pi}^*\gp$ is isomorphic to $\gb$. Hence, $\bar{\pi}^*\phi(\gp)$ is isomorphic to $\phi(\gb)$, which carries a Hermitian metric induced by $g^\gb$ and a Hermitian connection induced by $\nabla^{\gb,\bv}$. Let $R^{\,\phi(\gb),\bv}$ be the curvature of this connection on $\phi(\gb)$. As in \eqref{xibv}, we obtain a $\bbz_2$-graded Hermitian vector bundle $\xi_\varpi = \xi^+_\varpi \oplus \xi^-_\varpi$ over $\ww^\lozenge$ such that
\begin{equation}\label{xisub}
    \xi^\pm_{\varpi} = \big(S_\pm(\ff)^* \otimes \phi(\gb) \otimes \bar{\pi}^*E_\varpi\big) \oplus \caly_\varpi\qquad\text{on $\calw_r^\lozenge$}.
\end{equation}

We can define a Clifford module as in \eqref{s1} and a self-adjoint Fredholm operator $\calb_{h_\delta,\psi,\beta^{-1},s}$ as in \eqref{bts}. The above constructions also applies to $(\wwin)^\lozenge \coloneqq \wwin \times (W^\lozenge/W)$. By \eqref{clifsub2}, \cref{band2} and the Atiyah-Singer Index Theorem, the formula \eqref{ind11} becomes (cf.\ \cite[(2.59)]{zhang17})
\begin{equation}\label{ind22}
    \ind \calb_{h_\delta,\psi,\beta^{-1},s}  = 
    2^{[\frac{m-q+1}{2}]}\big\langle\widehat{A}(Q\cap T\partial_-W^\lozenge)\widehat{L}(\gp)\ch \ket{\phi(\gp)}\ch (E_\varpi),[\partial_-W^\lozenge] \big\rangle.
\end{equation}
Next, we examine the asymptotic estimates in the proof of \cref{widmain}.

{\bf Step 2:} By \cref{aliso2}, the foliated manifold $(\calu,\calf)$ is also almost isometric with respect to \eqref{met3}. Repeat Step 3 in the proof of \cref{widmain}. Suppose $m' = \dim W^\lozenge$ and $n'=\dim \calw^\lozenge$. Recall $q=\rk \calf$. Let $\{e_j\}_{j=1}^q, \{e_j\}_{j=q+1}^{m'}$ and $\{e_j\}_{j=m'+1}^{n'}$ be orthonormal bases of $\calf, \gb$ and $\ggb$ respectively, which are associated with \eqref{met3}. By \eqref{cur31}, \eqref{cur43}, \eqref{ind21} and \eqref{xisub}, we have on $\calw_r^\lozenge$ that (compare \eqref{cur44})
    \begin{align}\label{cur446}
        \begin{split}
            R^{\xi_\varpi}(e_i,e_j)|_{S(\ff)^*\otimes \phi(\gb)\otimes\bar{\pi}^*E_\varpi} &= R^{S_\pm(\ff)^*}(e_i,e_j) \otimes \id \otimes \id \\
            &\qquad\quad + \id\otimes R^{\phi(\gb),\bv}(e_i,e_j) \otimes \id\\
            &\qquad\quad + \id \otimes \id \otimes  R^{\bar{\pi}^*E_\varpi}(e_i,e_j)\\
            &=
            \begin{cases}
                O_r(\varepsilon^2) + O\bket{\frac{\varpi}{\beta^2}}, & \text{if } 1\le i,j \le q,\\
               O_{r,\varpi}\bket{\frac{1}{\beta}}, & \text{otherwise}.
            \end{cases}
        \end{split}
    \end{align}
Consider the partial sum associated with the subbundle $\calf\subset T\calw^\lozenge$ of rank $q$. By \eqref{partterm}, \eqref{cur1} and \eqref{sub2}, for $a\in\{q,n'\}$, we have on $\calw_r^\lozenge$ that (compare \eqref{cur13})
    \begin{align}\label{add2}
        \begin{split}
            \calr^\bv_a|_{\mu_\bv\otimes\xi_\varpi}
            =&-\frac{1}{8}\sum_{i,j=1}^{a}\sum_{k,l=1}^{n'} \!R^{T\calw_r^\lozenge,\bv}_{ijkl}\,c^{\,\mu_\bv}(e_i)\,c^{\,\mu_\bv}(e_j)\,c^{\,\mu_\bv}(e_k)\,c^{\,\mu_\bv}(e_l)\otimes\id\\
            &+\frac{1}{8}\sum_{i,j=1}^{a}\sum_{k,l=q+1}^{m'}\! R^{\gb,\bv}_{ijkl}\,c^{\,\mu_\bv}(e_i)\,c^{\,\mu_\bv}(e_j)\, \widehat{c}^{\,\mu_\bv}(e_k)\, \widehat{c}^{\,\mu_\bv}(e_l) \otimes \id\\
            &\qquad+ \frac{1}{2}\sum_{i,j=1}^{a} c^{\,\mu_\bv}(e_i)\,c^{\,\mu_\bv}(e_j) \otimes R^{\,\xi_\varpi}(e_i,e_j).
        \end{split}
    \end{align}
    Apply \cref{partcur1} and \eqref{cur31} to the almost isometric foliation $(\calw^\lozenge,\calf)$ and the orthogonal splitting \eqref{met3}, the formula \eqref{cur12} remains valid. 
    
    Again, applying \cref{part3} to the almost isometric foliation $(\calw^\lozenge,\calf)$ and the orthogonal splitting \eqref{met3}, the formula \eqref{est9} remains valid:
    \begin{align}
        \begin{split}
            \int_{\calw_r^\lozenge} |D_\bv u_1|^2
            &\ge
            \int_{\calw_r^\lozenge} \bbbac{\frac{q}{q-1}\cdot \frac{k^{\calf}}{4\beta^2} + O_{r,\varpi}\bbbket{\frac{1}{\beta} + \frac{\varepsilon}{\beta^2}} + O\bbbket{\frac{\varpi}{\beta^2}}\!}|u_1|^2\\
            &\qquad\qquad\quad+ \int_{\partial\ww^\lozenge\,\cap\,\calw_r^\lozenge} \bbbac{\frac{qH^{\calf}}{2\beta} + O_r(1)}|u_1|^2,
        \end{split}
    \end{align}
    where $u_1$ is defined as in \eqref{splitvarphi1}. Following the proof of \cref{widmain}, there exist sufficiently large $r>0$ and sufficiently small $\varpi,\bv,\delta>0$ such that
    \begin{equation}\label{b22}
        \ind \calb_{h_\delta,\psi,\beta^{-1},s} = 0.
    \end{equation}
    
    Finally, it follows from \eqref{ind22} and \eqref{b22} that
    \begin{equation}\label{ind60}
        \big\langle\widehat{A}(Q\cap T\partial_-W^\lozenge)\widehat{L}(Q^\perp)\ch \ket{\phi(Q^\perp)}\ch (E_\varpi),[\partial_-W^\lozenge] \big\rangle = 0
    \end{equation}
    for any polynomial $\phi$ with integral coefficients. 
    Any rational Pontryagin class $p(Q^\perp)$ of $Q^\perp$ can be expressed as a rational linear combination of classes of the form $\widehat{L}(Q^\perp)\ch \ket{\phi(Q^\perp)}$. So we have
    \begin{equation}\label{ind61}
        \binn{\widehat{A}(Q\cap T\partial_-W^\lozenge)\,p(Q^\perp)\ch (E_\varpi),[\partial_-W^\lozenge]} = 0;
    \end{equation}
    cf.\ \cite[(2.61)]{zhang17}. By \eqref{bdsplit1}, we have $T\partial_-W^\lozenge = (Q\cap T\partial_-W^\lozenge) \oplus Q^\perp|_{\partial_-W^\lozenge}$. So $\widehat{A}(T\partial_- W^\lozenge) = \widehat{A}(Q\cap T\partial_-W^\lozenge)\widehat{A}(Q^\perp)$. Taking $p(Q^\perp) = \widehat{A}(Q^\perp)$ in \eqref{ind61}, we obtain
    \begin{equation}
        \binn{\widehat{A}(T\partial_- W^\lozenge)\ch (E_\varpi),[\partial_-W^\lozenge]}=0,
    \end{equation}
    which contradicts \eqref{ind21}. So we conclude $\wid_F(W)\le t_+ - t_-$. The analogue version of \cref{widmain2} follows as a corollary.
\end{proof}

\appendix
\section{Alternative Proof of Proposition \ref{key2}}\label{leafwisedis}
\begin{proof}[Proof of \cref{key2}]
    Without loss of generality, we assume $K_0,K_1$ are disjoint and nonempty. For $0<\beta<\beta'\le 1$, we have $g_{\beta'} \le g_\beta$, and then
    \begin{equation}\label{d41}
        \dist_{\beta'}(K_0,K_1) \le \dist_\beta(K_0,K_1)\le \dist_F(K_0,K_1).
    \end{equation}

    Suppose to the contrary that there exists a finite number $C>0$ such that
    \begin{equation}\label{d61}
        \lim_{\beta\to 0^+} \dist_\beta(K_0,K_1) \le C < \dist_F(K_0,K_1).
    \end{equation}
    Since $g\le g_\beta$ for $0<\beta\le 1$, the Riemannian metrics $g_\beta$ are complete since $g$ is. Let $\gamma_\beta\colon [0,1]\to M$ be a minimizing geodesic in $(M,g_\beta)$ realizing $\dist_\beta(K_0,K_1)$.

    {\bf Step 1:} Since $|\dot{\gamma}_\beta|_{\beta}$ is constant, it follows from \eqref{d41}--\eqref{d61} that, for $t_0, t_1 \in [0,1]$,
    \begin{equation}\label{d11}
            \dist\bket{\gamma_\beta(t_0),\gamma_\beta(t_1)}
            \le
            \dist_{\beta}\bket{\gamma_\beta(t_0),\gamma_\beta(t_1)}            
            \le
            C|t_0-t_1|.
    \end{equation}
    So $\{\gamma_\beta\}$ is equicontinuous and uniformly bounded in $(M,g)$. By Arzelà--Ascoli Theorem, there exists a subsequence $\{\gamma_{\beta_j}\}$ converging uniformly to a continuous curve
    \begin{equation}
        \zeta\colon[0,1]\to M.
    \end{equation}
    Since $K_0$ and $K_1$ are compact, we have $\zeta(0)\in K_0$ and $\zeta(1)\in K_1$. Passing to a subsequence, we can assume without loss of generality that
    \begin{equation}\label{d12}
        \dist\bket{\gamma_{\beta_j}(t),\zeta(t)}< j^{-1}
    \end{equation}
    holds for all $t\in[0,1]$ and all integers $j>0$. For each $t\in[0,1]$, let $\eta_{j,t}: [0,1]\to M$ be the minimizing geodesic in $(M,g)$ realizing $\dist\bket{\gamma_{\beta_j}(t),\zeta(t)}$. 
    
    Suppose $q=\rk F$ and $n=\dim M$. Since $\{\gamma_{\beta_j}\}$ is uniformly bounded in $(M,g)$, there exists a constant $\alpha\in(0,1)$ such that, for any $t_0,t_1\in[0,1]$ with $|t_0-t_1|<\alpha$, there exists a foliation chart 
    \begin{equation}\label{fchart}
        \varphi  \colon \scru \to \varphi(\scru) \subset \bbr^q \times \bbr^{n-q}
    \end{equation}
    containing the images $\gamma_{\beta_j}([t_0,t_1])$, $\eta_{j,t}([0,1])$ and $\zeta([t_0,t_1])$ for all $j>\alpha^{-1}$. Here $\varphi$ is a diffeomorphism on the open subset $\scru$, $\bbr^{n-q}$ is the direction transverse to leaves, and we can assume that $\overline{\scru}\subseteq M$ is compact. Define 
    \begin{equation}
        \phi = \pro_2 \circ \varphi\colon \scru \to \bbr^{n-q}.
    \end{equation}

    Let $g_0$ be the standard flat metric on $\bbr^{n-q}$ with distance function $\dist_{g_0}$. Given $j>\alpha^{-1}$, $|t_0-t_1|<\alpha$ and \eqref{fchart}, there exists $C_1>0$ depending on $\scru$ such that
    \begin{equation}\label{d13}
        \phi^*g_0 \le C^2_1 \cdot g \quad\text{and}\quad \phi^*g_0 \le C^2_1\beta^2_j \cdot g_{\beta_j} \qquad \text{on } \scru,
    \end{equation}
    where the second formula follows from $\di\phi|_F = 0$.
    By \eqref{d12}, \eqref{d13} and the last inequality in \eqref{d11}, we have (cf.\ \eqref{dis621}--\eqref{dis62})
    \begin{align}\label{dis14}
        \begin{split}
            &\dist_{g_0}\bket{\phi(\zeta(t_0)),\phi(\zeta(t_1))}\\
            \le\; &
            \sum_{b=0}^1 \dist_{g_0}\bket{\phi(\zeta(t_b)),\phi\ket{\gamma_{\beta_j}(t_b)}}+ \dist_{g_0}\bket{\phi\ket{\gamma_{\beta_j}(t_0)}, \phi\ket{\gamma_{\beta_j}(t_1)}}\\
            \le\; &
            C_1\sum_{b=0}^1 \dist\bket{\zeta(t_b),\gamma_{\beta_j}(t_b)}
            + C_1\beta_j\cdot \dist_{\beta_j}\bket{\gamma_{\beta_j}(t_0), \gamma_{\beta_j}(t_1)}\\
            \le\; & 
            2C_1 j^{-1} + C_1\beta_j C|t_0-t_1|.
        \end{split}
    \end{align}

    Taking $j\to \infty$ in \eqref{dis14}, we obtain $\phi(\zeta(t_0)) = \phi(\zeta(t_1))$. So the continuous curve $\zeta$ is contained in a single leaf $\scrl$, and then $\dist_F(K_0,K_1)<\infty$.

    \textbf{Step 2:} 
    Recall the number $\alpha>0$ chosen in Step 1. Given $t_0,t_1$ such that $|t_0-t_1|<\alpha$, choose a foliation chart $\scru$ as in \eqref{fchart} such that $\overline{\scru}$ is compact. Let $\scrl_0$ be the connected component of $\scru\cap\scrl$ containing $\zeta([t_0,t_1])$. 
    
    The metric $g = g^F \oplus g^{F^\perp}$ induces a normal exponential map $\exp^\perp: F^\perp|_{\scrl_0} \to M$, which restricts to a diffeomorphism from a neighborhood of the zero section in $F^\perp|_{\scrl_0}$ to a tubular neighborhood of $\scrl_0$ in $M$. Let $\Phi$ be the natural projection from this tubular neighborhood to $\scrl_0$. 
    
    We have $\di \Phi = \id$ on $F|_{\scrl_0}$ and $\di \Phi = 0$ on $F^\perp|_{\scrl_0}$. By continuity and the fact $g\le g_{\beta_j}$, for every integer $k>0$, there exists a sufficiently small $\delta_k > 0$ such that the $\delta_k$-neighborhood of the compact subset $\zeta{[t_0,t_1]}$ is contained in $\scru$ and
    \begin{equation}\label{lip2}
        \Phi^*\bket{g^F|_{\scrl_0}} \le (1+k^{-1})^2 \, g \le (1+k^{-1})^2 \, g_{\beta_j}
    \end{equation}
    on the $\delta_k$-neighborhood of $\zeta{[t_0,t_1]}$. Assume $\delta_k\to 0$ as $k\to \infty$. By \eqref{d11}, \eqref{d12}, \eqref{lip2} and $\zeta(t_0) = \Phi(\zeta(t_0))$, for sufficiently large $k>0$ and $j =[\delta_k^{-1}] + 1$, we have 
    \begin{align}\label{estidis1}
        \begin{split}
            &\dist_F\bket{\zeta(t_0),\zeta(t_1)}\\
            \le\;&
            \sum_{b=0}^{1} \dist_F\bket{\Phi(\zeta(t_b)),\Phi\ket{\gamma_{\beta_j}(t_b)}} + 
            \dist_F\bket{\Phi(\gamma_{\beta_j}(t_{0})),\Phi(\gamma_{\beta_j}(t_{1}))}\\
            \le\;&
            (1+k^{-1})\sum_{b=0}^{1} \dist\bket{\zeta(t_b),\gamma_{\beta_j}(t_b)} + 
            (1+k^{-1})\cdot \dist_{\beta_j}\bket{\gamma_{\beta_j}(t_{0}),\gamma_{\beta_j}(t_{1})}\\
            \le\;& (1+k^{-1})\bket{2j^{-1} + C|t_0-t_1|}.
        \end{split}
    \end{align}
    Taking $k\to \infty$ in \eqref{estidis1}, we obtain $\dist_{\scrl}\bket{\zeta(t_0),\zeta(t_1)}\le C|t_0-t_1|$. Dividing $[0,1]$ into small intervals, we deduce that $\dist_F\bket{K_0,K_1}\le C$, which is a contradiction to the assumption \eqref{d61}.
\end{proof}

\section{Proof of Propositions \ref{part} and \ref{partcurva}}\label{app}
\begin{proof}[Proof of \cref{part}]
(cf.\ \cite[II.~Thm.~8.2]{lm89}). For $u\in\Gamma_\cm(S)$, define
\begin{equation}\label{divergence}
    X_1 = \sum_{j=1}^{q}\inn{u,c(e_j)D_qu}\,e_j
    \quad\text{and}\quad
    X_2 = \sum_{j=1}^{q}\inn{u,\nabla_{e_j}u}\,e_j,
\end{equation}
which are independent of the choice of the orthonormal frame $\{e_j\}_{j=1}^q$ spanning $F$. Recall that their divergences on $(M,g^{TM})$ are $\diver X_b = \sum_{k=1}^{n}\binn{\nabla^{TM}_{e_k}X_b,e_k}$. So
\begin{align}
\begin{split}
    \diver X_1 
    &= \sum_{j=1}^{q}e_j\binn{u,c(e_j)D_qu} - \sum_{k=1}^{n}\,\sum_{j=1}^{q}\binn{u,c(e_j)D_qu}\binn{e_j,\nabla^{TM}_{e_k}e_k}\\
    &= \sum_{j=1}^{q}e_j\binn{u,c(e_j)D_qu} - \sum_{k=1}^{n} \binn{u,c\bket{p\nabla^{TM}_{e_k}e_k}D_qu},
\end{split}\\
\begin{split}
    \diver X_2 
    &= \sum_{j=1}^{q}e_j\binn{u,\nabla_{e_j}u} - \sum_{k=1}^{n}\,\sum_{j=1}^{q}\binn{u,\nabla_{e_j}u}\binn{e_j,\nabla^{TM}_{e_k}e_k}\\
    &= \sum_{j=1}^{q}e_j\binn{u,\nabla_{e_j}u} - \sum_{k=1}^{n} \binn{u,\nabla_{p\nabla^{TM}_{e_k}e_k}u}.
\end{split}
\end{align}
By \eqref{a3} and the fact that $c(e_j)$ is skew Hermitian, we have
\begin{align}\label{f0}
    \begin{split}
        |D_q u|^2
        &= \sum_{j=1}^{q}\binn{c(e_j)\nabla_{e_j}u,D_qu} 
        = -\sum_{j=1}^{q}\binn{\nabla_{e_j}u,c(e_j)D_qu}\\
        &= -\sum_{j=1}^{q}\bbac{e_j\binn{u,c(e_j)D_qu} - \binn{u,\nabla_{e_j}\bket{c(e_j)D_qu}}}\\
        &= -\sum_{j=1}^{q}e_j\binn{u,c(e_j)D_qu} + \sum_{j=1}^{q}\bbinn{u,\bbket{c(e_j)\nabla_{e_j} + c\bket{\nabla^{TM}_{e_j}e_j}}D_qu}\\
        &=
        \binn{u,D_q^2 u} - \diver X_1 +  \bbinn{u,c\bbket{\sum_{j=1}^{q} p^\perp\nabla^{TM}_{e_j}e_j- \sum_{k=q+1}^{n}p\nabla^{TM}_{e_k}e_k} D_qu}.
    \end{split}
\end{align}
By \eqref{a3}, we have
\begin{align}\label{licq1}
    \begin{split}
        D_q^2 
        &= \sum_{i,j=1}^{q}c(e_i)\nabla_{e_i}\,c(e_j)\nabla_{e_j}\\
        & = \sum_{i,j=1}^{q}c(e_i)\,c(e_j)\nabla_{e_i}\nabla_{e_j} + \sum_{i,j=1}^{q}c(e_i)\,c\bket{\nabla^{TM}_{e_i}e_j}\nabla_{e_j}.
    \end{split}
\end{align}
The first term in the last line of \eqref{licq1} is
\begin{align}\label{f1}
    \begin{split}
        \sum_{i,j=1}^{q}c(e_i)c(e_j)\nabla_{e_i}\nabla_{e_j} 
        &= - \sum_{j=1}^{q}\nabla_{e_j}\nabla_{e_j} + \frac{1}{2}\sum_{i,j=1}^{q}c(e_i)c(e_j)\bbac{\nabla_{e_i}\nabla_{e_j}- \nabla_{e_j}\nabla_{e_i}}\\
        &= - \sum_{j=1}^{q}\nabla_{e_j}\nabla_{e_j} + \frac{1}{2}\sum_{i,j=1}^{q}c(e_i)c(e_j)\bbac{R(e_i,e_j) + \nabla_{[e_i,e_j]}}\\
        &= - \sum_{j=1}^{q}\nabla_{e_j}\nabla_{e_j} + \calr_q + \frac{1}{2}\sum_{i,j=1}^{q}c(e_i)c(e_j)\nabla_{[e_i,e_j]}.
    \end{split}
\end{align}
The second term in the last line of \eqref{licq1} is
\begin{align}\label{f2}
    \begin{split}
        &\quad\;\sum_{i,j=1}^{q}c(e_i)\,c\bket{\nabla^{TM}_{e_i}e_j}\nabla_{e_j} = \sum_{i,j=1}^{q}\,\sum_{k=1}^{n}c(e_i)\,c\ket{e_k}\nabla_{\inn{\nabla^{TM}_{e_i}e_j,e_k}e_j}\\ 
        &= -\sum_{i,j=1}^{q}\,\sum_{k=1}^{n}c(e_i)\,c\ket{e_k}\nabla_{\inn{e_j,\nabla^{TM}_{e_i}e_k}e_j}= - \sum_{i=1}^{q}\,\sum_{k=1}^{n}c(e_i)\,c(e_k)\nabla_{p\nabla^{TM}_{e_i}e_k}\\
        &= \sum_{i=1}^{q}\nabla_{p\nabla^{TM}_{e_i}e_i} - \sum_{\substack{i\ne k\\1\le i,k \le q}}c(e_i)\,c(e_k)\nabla_{p\nabla^{TM}_{e_i}e_k} - \sum_{i=1}^{q}\,\sum_{k=q+1}^{n}c(e_i)\,c(e_k)\nabla_{p\nabla^{TM}_{e_i}e_k}\\
        &= \sum_{i=1}^{q}\nabla_{p\nabla^{TM}_{e_i}e_i} - \frac{1}{2}\sum_{i,k=1}^{q}c(e_i)\,c(e_k)\nabla_{p[e_i,e_k]} - \sum_{i=1}^{q}\,\sum_{k=q+1}^{n}c(e_i)\,c(e_k)\nabla_{p\nabla^{TM}_{e_i}e_k}.\\
    \end{split}
\end{align}
By \eqref{f1}--\eqref{f2}, the formula \eqref{licq1} becomes
\begin{align}
    \begin{split}
        D_q^2 &= -\Delta_q + \calr_q + \frac{1}{2}\sum_{j,k=1}^{q}c(e_j)\,c(e_k)\nabla_{p^\perp[e_j,e_k]} \\
        &\qquad\qquad\quad\;\;\,- \sum_{j=1}^{q}\,\sum_{k=q+1}^{n}c(e_j)\,c(e_k)\nabla_{p\nabla^{TM}_{e_j}e_k}.\label{licq}
    \end{split}
\end{align}
Meanwhile, we have
\begin{align}\label{f3}
        \begin{split}
            \inn{u,-\Delta_q u} 
            &= -\sum_{j=1}^{q} \bbinn{u,\bbket{\nabla_{e_j}\nabla_{e_j} - \nabla_{p\nabla^{TM}_{e_j}e_j}}u}\\
            & = -\sum_{j=1}^{q} \bbac{e_j\inn{u,\nabla_{e_j}u} - |\nabla_{e_j}u|^2} + \sum_{j=1}^{q} \binn{u,\nabla_{p\nabla^{TM}_{e_j}e_j}u}\\
            &= 
            \sum_{j=1}^{q}|\nabla_{e_j}u|^2 - \diver X_2 - \sum_{k=q+1}^{n}\binn{u, \nabla_{p\nabla^{TM}_{e_k}e_k}u}.
        \end{split}
\end{align}

Now we integrate \eqref{f0}, \eqref{licq} and \eqref{f3} over $M$. In \cref{3.1}, the inward unit normal $\nu\in\Gamma(F|_{\partial M})$. For $b=1,2$, applying the divergence theorem
\begin{equation}
    -\int_M \diver X_b = \int_{\partial M} \tinn{X_b,\nu}
\end{equation}
to \eqref{divergence}, we obtain
\begin{align}\label{licpart}
    \begin{split}
        \int_M|D_q u|^2 &= \int_M \bbac{\sum_{j=1}^{q} |\nabla_{e_j} u|^2 + \inn{u,\calr_q u}} + I_1 + I_2 +I_3\\
        &\qquad\qquad\qquad+ \int_{\partial M}\binn{u,\bket{c(\nu)D_q + \nabla_\nu}u}.
    \end{split}
\end{align}
Since $\inn{\nabla^{TM}_{e_j}\nu,\nu}= 0$, we have
\begin{equation}
    \nabla^{TM}_{e_j}\nu = \sum_{k=2}^{n}\binn{\nabla^{TM}_{e_j}\nu,e_k}\,e_k
    \quad\text{and}\quad c(\nu) \, c\bket{\nabla^{TM}_{e_j}\nu} = - c\bket{\nabla^{TM}_{e_j}\nu} \, c(\nu).
\end{equation}
Under \cref{3.1}, the local orthonormal frame $\{\nu,e_2,\dots,e_q\}$ spans $F|_{\partial M}$, where $e_2,\dots,e_q$ are local sections of $T\partial M$. We have
\begin{align}\label{bdmean}
    \begin{split}
        D^\partial_qu &= \sum_{j=2}^{q}c(e_j)\,c(\nu)\,\bbac{\nabla_{e_j}+ \frac{1}{2}\,c\bket{\nabla^{TM}_{e_j}\nu}\,c(\nu)}u\\
        &= -\sum_{j=2}^{q}c(\nu)\,c(e_j)\nabla_{e_j}u + \frac{1}{2}\sum_{j=2}^{q}c(e_j)\,c\bket{\nabla^{TM}_{e_j}\nu}u\\
        &= -\bket{c(\nu)D_q + \nabla_\nu}u + \frac{1}{2}\sum_{j=2}^{q}\,\sum_{k=2}^{n}\binn{\nabla^{TM}_{e_j}\nu,e_k}\,c(e_j)\,c(e_k)u.
    \end{split}
\end{align}
Since $[e_j,e_k]\in T\partial M$ when $2\le j,k\le q$, we have $\inn{\nu,[e_j,e_k]}=0$. It follows that
\begin{align}\label{bdmean2}
    \begin{split}
        &\sum_{\substack{2\le j,k\le q}}\binn{\nabla^{TM}_{e_j}\nu,e_k}\,c(e_j)\,c(e_k)\\
        &\qquad = -\sum_{j=2}^q\binn{\nabla^{TM}_{e_j}\nu,e_j} - \sum_{\substack{j\ne k\\ 2\le j,k\le q}}\,\binn{\nu,\nabla^{TM}_{e_j}e_k}\,c(e_j)\,c(e_k)\\
        &\qquad = (q-1)H^{TM}_q - \frac{1}{2} \sum_{\substack{j\ne k\\ 2\le j,k\le q}}\binn{\nu,[e_j,e_k]}\,c(e_j)\,c(e_k)\\
        &\qquad = (q-1)H^{TM}_q.
    \end{split}
\end{align}
By \eqref{bdmean}--\eqref{bdmean2}, we have
\begin{align}\label{bdint}
    \begin{split}
        D^\partial_q u &= -\bket{c(\nu)D_q + \nabla_\nu}u + \frac{q-1}{2}H^{TM}_q u\\
        &\qquad\qquad+ \frac{1}{2}\sum_{j=2}^{q}\,\sum_{k=q+1}^{n}\binn{\nabla^{TM}_{e_j}\nu,e_k}\,c(e_j)\,c(e_k)\,u.
    \end{split}
\end{align}
Combining \eqref{licpart} and \eqref{bdint}, we complete the proof.
\end{proof}

\begin{proof}[Proof of \cref{partcurva}]
    This is similar to \cite{lic63}; cf.\ \cite[II.~Thm.~8.8]{lm89}. \phantomsection\label{3.4prop}
    The first Bianchi identity $R^{TM}_{ijkl} + R^{TM}_{jkil} + R^{TM}_{kijl} =0$ implies that
    \begin{equation}\label{distinct}
        \sum_{\substack{1\le i,j,k \le q\\\text{distinct}}}R^{TM}_{ijkl}\,c(e_i)\,c(e_j)\,c(e_k) = 0.
    \end{equation}
    By \eqref{distinct}, we have
    \begin{align}\label{a15}
        \begin{split}
            &\sum_{i,j,k,l=1}^{q} R^{TM}_{ijkl}\,c(e_i)\,c(e_j)\,c(e_k)\,c(e_l)\\
            &\qquad =
            \sum_{l=1}^{q}\,\sum_{i,j=1}^{q}\bbket{R^{TM}_{ijil}\,c(e_j) - R^{TM}_{ijjl}\,c(e_i)}\,c(e_l) \\
            &\qquad = 
            2\sum_{i,j,l=1}^{q}R^{TM}_{ijil}\,c(e_j)\,c(e_l)\\
            &\qquad =
            -2\sum_{i,j=1}^{q}R^{TM}_{ijij} + \sum_{i=1}^{q}\,\sum_{\substack{j\ne l \\1\le j,l \le q}}\bket{R^{TM}_{ijil} - R^{TM}_{ilij}}\,c(e_j)\,c(e_l)\\
            &\qquad =  
            -2\sum_{i,j=1}^{q}R^{TM}_{ijij}.
        \end{split}
    \end{align}
    Similarly, by \eqref{distinct}, we have
    \begin{align}\label{a16}
        \begin{split}
            &\sum_{i,j,l=1}^{q}\,\sum_{k=q+1}^{n} R^{TM}_{ijkl}\,c(e_i)\,c(e_j)\,c(e_k)\,c(e_l)\\ 
            &\qquad =
            \sum_{i,j,k=1}^{q}\,\sum_{l=q+1}^{n} R^{TM}_{ijkl}\,c(e_i)\,c(e_j)\,c(e_k)\,c(e_l)\\
            &\qquad =
            \sum_{l=q+1}^{n}\,\sum_{i,j=1}^{q}\bbket{R^{TM}_{ijil}\,c(e_j) - R^{TM}_{ijjl}\,c(e_i)}\,c(e_l) \\
            &\qquad =
            2\sum_{l=q+1}^{n}\,\sum_{i,j=1}^{q}R^{TM}_{ijil}\,c(e_j)\,c(e_l).
        \end{split}
    \end{align}
    The result follows directly from \eqref{a15}--\eqref{a16}.
\end{proof}

\noindent {\bf Acknowledgments.} The author is grateful to Professor Weiping Zhang for his guidance and invaluable suggestions. The author also thanks Chenkai Song and Professor Guangxiang Su for many enlightening discussions. This work was partially supported by National Key R\&D Program of China (2024YFA1013202), NSFC Grant No.\ 11931007 and 12621005, and Nankai Zhide Foundation.
\printbibliography

\end{document}